\documentclass[11pt,reqno]{article}

\usepackage{array,amsmath,amssymb,amsthm,color,vmargin,overpic,amstext,mathtools,yhmath}
\usepackage{amsfonts,amscd,bm,float}
\usepackage[normalem]{ulem}
\usepackage{tikz}
\usepackage{cite}
\usepackage{ytableau}
\usepackage{caption}
\usepackage{subfig} 
\usetikzlibrary{positioning}
\usepackage[pdftex,colorlinks=true,linkcolor=blue,citecolor=blue,urlcolor=blue]{hyperref}

\newtheorem{theorem}{Theorem}

\newtheorem{remark}[theorem]{Remark}

\newtheorem{lemma}[theorem]{Lemma}
\newtheorem{proposition}[theorem]{Proposition}
\newtheorem{corollary}[theorem]{Corollary}
\newtheorem{conjecture}{Conjecture}

\numberwithin{equation}{section}
\numberwithin{theorem}{section}
\usepackage{authblk}

\newcommand{\E}{\mathbb{E}}

\def\ba{\begin{array}}
\def\ea{\end{array}}

\def\0{{\bm 0}}
\def\i{{\mathrm{i}}}

\def\x{{\bm x}}

\def\z{{\bm{z}}}
\def\m{{\bm{m}}}
\def\n{{\bm{n}}}

\def\q{{\bm q}}

\def\h{{\bm h}}
\def\w{{\bm{w}}}

\renewcommand{\Re}{\operatorname{Re}}

\newcommand{\be}{\begin{equation}}
\newcommand{\ee}{\end{equation}}
\newcommand{\bea}{\begin{eqnarray}}
\newcommand{\eea}{\end{eqnarray}}
\newcommand{\bes}{\begin{equation*}}
\newcommand{\ees}{\end{equation*}}
\newcommand{\beas}{\begin{eqnarray*}}
\newcommand{\eeas}{\end{eqnarray*}}

\allowdisplaybreaks
\usepackage{pgfplots}

\pgfplotsset{compat=1.17}

\title{On moments of the derivative of CUE characteristic polynomials and the Riemann zeta function}

\author{Nick Simm\footnote{n.j.simm@sussex.ac.uk} }
\author{Fei Wei\footnote{weif0831@gmail.com}}
\affil{Department of Mathematics, University of Sussex, 
Brighton, BN1 9RH, UK}

\date{}

\begin{document}

\maketitle

\begin{abstract}
We study the derivative of the characteristic polynomial of $N \times N$ Haar distributed unitary matrices. We obtain the first explicit formulae for complex-valued moments when the spectral variable is inside the unit disc, in the limit $N \to \infty$. These formulae are expressed in terms of the confluent hypergeometric function of the first kind. As an application, we provide an alternative method to re-obtain Mezzadri's result [J. Phys. A, 36(12):2945-2962, 2003] on the asymptotic density of zeros of the derivative as $N \to \infty$. We explore the connection between these moments and those of the derivative of the Riemann zeta function away from the critical line. Under the Lindel\"of hypothesis, we prove that all positive integer moments agree with our random matrix results up to an arithmetic factor. Inspired by this finding, we propose a conjecture on the asymptotics of non-integer moments of the derivative of the Riemann zeta function off the critical line. Within random matrix theory, we also investigate the microscopic regime where the spectral variable $z$ satisfies $|z|^{2}=1-\frac{c}{N}$ for a fixed constant $c$. We obtain an asymptotic formula for the moments in this regime as a determinant involving the finite temperature Bessel kernel, which reduces to the Bessel kernel when $c=0$. For finite matrix size, we provide an exact formula for the moments of the derivative inside the unit disc, expressed as polynomials of the inverse of the distance from the circle, with coefficients given by combinatorial sums.
\end{abstract}

\maketitle

\tableofcontents

\section{Introduction and main results}
Let $U\in \mathbb{U}(N)$ be sampled from the group of $N\times N$ unitary matrices with respect to the normalized Haar measure $d\mu$. A unitary matrix defined in this way is said to belong to the CUE (Circular Unitary Ensemble) of random matrices. We define the characteristic polynomial of $U$ as $\Lambda_N(z) = \det(I - z U^{\dagger})$, where $U^{\dagger}$ denotes the conjugate transpose of $U$. The moments we study in this paper are defined as
\begin{equation}
\E[|\Lambda_N'(z)|^{2s}] = \int_{\mathbb{U}(N)} |\Lambda_N'(z)|^{2s} d\mu, \label{moms-intro}
\end{equation}
where $z$ is a complex number. By the translation-invariance of the Haar measure, \eqref{moms-intro} only depends on $z$ through its radial component $|z|$. 

In the case $|z|=1$, moments of type \eqref{moms-intro} received a lot of attention over the last twenty years, e.g., \cite{CRS06,dehaye2008joint,hughes2001theis,w12,forrester2006boundary,hughes2000random} and more recently \cite{barhoumi2020new,KW23a,KW23b,A24,AGKW24,assiotis2022joint,bailey2019mixed,assiotis2021distinguished,basor2019representation}. Part of the reason for this activity is the conjectured relation with the corresponding moments of the derivative of the Riemann zeta function. In addition, these moments have a rich structure from the viewpoint of some other fields. For example, the moments are related to solutions of $\sigma$-Painlev\'{e} equations \cite{basor2019representation,forrester2006boundary,KW23b,A24,assiotis2022joint,assiotis2021distinguished,AGKW24,bailey2019mixed}, and expectations of random variables arising from the Hua-Pickrell determinantal point process \cite{assiotis2022joint,AGKW24}. See Section \ref{se:mom-intro-riemann} for further background and discussion.

In this paper, we shall investigate \eqref{moms-intro} for $|z|\leq1$, both for finite $N$ and asymptotically as $N \to \infty$. We distinguish three regimes of interest in the asymptotics. Firstly, consider the case that $N \to \infty$ with fixed $0\leq |z|<1$. Since $z$ lies a macroscopic distance from the unit circle where the eigenvalues of the unitary matrix $U$ are located, we refer to this as the \textit{global regime}. We define the \textit{mesoscopic regime} to be such that $|z|^{2} = 1-N^{-\alpha}$ for some parameter $\alpha \in (0,1)$ and finally the \textit{microscopic regime} is such that $|z|^{2}=1-c/N$ where $c \in \mathbb{R}$ is a fixed constant. In the last regime, the relevant distance is on the order of the mean separation of eigenvalues, which is $1/N$. The vast majority of known results for \eqref{moms-intro} assume that $|z|=1$, which corresponds to the particular case $c=0$ of the microscopic regime. In Appendix \ref{Painlevesix} we include a discussion on the known results for moments of $\Lambda_{N}(z)$ (i.e. without taking a derivative) in these regimes and connections to $\sigma$-Painlev\'e equations.


\subsection{Global regime}

Our first main result gives an explicit form of the limit as $N \to \infty$ of \eqref{moms-intro} for $|z|<1$ with non-integer exponent $s$.
\begin{theorem}
\label{th:non-integer-intro}
For any fixed $z$ with $|z|<1$ and any $s \in \mathbb{C}$ with $\mathrm{Re}(s) > -1$, we have
\begin{equation}
\lim_{N \to \infty}\mathbb{E}\left(|\Lambda_{N}'(z)|^{2s}\right) = \frac{e^{-s^{2}|z|^{2}}\Gamma(s+1)}{(1-|z|^{2})^{s^{2}+2s}}\,{}_1F_1(s+1,1;s^{2}|z|^{2})\label{derivmomsthm1}
\end{equation}
where ${}_1F_1(a,b;z)$ is the confluent hypergeometric function of the first kind given by
\begin{equation}
{}_1F_1(a,b;z) = \sum_{k=0}^{\infty}\frac{a^{(k)}}{b^{(k)}}\,\frac{z^{k}}{k!} \label{1f1},
\end{equation}
and $a^{(k)} =\Gamma(a+k)/ \Gamma(a)$. When $s$ is a positive integer, the hypergeometric function simplifies and we obtain
\begin{equation}
    \lim_{N \to \infty}\mathbb{E}\left(|\Lambda'_{N}(z)|^{2s}\right) = \frac{1}{(1-|z|^{2})^{s^{2}+2s}}\,s!\,L_{s}(-|z|^{2}s^{2}),
\label{derivmoms2}
\end{equation}
where $L_{s}(x)$ is the Laguerre polynomial
\begin{equation}\label{definitionoflaguarrepolynomial}
    L_{s}(x) = \sum_{k=0}^{s}\binom{s}{k}\frac{(-1)^{k}}{k!}\,x^{k}.
\end{equation}
\end{theorem}
Our approach to proving Theorem \ref{th:non-integer-intro} is based on a connection between the limit of $\log\Lambda_N(z)$ and a complex multivariate Gaussian distribution. This distribution encodes the statistical correlations between $\log \Lambda_{N}(z)$ and its derivative as $N \to \infty$. By computing the appropriate joint moments of the limiting Gaussian field, we obtain the formulae given in Theorem \ref{th:non-integer-intro}. Proving that we can pass the convergence in distribution inside the expectation on the left-hand side of \eqref{derivmomsthm1} requires extra work in the form of uniform integrability estimates. Especially for negative moments $-1 < \mathrm{Re}(s) < 0$, one of the challenges is to control the probability of small values of the derivative. In this regime, we estimate such probabilities in terms of the Fourier transform of the distribution. We shall explain the idea further in Section \ref{se:glob}. The same approach leads to the following formula for joint moments at different points, which reduces to Theorem \ref{th:non-integer-intro} by taking $h=s$ and $z_{1}=z_{2}$.
\begin{theorem}
\label{th:jointmoms-intro}
Let $|z_1|<1$ and $|z_2|<1$ be fixed. Then for any $s,h \in \mathbb{C}$ with $\mathrm{Re}(h)>-1$, we have
\begin{equation}
\lim_{N \to \infty}\mathbb{E}\left(\bigg{|}\frac{\Lambda_{N}'(z_2)}{\Lambda_{N}(z_2)}\bigg{|}^{2h}|\Lambda_{N}(z_1)|^{2s}\right) = \frac{e^{-s^{2}\rho_{z_1,z_2}}\Gamma(h+1)\,{}_1F_1\left(h+1,1;s^{2}\rho_{z_1,z_2}\right)}{(1-|z_{2}|^{2})^{2h}(1-|z_{1}|^{2})^{s^{2}}} \label{jmoms-intro}
\end{equation}
where
\begin{equation}
\rho_{z_1,z_2} = \frac{|z_{1}|^{2}(1-|z_{2}|^{2})^{2}}{|1-z_{1}\overline{z_2}|^{2}}.
\end{equation}
\end{theorem}
\begin{proof}
See Sections \ref{se:glob} and \ref{se:globpfs}.
\end{proof}
One of the motivations for understanding the moments for $|z| < 1$, especially for non-integer $s$, is the possibility to extract information on the underlying zero distribution of $\Lambda_{N}'(z)$. To our knowledge, this procedure was not carried out explicitly before, partly because it is necessary to have formulas for the moments valid inside the disc and not exactly on the unit circle, as in previous works. Using Theorem \ref{th:non-integer-intro}, we can address this problem in the global regime $|z|<1$. The connection between moments and the distribution of zeros can be seen using Jensen's formula. Let $n_{N}(t)$ denote the number of zeros of $\Lambda_{N}'(z)$ inside the disc $D(0,t)$ centered at $0$ with radius $t$. Then Jensen's formula implies that
\bea\label{jensenformula-intro}
\frac{1}{2\pi}\int_{0}^{2\pi}\log|\Lambda'_{N}(re^{\i \theta})|d\theta-\log|\Lambda'_{N}(0)|=\int_{0}^{r}\frac{n_{N}(t)}{t}dt.
\eea
Taking expectations on both sides of \eqref{jensenformula-intro}, by the translation-invariance of the Haar measure $\mathbb{E}(\log|\Lambda'_{N}(re^{\i \theta})|)$ is independent of $\theta$ and we obtain
\begin{equation}
\label{jensenformula3intro}
\begin{split}
\int_{0}^{r}\frac{\mathbb{E}(n_{N}(t))}{t}dt &= \mathbb{E}(\log|\Lambda_{N}'(r)|)-\mathbb{E}(\log|\Lambda_{N}'(0)|)\\
&=\frac{d}{ds} \mathbb{E}\left[|\Lambda_{N}'(r)|^{s}\right]\Bigg|_{s=0}-\frac{d}{ds} \mathbb{E}\left[|\Lambda_{N}'(0)|^{s}\right]\Bigg|_{s=0}.
\end{split}
\end{equation}
In the limit $N \to \infty$, the above can be evaluated explicitly using Theorem \ref{th:non-integer-intro}. In this way, we obtain the limiting expected number of zeros, which recovers a result of Mezzadri \cite[Sec. 3]{mezzadri2003random}.

\begin{corollary}
\label{th:jensen}
Let $n_{N}(r)$ be the number of zeros of $\Lambda'_{N}(z)$ inside the disc of radius $r$ centered at the origin. Then uniformly with respect to $r$ on any closed subset of $[0,1)$, we have
\begin{equation}
\lim_{N \to \infty}\int_{0}^{r}\frac{\mathbb{E}(n_{N}(t))}{t}dt = -\log(1-r^{2})
\label{0901-eq2}
\end{equation}
and
\begin{equation}
\lim_{N \to \infty}\mathbb{E}(n_{N}(r)) = \frac{2r^{2}}{1-r^{2}}. 
\label{nr-intro}
\end{equation}
\end{corollary}

\begin{proof}
See Section \ref{se:globpfs}.
\end{proof}
The radial distribution of zeros of $\Lambda_{N}'(z)$ is believed to correspond to the horizontal distribution of zeros of the derivative of the Riemann zeta function $\zeta'(w)$ (see \cite{duenez2010roots} and \cite{mezzadri2003random} for evidence in this direction). The latter distribution is of interest in number theory, for example, it is known that the absence of zeros of $\zeta'(w)$ in the strip $0 <\mathrm{Re}(w) < \frac{1}{2}$ is equivalent to the Riemann hypothesis. A question posed by Conrey\footnote{John Brian Conrey, personal communication.} asks whether one can evaluate \eqref{moms-intro} for a real number $s$ in a small neighborhood of $0$ as $z$ approaches $1$, and use this, along with Jensen's formula (\ref{jensenformula3intro}) to obtain the radial distribution of zeros of $\Lambda_N'(z)$. Corollary \ref{th:jensen} gives a positive answer to this question in the limit $N \to \infty$.

As we discuss further in Section \ref{se:mom-intro-riemann}, Theorem \ref{th:non-integer-intro} gives insight into the behavior of the moments of the derivative of the Riemann zeta function off the critical line. The latter is closely related to the Lindel\"of hypothesis, which is a long-standing conjecture in number theory and is implied by the Riemann hypothesis. More precisely, using the explicit formula on the right-hand side of \eqref{derivmomsthm1} we predict an asymptotic formula for the moments of $\zeta'(\sigma+\i t)$ as $\sigma\rightarrow 1/2$ from above, see Conjecture \ref{conjectureonthemomentsoffthecriticalline}. Assuming the Lindel\"of hypothesis, we shall prove this conjecture in the special case of integer moments. We conclude that our conjectured relation between random matrix theory and the moments of the derivative of the Riemann zeta function is new, since to the best of our knowledge, all previous conjectures are on the critical line $\sigma=1/2$. 


\subsection{Mesoscopic and microscopic regimes}
In the case of positive integer exponents $s$ in \eqref{moms-intro} we investigate the structure of the moments for finite matrix size $N$. This allows us to obtain the asymptotics when $z$ comes close to the unit circle. In particular, we show the following in the mesoscopic regime.
\begin{theorem}
\label{th:meso}
Let $s$ be a positive integer and set $|z|^{2}=1-N^{-\alpha}$ with $0 < \alpha < 1$. Then as $N \to \infty$
\begin{equation}
\mathbb{E}\left(|\Lambda'_{N}(z)|^{2s}\right) \sim N^{\alpha(s^{2}+2s)}\,s!\,L_{s}(-s^{2}),\nonumber
\label{derivmoms2meso}
\end{equation}
where $L_{s}(x)$ is the Laguerre polynomial, as in \eqref{derivmoms2}.
\end{theorem}
\begin{proof}
See Section \ref{se:globpfs}.
\end{proof}
We obtain this result, together with an asymptotic formula valid in the microsopic regime given below, starting from an exact finite $N$ formula expressed in terms of sums over partitions. To state these results we introduce the relevant notation. A partition is a weakly decreasing sequence of positive integers $\lambda = (\lambda_{1},\lambda_{2},\ldots)$ with only finitely many non-zero terms. The number of non-zero terms is the length $l(\lambda)$ of the partition and the weight is $|\lambda|=\sum_{i=1}^{l(\lambda)}\lambda_{i}$. Partitions can be represented by their Young diagram, defined as a finite collection of boxes arranged in left-justified rows, with $\lambda_{j}$ boxes in row $j$ for each $j=1,\ldots,l(\lambda)$.

Let $(i,j)$ be a box in a Young diagram $\lambda$. The hook of $(i,j)$ is the union of all boxes in $i$-th row to the right of $(i,j)$, all boxes in $j$-th column below $(i,j)$ and box $(i,j)$. The hook length $h(i,j)$ is the number of boxes in the hook of $(i,j)$. A standard Young tableaux is a filling of a Young diagram of type $\lambda$ by numbers $1,\ldots,|\lambda|$ such that each number appears precisely once, and entries increase in rows (to the right) and in columns (downwards). Denote by $f_{\lambda}$ the number of standard Young tableaux of type $\lambda$. It is known that this quantity is related to the hook length via
\be
f_{\lambda}=\frac{|\lambda|!}{\prod_{(i,j)\in \lambda}h(i,j)}.
\label{hooklengthformula}
\ee
We denote by $Y_{m}$ the set of partitions $\lambda$ satisfying $|\lambda|=m$. Note that this automatically guarantees that the length satisfies $l(\lambda)\leq m$. Throughout the paper, set $\lambda_{l(\lambda)+1}=\cdots=\lambda_{m}=0$ when $l(\lambda)\leq m$. For a partition $\lambda \in Y_{m}$, throughout the paper, denote
\bea\label{somenotation}
\lambda! = \prod_{i=1}^{m}(\lambda_i+m-i)!.  
\eea
\begin{theorem}
\label{th:exact-intro}
We have the following exact formula, valid for any $z\in \mathbb{C}$ and any positive integers $N,s$,
\begin{equation}
\E[|\Lambda_{N}'(z)|^{2s}] = \sum_{\lambda, \mu \in Y_{s}}\frac{f_\lambda f_\mu}{\lambda! \mu!}
\det\bigg\{(u^{\lambda_i+s-i} K_{N}^{(\lambda_i+s-i)}(u))^{(\mu_j+s-j)}\bigg\}_{i,j=1}^{s} \label{e2sintro}
\end{equation}
where $u=|z|^{2}$ and
\begin{equation}
K_{N}(u) = \sum_{j=0}^{N+s-1}u^{j}.
\end{equation}
The bracketed superscripts on the right-hand side in \eqref{e2sintro} denote repeated differentiations $f^{(n)}(u) = \frac{d^{n}f}{du^{n}}$.
\end{theorem}
\begin{proof}
See Section \ref{se:exact}.
\end{proof}

We remark that in recent work, Alvarez, Conrey, Rubinstein, and Snaith \cite[Theorem 3.1]{A24} use a different method to obtain an exact formula for $\E[|\Lambda_{N}'(z)|^{2s}]$ in terms of derivatives of determinants whose entries are contour integrals. For our purposes, the primary goal of Theorem \ref{th:exact-intro} is to derive the following asymptotic formula in the microscopic regime.

\begin{theorem}\label{th:micro-intro}
Let $|z|^{2}=1-\frac{c}{N}$ for $c \in \mathbb{R}$ fixed. In particular, the case $|z|=1$ corresponding to $c=0$ is allowed. Then for any positive integer $s$, as $N \to \infty$ we have
\begin{equation}
\label{micro1}
\begin{split}
\E[|\Lambda_{N}'(z)|^{2s}] \sim N^{s^2+2s}\sum_{\lambda,\mu \in Y_{s}}\frac{f_\lambda f_\mu}{\lambda! \mu!}
\det\bigg\{\int_{0}^{1}x^{\lambda_{i}+\mu_{j}+2s-i-j}e^{-cx}dx\bigg\}_{i,j=1}^{s}
\end{split}
\end{equation}
We also obtain the equivalent expression, as $N \to \infty$,
\begin{equation}
\label{micro2}
\mathbb{E}(|\Lambda_{N}'(z)|^{2s}) \sim N^{s^{2}+2s}\frac{\partial^{2s}}{\partial v^{s}\partial w^{s}}\,\mathrm{det}\bigg\{\frac{\partial^{i+j-2}F_{c}(v,w)}{\partial v^{i-1}\partial w^{j-1}}\bigg\}_{i,j=1}^{s}\bigg{|}_{v=w=0},
\end{equation}
where
\begin{equation}
F_{c}(v,w) = \int_{0}^{1}J_{0}(2\sqrt{vx})J_{0}(2\sqrt{wx})e^{-cx}dx, \label{finiteBessel}
\end{equation}
and $J_{0}$ is the Bessel function of the first kind. Furthermore, the leading terms on the right-hand sides of \eqref{micro1} and \eqref{micro2} are strictly positive.
\end{theorem}
\begin{proof}
See Section \ref{se:micropfs}.
\end{proof}
To the best of our knowledge, Theorem \ref{th:micro-intro} is the first result obtained for moments of $\Lambda'_{N}(z)$ that is valid in the full microscopic regime, i.e., for all values of the parameter $c$. In particular, it covers the case $c=0$, the moments on the unit circle. When $c=0$, the entries in the determinant of \eqref{micro1} can be computed explicitly. Then by the property of the Cauchy determinant, the determinant has an explicit combinatorial formula. This expression must coincide with \cite[Theorem 1]{CRS06}, which is given in terms of a Hankel determinant. 

\begin{remark}
\label{rem:bothner}
The expression in \eqref{finiteBessel} can be recognised as a particular case of the finite temperature Bessel kernel. The properties of this kernel and its connections to integrable systems and integrable probability have recently attracted interest, see \cite{BO22, LLMS18, R24}. When $c=0$, \eqref{finiteBessel} is exactly the Bessel kernel, due to the identity
\begin{equation}
\int_{0}^{1}J_{0}(2\sqrt{vx})J_{0}(2\sqrt{wx})dx = 2\frac{\sqrt{w}J_{0}(2\sqrt{z})J_{0}'(2\sqrt{w})-\sqrt{z}J_{0}'(2\sqrt{z})J_{0}(2\sqrt{w})}{z-w}.
\end{equation}
\end{remark}

\subsection{Moments of the derivative of the Riemann zeta function} \label{se:mom-intro-riemann}
We now discuss the connection between Theorem \ref{th:non-integer-intro} and the mean value of the derivative of the Riemann zeta function defined as
\bea
\lim_{T\rightarrow \infty} \frac{1}{T} \int_1^T |\zeta'(\sigma+\i t)|^{2s} dt
\label{0805eq30}
\eea
for $\sigma > 1/2$. We start with some background on connections between number theory and random matrix theory. These trace back to Montgomery's conjecture \cite{montgomery1973pair}, which suggests that suitably normalised pairs of zeros of the Riemann zeta function follow a statistical distribution that coincides with Dyson's result \cite{dyson1962statistical} on the spacings between pairs of suitably scaled eigenvalues of random unitary matrices. Another important connection is the work of Keating-Snaith's, who used the characteristic polynomials of random unitary matrices to model the Riemann zeta function on the critical line. A well-known result \cite{keating2000random}, among their series of works, led them to conjecture that for $\Re(s)>-1/2$ 
\be
\frac{1}{T} \int_0^T |\zeta(\frac{1}{2}+\i t)|^{2s} dt \sim 
a_s g_{s} (\log T)^{s^2},
\label{0805-eq1}
\ee
as $T\rightarrow \infty$, where $a_s$ is the arithmetic factor defined as
\be\label{definitionofthearithmeticfactor}
a_s = \prod_{\text{primes}~p} (1-p^{-1})^{s^2} \sum_{m=0}^\infty \left( \frac{\Gamma(s+m)}{\Gamma(m+1)\Gamma(s)}\right)^2 p^{-m},
\ee
and $g_{s}$ is determined by random matrix theory,
\be
g_{s} = \lim_{N\rightarrow \infty} \frac{1}{N^{s^2}} \E(|\Lambda_N(1)|^{2s}). 
\ee
For more details about $g_{s}$ and further background about moments of $\Lambda_{N}(z)$  for general $z$, see Appendix \ref{Painlevesix}. Some special cases of the conjecture \eqref{0805-eq1} have been supported by results in number theory \cite{heath1993fractional,conrey1984mean,conrey2001high,hardy1916contributions}. The philosophy used to propose \eqref{0805-eq1} was later applied by Hughes in his thesis \cite{hughes2001theis} to propose a conjecture for the moments of the derivative of the Riemann zeta-function on the critical line. Specifically,
\be
\frac{1}{T} \int_0^T |\zeta'(\frac{1}{2}+\i t)|^{2s} dt \sim a_s b_{s}(\log T)^{s^2+2s}, \qquad T \to \infty,
\label{0805-eq2}
\ee
for integer $s\geq 1$, where
\be\label{sep3}
b_{s} = \lim_{N\rightarrow \infty} \frac{1}{N^{s^2+2s}} \E(|\Lambda'_N(1)|^{2s}).
\ee
The conjecture \eqref{0805-eq2} was verified for $s=1$ \cite{ingham1928mean} and $s=2$ \cite{conrey1988fourth}.
There are many investigations related to \eqref{sep3}. For example, 
in \cite{hughes2001theis}, Hughes derived an expression for $b_s$;
in \cite{CRS06}, Conrey, Rubinstein and Snaith gave an alternative explicit formula for $b_{s}$ in terms of a Hankel determinant; in \cite{assiotis2022joint}, for non-integer and real number $s$, Assiotis, Keating and Warren studied \eqref{sep3} and proposed a version of (\ref{0805-eq2}). 

The above conjectures (\ref{0805-eq1}) and (\ref{0805-eq2}) are about the moments of the Riemann zeta function and its derivative on the critical line, respectively. For the moments of the derivative of the  Riemann zeta function off the critical line (\ref{0805eq30}), based on our result Theorem \ref{th:non-integer-intro}, we propose the following conjecture.

\begin{conjecture}\label{conjectureonthemomentsoffthecriticalline}
For any $s\in \mathbb{C}$ with $\Re(s)>0$, we have the following asymptotic formula as $\sigma \rightarrow \frac{1}{2}$,
\be\label{newJune14formula22}
\lim_{T\rightarrow \infty}\frac{1}{T}\int_{1}^{T}|\zeta'(\sigma+\i t)|^{2s} dt\sim \frac{a_{s}h_{s}}{(2\sigma-1)^{s^2+2s}},
\ee
where $a_{s}$ is the arithmetic factor given as (\ref{definitionofthearithmeticfactor}) and
\begin{equation}
\begin{split}
\label{thelimtintherandommatrixside}
h_{s}&=\left((1-r^2)^{s^{2}+2s}\lim_{N \to \infty}\mathbb{E}(|\Lambda'_{N}(r)|^{2s})\right)\bigg{|}_{r=1}\\
&=e^{-s^{2}}\Gamma(s+1)\,_1F_{1}(s+1,1;s^{2}).
\end{split}
\end{equation}
\end{conjecture}

\subsubsection{A motivation to propose Conjecture \ref{conjectureonthemomentsoffthecriticalline}}

Consider the mean value
\be
\text{M}(\sigma,s) := \lim_{T\rightarrow \infty} \frac{1}{T} \int_1^T |\zeta(\sigma+\i t)|^{2s} dt
\label{0805-eq3}
\ee
for $\sigma >1/2$. The quantity \eqref{0805-eq3} has been studied for a long time in number theory, mainly because of its important role in the Lindel\"of hypothesis, which states that for any given $\varepsilon>0, \zeta(\frac{1}{2}+\i t) = O(t^{\epsilon})$ as $t \to \infty$. Hardy and Littlewood \cite{hardy1923lindelofs} proved that the Lindel\"of hypothesis is equivalent to the claim that for any integer $s\geq 1$ and any $\sigma>1/2$ we have
\be
\label{0808}
\text{M}(\sigma,s) = \sum_{n=1}^\infty \frac{(d_s(n))^2}{n^{2\sigma}},
\ee
where $d_s(n)$ is the $s$-fold divisor function, which is the number of ways to express a natural number $n$ as an ordered product of $s$
positive integers. Unconditionally, \eqref{0808} is proved when $s=1$ by the mean-value theorem for Dirichlet series, and when $s=2$ by Hardy and Littlewood \cite{hardy1923approximate}. Moreover, the method used in \cite{hardy1923approximate} can be extended to establish \eqref{0808} for integer $s\geq 3$ and $\sigma>1-{1}/{s}$ (see, e.g., \cite[Chapter 7]{titchmarsh1986theory}).

Observe that $d_s(n)$ is a multiplicative function, that is, $d_{s}(nm)=d_{s}(n)d_{s}(m)$ when $n$ and $m$ are coprime. So by the Euler product formula, it is the $n$-th Dirichlet coefficient of $\zeta(w)^s$. By this property, we can extend the definition of $d_s(n)$ from integer $s$ to a real number $s$ (or to a complex number $s$)
by
\[
d_s(p^m) = \frac{\Gamma(s+m)}{ \Gamma(m+1) \Gamma(s)}.
\]
Under this definition, Titchmarsh posed the question whether \eqref{0808} holds for positive non-integers \cite{Titchmarsh1939mean}. Later, there was some research on this direction. For example, Ingham \cite{ingham1933mean} established \eqref{0808} for any number $s\in(0,2]$ and $\sigma>1/2$ (also see a proof of Davenport \cite[Theorem 1]{davenport1935note}). Haselgrove \cite{Haselgrove1949} extended the real number $s$ to a wider range and obtained the corresponding range for $\sigma$. Also, see Bohr and Jessen \cite{BohrJessen} for relevant research in this direction.

According to our above analysis, we know that under the  the Lindel\"of hypothesis, for any $s\in \mathbb{R}$, and unconditionally for $s=1,2$,
\be
M(\sigma,s) \sim \frac{a_s}{(2\sigma-1)^{s^2}}, \quad 
\sigma \rightarrow \frac{1}{2}.
\label{0805-eq5}
\ee
This can be deduced from the fact that
\[
\sum_{n=1}^\infty \frac{(d_s(n))^2}{n^{w}} = \zeta(w)^{s^2} \prod_{\text{primes}~p} (1-p^{-w})^{s^2} \sum_{m=0}^\infty (d_s(p^m))^2 p^{-mw},
\]
and that the latter infinite product is harmonic on $\{w:\Re(w)>1/2\}$. One can also consider the uniformity of the order of $s^2$ of $(2\sigma-1)^{-1}$ in \eqref{0805-eq5} for a certain range of $\sigma$. Specifically, studying the uniformity of 
\be
\label{0903}
\int_1^T |\zeta(\sigma+\i t)|^{2s} dt \approx T(2\sigma-1)^{-s^2}, \qquad T \to \infty
\ee
for all $\sigma$  with  $(\log T)^{-1} \ll \sigma -1/2  \ll 1$. Here the notation $f(x)\approx g(x)$ means that $f(x) \ll g(x)$ and $g(x) \ll f(x)$, and $f(x) \ll g(x)$ means $|f(x)| \leq c g(x)$ for some constant $c>0$ independent of $x$.  

When $s>0$, see, for example, \cite{heath1981fractional} for the related investigation.
When $s<0$, \eqref{0903} is conjectured in \cite{gonek1989negative} to hold for $(\log T)^{-1} \ll \sigma-1/2\ll 1$. The latter has been proved in \cite{bui2024negative} to be valid for $(\log T)^{-a} \ll \sigma-1/2\ll 1$ for a certain range of $a$ in $(0,1)$ under the Riemann hypothesis. Moreover, the authors in \cite{bui2024negative} also obtained the following asymptotic formula in this region
\beas
\int_1^T |\zeta(\sigma+\i t)|^{2s} dt \sim T \sum_{n=1}^\infty \frac{(d_s(n))^2}{n^{2\sigma}}, \qquad T \to \infty.
\eeas
It is known for $|z|<1$ that we have \cite[Corollary 2]{FK04},
\[
\lim_{N \rightarrow\infty} \E(|\Lambda_N(z)|^{2s}) = \frac{1}{(1-|z|^2)^{s^2}}.
\]
Compared to \eqref{0805-eq5}, we observe that the above result in random matrix theory has the same order $s^2$ when $|z|\rightarrow 1$ as that of \eqref{0805-eq5} when $\sigma \rightarrow 1/2$. This motivates the conjecture that the order $s^{2}+2s$ in the exponent of $(2\sigma-1)^{-1}$ in (\ref{newJune14formula22}), should coincide with the order of $(1-|z|^{2})^{-1}$ in the expression given in Theorem \ref{th:non-integer-intro}. In the next subsection, we provide some evidence for the conjectured form of the coefficient $a_{s}h_{s}$ in (\ref{newJune14formula22}).

\subsubsection{The validity of Conjecture \ref{conjectureonthemomentsoffthecriticalline} for integer \texorpdfstring{$s$}{s}}
Here we give results which verify Conjecture \ref{conjectureonthemomentsoffthecriticalline} for $s=2$ and, assuming the Lindel\"of hypothesis, for any integer $s\geq3$. The case $s=1$ can be verified by the mean-value theorem for Dirichlet series, e.g., see \cite[Theorem 7.1]{titchmarsh1986theory}. For $s=2$, we have the following result.
\begin{theorem}\label{checkthecasek=2}
As $\sigma\rightarrow \frac{1}{2}$,
\bea\label{verifyk=2moment}
\lim_{T\rightarrow \infty} \frac{1}{T} \int_1^T |\zeta'(\sigma+\i t)|^{4} dt
\sim \frac{6}{\pi^2}\frac{34}{(2\sigma-1)^{8}}.
\eea
\end{theorem}
\begin{proof}
See Section \ref{derivativeoftheRiemannzetafunction}.
\end{proof}

It is known that the arithmetic factor $a_{s}$ in (\ref{definitionofthearithmeticfactor}) equals $\frac{6}{\pi^2}$ when $s=2$. The coefficient $34$ in (\ref{verifyk=2moment}) coincides with $2L_{2}(-4)$ in \eqref{derivmoms2}. So Conjecture \ref{conjectureonthemomentsoffthecriticalline} is verified for the case $s=2$.
For higher order moments, we verify Conjecture \ref{conjectureonthemomentsoffthecriticalline} for any integer $s\geq 3$ under the Lindel\"of hypothesis.
\begin{theorem}\label{assumetheLindelofconjecture}
Assuming the Lindel\"of hypothesis, for any integer $s\geq 3$, we have the following asymptotic formula, as $\sigma \rightarrow 1/2$
\be\label{June14formula2}
\lim_{T\rightarrow \infty}\frac{1}{T}\int_{1}^{T}|\zeta'(\sigma+\i t)|^{2s} dt\sim a_{s}\frac{s!L_{s}(-s^2)}{(2\sigma-1)^{s^2+2s}},
\ee
where $a_{s}$ is the arithmetic factor given in \eqref{definitionofthearithmeticfactor} and $L_{s}$ is the Laguerre polynomial given in \eqref{definitionoflaguarrepolynomial}.
\end{theorem}
\begin{proof}
See Section \ref{derivativeoftheRiemannzetafunction}.
\end{proof}

In the proof of the above two results, to understand the asymptotic behavior \eqref{0805eq30} when $\sigma$ approaches $1/2$, we use that
\bea\label{restrictregionforintegersderivative}
\lim_{T\rightarrow \infty} \frac{1}{T} \int_1^T |\zeta'(\sigma+\i t)|^{2s} dt
=\sum_{n=1}^\infty \frac{((\log*\cdots*\log)(n))^2}{n^{2\sigma}}.
\eea
In the above, there are $s$ convolutions in $(\log*\cdots*\log)(n)$ and
for two arithmetic functions $f(n)$ and $g(n)$, their convolution is defined by $f*g(n)=\sum_{d|n}f(d)g(n/d)$.
Equation \eqref{restrictregionforintegersderivative} holds for any $\sigma>1/2$ when $s=1,2$, and holds for any $\sigma>1-1/s$ when $s\geq 3$ is an integer. In addition, \eqref{restrictregionforintegersderivative} holds for all integer $s$ and for any $\sigma>1/2$ assuming the Lindel\"of hypothesis. Compared to the right-hand side of \eqref{0808}, the main difference is that the function $(d_s(n))^2$ in \eqref{0808} is multiplicative, while the function $((\log*\cdots*\log)(n))^2$ is not multiplicative. This explains why different methods are required to analyze the asymptotic behavior when $\sigma \to 1/2$. For $(d_s(n))^2$, we use the Euler product, while for $((\log*\cdots*\log)(n))^2$, we use derivatives of Dirichlet series generated by appropriate multiplicative functions.

For \eqref{0808}, it is natural to extend $(d_s(n))^2$ to real numbers $s$ because of the multiplicity. However, it is not so clear how to extend $((\log*\cdots*\log)(n))^2$ to general real values of $s$. So even if we assume the Lindel\"of hypothesis, it is still not clear how to prove Conjecture \ref{conjectureonthemomentsoffthecriticalline} when $s$ is real. It is also interesting to study the uniformity of 
\[
\int_1^T |\zeta'(\sigma+\i t)|^{2s} dt \sim T\sum_{n=1}^\infty \frac{((\log*\cdots*\log)(n))^2}{n^{2\sigma}}
\]
for positive integer $s$ and $\sigma$ in a certain range. Theorem \ref{th:meso} predicts that the above uniformity holds when $(\log T)^{-\alpha}\ll \sigma-1/2\ll 1$ for any $\alpha$ with $0<\alpha <1$, which corresponds to the mesoscopic regime. This regime has been studied at the level of linear statistics of the Riemann zeros, see e.g., \cite{B10, R14} for results and discussion. On microscopic scales, our random matrix result in Theorem \ref{th:micro-intro} predicts that $\int_1^T |\zeta'(\sigma+\i t)|^{2s} dt \approx T(2\sigma-1)^{-s^2-2s}$ uniformly for $(\log T)^{-1}\ll \sigma-1/2\ll 1$.

\subsection{An alternative explicit formula for finite matrix size}

Here we present an alternative explicit formula for $\mathbb{E}[|\Lambda_{N}'(z)|^{2s}]$ to the one given in Theorem \ref{th:exact-intro}. 
Our motivation here is to seek a formula for finite matrix size such that we can easily (compared with Theorem \ref{th:exact-intro}) recover our limit case (i.e., \eqref{derivmomsthm1} in Theorem \ref{th:non-integer-intro}) as the matrix size goes to infinity for positive integer $s$. First, we introduce two functions in the form of determinants. For any integers $N,s\geq 1$, introduce the block matrix
\begin{equation}
\label{defofA}
A(\z,\w):= \begin{pmatrix} V_{\bm{z}} & Y_{\bm{z}}\\
Y_{\bm{w}} & V_{\bm{w}}
\end{pmatrix},
\end{equation}
where $(V_{\bm{z}})_{ij}=z_{i}^{j-1}$ and $(Y_{\bm{z}})_{ij} = z_{i}^{N+2s-j}$ for $i,j=1,\ldots,s$. Further, define
\bea\label{defofG}
G(\z,\w)=\frac{A(\z,\w)}{\Delta(\z)\Delta(\w) }, \qquad \Delta(\z) = \det(V_{\bm{z}}).
\eea
We remark that $G(\bm{z},\bm{w})$ is a linear combination of products of Schur polynomials  of $\z$ and those of $\w$ by the Laplace expansion. We shall provide more properties of $A(\bm{z},\bm{w})$ and $G(\bm{z},\bm{w})$ in Section \ref{se:exact}.


\begin{theorem}\label{structuretheorem}
Let $z\in \mathbb{C}$ with $|z|\neq 1$ and $s>0$ be an integer. Then 
\bea\label{structureexpression}
\E[|\Lambda_N'(z)|^{2s}]
= \sum_{h=0}^{\infty} C_h(N,|z|) \frac{1}{(1-|z|^2)^{s^2+2s-h}}.
\eea
Here for any integer $h\geq0$,
\be\label{definitionofC}
C_h(N,|z|)=2\sum_{\substack{h_{2}>h_{1} \geq 0\\h_{1}+h_{2}=h}}a_{h_{1},h_{2}}(|z|)b_{h_{1},h_{2}}(N,|z|)+\sum_{\substack{h_{2} \geq 0 \\ 2h_{2}=h}}a_{h_{2},h_{2}}(|z|)b_{h_{2},h_{2}}(N,|z|),
\ee
where
\bea\label{defofa}
a_{h_{1},h_{2}}(|z|)
&=&\frac{1}{h_{1}!h_{2}!}\frac{1}{\Gamma(h_{2}-h_{1}+1)}\nonumber\\
&&\times \, \frac{(\Gamma(s+1))^2}{\Gamma(s-h_{2}+1)}e^{-s^2|z|^2} {}_1F_1(s+1-h_{1},h_{2}-h_{1}+1;s^{2}|z|^2),
\eea
and
\be\label{defofb}
b_{h_{1},h_{2}} (N,|z|) = (-s|z|)^{h_{2}-h_{1}} 
\prod_{j=1}^{h_2} \frac{\partial}{\partial w_j}  \prod_{i=1}^{h_1} \frac{\partial}{\partial z_i} G(\z,\w)\Big|_{\z=\w=-|z|}.
\ee
Moreover,

(i) For any integer $h\geq 0$, $C_h(N,|z|)$ is a polynomial in $|z|^2$
with coefficients that are polynomials in $N$. The degree of $C_h(N,|z|)$ as a polynomial in $|z|^2$ is $Ns+s^2+s-h$ and the leading coefficient is given as \eqref{leading coefficient}.

(ii) When $|z|<1$, we have
\bea\label{C0}
\lim_{N\rightarrow \infty} C_0(N,|z|) = e^{-s^{2}|z|^2}\Gamma(s+1)\,{}_1F_1(s+1,1;s^{2}|z|^2),
\eea
and for $h\neq 0$,
\bea\label{otherC}
\lim_{N\rightarrow \infty} C_h(N,|z|) = 0.
\eea

(iii) 
\bea\label{bh1h2}
b_{h_{1},h_{2}}(N,|z|)=s^{h_{2}-h_{1}}\sum_{m=0}^{\infty}\sum_{l=\min\{1,h_{1}+h_{2}\}}^s
d_{m,l,s}(h_{1},h_{2})|z|^{2Nl- s^2+s -2h_{1}+2m},
\eea
where $d_{m,l,s}(h_{1},h_{2})$ 
are expressed explicitly as certain partition sums in \eqref{defineofdbarml} and \eqref{defofdml1new}.
\end{theorem}

\begin{proof}
See Section \ref{proof the second main result}.
\end{proof}

\begin{remark}\rm{
Note that the function on the right-hand side of (\ref{defofa}) is analytic when $\Re(s)>-1$.
When $s$ is a positive integer, $1/\Gamma(s-h_{2}+1)=0$ in the right hand-side of (\ref{defofa}) if $h_{2}\geq s+1$, so the summation over $h_{1},h_{2}$ in (\ref{defofa}) and $h$ in \eqref{definitionofC} can be restricted to $0\leq h_{1}<h_{2}\leq s$ and $0\leq h\leq 2s$.}
\end{remark}

Compared to Theorem \ref{th:exact-intro}, it is challenging to use \eqref{structureexpression} to obtain the leading order and a compact form of the leading coefficient in the microscopic regime as obtained in Theorem \ref{th:micro-intro}. This mainly because it is not straightforward to explicitly obtain the asymptotic formula when $|z|^2=1-c/N$, due to cancellations in some $N$-terms in $C_{h}(N,|z|)$.

From (i) in the above theorem, we see that for positive integer $s$, (\ref{moms-intro}) is a polynomial in $\frac{1}{1-|z|^2}$ with coefficients that are polynomials in $|z|^2$. The conclusion (ii) provides a different approach to that of Theorem \ref{th:non-integer-intro} for positive integers $s$. Note that $a_{h_1, h_2}(|z|)$ in the expression of $C_h(N,|z|)$ has an explicit expression \eqref{defofa}, which is analytic for $\Re(s)>-1$. So by (\ref{structureexpression}), to obtain an explicit formula for (\ref{moms-intro}) it suffices to express $b_{h_1,h_2}(N,|z|)$ explicitly. This is the purpose of (iii). Especially, $b_{0,0}(N,|z|)=(1-|z|^2)^{s^2}\E[|\Lambda_N(|z|)|^{2s}]$ by  Lemma \ref{alternative expression2}, which can be represented in terms of a solution of the $\sigma$-Painlev\'e VI equation and can be analytically extended to $s\in \mathbb{C}$ with $\Re(s)>-1$, see \eqref{pvimn} and Theorem \ref{lem:PVI}. These properties also hold for $C_0(N,|z|)$ due to
$C_0(N,|z|) = \Gamma(s+1) e^{-s^2 |z|^2} {}_1F_{1}(s+1, 1; s^2 |z|^2) \cdot b_{0,0}(N,|z|)$ by \eqref{definitionofC}.



\section*{Acknowledgements}
We express our gratitude to Thomas Bothner for pointing out to us the content of Remark \ref{rem:bothner}. We would like to thank Marius Tiba for valuable discussions on some combinatorial properties of Young tableaux. N. S. is grateful for support from the Royal Society, grant URF\textbackslash R\textbackslash231028. 
F. W. was supported by NSF grant DMS-1854398.  Both authors are grateful to John Brian Conrey for encouragement and discussions about the problems addressed in this paper.

\section{Global regime and non-integer moments}
\label{se:glob}
In this section, we approximate the finite $N$ characteristic polynomial in Theorems \ref{th:non-integer-intro} and \ref{th:jointmoms-intro} by a limiting object constructed in terms of a certain Gaussian field. Then we compute the associated joint moments of the limit to give the formulas in Theorems \ref{th:non-integer-intro} and \ref{th:jointmoms-intro}. Later in Section \ref{se:mainpfs} we establish uniform integrability bounds in order to show convergence of expectations to these moments.
\subsection{A limiting Gaussian field}
Inside the unit disc $|z|<1$, the characteristic polynomial can be expanded as
\be
G_{N}(z) = \log \Lambda_{N}(z) = -\sum_{k=1}^{\infty}\frac{\mathrm{Tr}(U^{-k})z^{k}}{k}, \qquad |z|<1.
\ee
The latter can also be written as a linear statistic $G_{N}(z) = \sum_{j=1}^{N}\log(1-ze^{-i\theta_{j}})$ where $\{e^{i\theta_{j}}\}_{j=1}^{N}$ are the eigenvalues of $U$. Similarly, for the logarithmic derivative $G_{N}'(z)$, we have
\begin{equation}
    G_{N}'(z) = -\sum_{k=1}^{N}\frac{e^{-i\theta_{j}}}{1-ze^{-i\theta_{j}}} = -\sum_{k=1}^{\infty}\mathrm{Tr}(U^{-k})z^{k-1}, \qquad |z| < 1.
\end{equation}
For a fixed $0 < |z|< 1$, the above linear statistics are smooth functions of the eigenvalues. Joint convergence in distribution of such linear statistics is a consequence of the strong Szeg\"{o} limit theorem for Toeplitz determinants, see \cite{K97} and references therein. The latter also implies that for any fixed positive integer $M$, we have the convergence in distribution
\begin{equation}
 \bigg\{\frac{\mathrm{Tr}(U^{-k})}{\sqrt{k}}\bigg\}_{k=1}^{M} \overset{d}{\longrightarrow} \{\mathcal{N}_{k}\}_{k=1}^{M}, \qquad N \to \infty,
 \end{equation}
 where $\{\mathcal{N}_{k}\}_{k=1}^{M}$ are i.i.d. standard complex normal random variables\footnote{Recall that a standard complex normal random variable is defined such that its real and imaginary parts are independent normal random variables with mean $0$ and variance $\frac{1}{2}$ each.}, see \cite{DS94}. For more background on the relation between the strong Szeg\"{o} limit theorem and convergence in distribution of linear statistics, see \cite{DE01,hughes2001characteristic}. These results immediately imply the following.
\begin{lemma}
\label{le:conv-law}
Define
\begin{equation}\label{gfield}
G(z) = \sum_{k=1}^{\infty}\frac{\mathcal{N}_{k}}{\sqrt{k}}\,z^{k} \quad \textrm{and} \quad G'(z) = \sum_{k=1}^{\infty}\sqrt{k}\,\mathcal{N}_{k}\,z^{k-1}, \qquad |z|<1.
\end{equation}
Fix $|z_1| < 1$ and $|z_2|<1$. We have the joint convergence in distribution
\begin{equation}
(G_{N}(z_1),G_{N}'(z_2)) \overset{d}{\longrightarrow} (G(z_1),G'(z_2)), \qquad N \to \infty.
\end{equation}
\end{lemma}

Therefore, to approximate $\Lambda_{N}(z)$, we introduce
\be
\Lambda(z) = e^{G(z)} = e^{\sum_{k=1}^{\infty}\frac{\mathcal{N}_{k}}{\sqrt{k}}\,z^{k} }.
\ee
We approximate $\Lambda_{N}'(z)$, for $|z|<1$ fixed, in terms of $\Lambda'(z) = G'(z)e^{G(z)}$ and replace the average in Theorem \ref{th:jointmoms-intro} with the expectation of
\begin{equation}
    \bigg{|}\frac{\Lambda'(z_2)}{\Lambda(z_2)}\bigg{|}^{2h}|\Lambda(z_1)|^{2s} = |G'(z_2)|^{2h}e^{sG(z_1)+s\overline{G(z_1)}}. \label{lambdatog}
\end{equation}
Note that $Q = (G(z_1),G'(z_2))$ is a centered multivariate complex normal vector. It is characterised by its $2 \times 2$ covariance matrix $\Gamma$ whose entries are easily obtained from \eqref{gfield}. We obtain the following correlation structure:
\begin{equation} \label{gamz1z2-intro}
\begin{split}
  \Gamma_{11} &= \mathbb{E}(|G(z_1)|^{2}) =  -\log(1-|z_1|^{2}),\\
   \Gamma_{12} &=  \mathbb{E}(G(z_1)\overline{G'(z_2)}) =  \frac{z_1}{1-z_1\overline{z_2}},\\
    \Gamma_{22} &= \mathbb{E}(|G'(z_2)|^{2}) = \frac{1}{(1-|z_2|^{2})^{2}},
\end{split}
\end{equation} 
and $\Gamma_{21} = \overline{\Gamma_{12}}$. From this, the appearance of the confluent hypergeometric function in Theorem \ref{th:jointmoms-intro} is a consequence of the following Gaussian computation.
\begin{lemma}
\label{le:non-integer-comp}
Let $|z_1|<1$ and $|z_2|<1$. For any $s \in \mathbb{C}$ and any $h \in \mathbb{C}$ with $\mathrm{Re}(h)>-1$, we have
\begin{equation}
\mathbb{E}\left(\bigg{|}\frac{\Lambda'(z_2)}{\Lambda(z_2)}\bigg{|}^{2h}|\Lambda(z_1)|^{2s}\right) = \frac{e^{-s^{2}\rho_{z_1,z_2}}\Gamma(h+1)\,{}_1F_1\left(h+1,1;s^{2}\rho_{z_1,z_2}\right)}{(1-|z_{2}|^{2})^{2h}(1-|z_{1}|^{2})^{s^{2}}} \label{jmoms}
\end{equation}
where
\begin{equation}
\rho_{z_1,z_2} = \frac{|z_{1}|^{2}(1-|z_{2}|^{2})^{2}}{|1-z_{1}\overline{z_2}|^{2}}.
\end{equation}
\end{lemma}
\begin{proof}
Since both sides of \eqref{jmoms} are entire functions of $s$, it suffices to prove the result for any $s \in \mathbb{R}$ and $\mathrm{Re}(h)>-1$. By the definitions in \eqref{gfield}, the mean vector and relation matrix of $Q = (G(z_1),G'(z_2))$ are identically zero. Then $Q$ is characterised by its covariance matrix $\Gamma$ in \eqref{gamz1z2-intro}. Defining $D = \det(\Gamma)$, the joint density of $Q$ is
\begin{equation}
\begin{split}
f_{Q}(w_1,w_2) &= \frac{1}{\pi^{2}D}\,\mathrm{exp}\left(-\bm{w}^{\dagger}\Gamma^{-1}\bm{w}\right)\\
&=  \frac{1}{\pi^{2}D}\,\mathrm{exp}\left(-\frac{1}{D}\left(|w_1|^{2}\Gamma_{22}-\overline{w_1}w_2\Gamma_{12}-\overline{w_2}w_{1}\overline{\Gamma_{12}}+|w_{2}|^{2}\Gamma_{11}\right)\right). \label{jQ}
\end{split}
\end{equation}
By expression \eqref{lambdatog}, the left-hand side of \eqref{jmoms} is
\begin{equation}
I=\int_{\mathbb{C}}d^{2}w_{2}\,\int_{\mathbb{C}}d^{2}w_{1}\,|w_{2}|^{2h}\,e^{sw_{1}+s\overline{w_{1}}}\,f_{Q}(w_1,w_2).
\end{equation}
In the integral over $w_{1}$, we complete the square. We write the relevant terms in the exponential as
\begin{equation}
-\frac{\Gamma_{22}}{D}\left(\bigg{|}w_{1}-\frac{\overline{w_2\Gamma_{12}}+sD}{\Gamma_{22}}\bigg{|}^{2}\right) + \frac{|\overline{w_2\Gamma_{12}}+sD|^{2}}{\Gamma_{22}D}.
\end{equation}
Then the integral over $w_{1}$ gives a contribution $D/\Gamma_{22}$. The remaining terms inside the exponential equal $\Gamma_{22}^{-1}$ multiplied by
\begin{equation}
    -\frac{|w_{2}|^{2}}{D}\left(\Gamma_{11}\Gamma_{22}-|\Gamma_{12}|^{2}\right) + sw_{2}\Gamma_{12}+s\overline{w_{2}\Gamma_{12}}+s^{2}D.
\end{equation}
Recalling $D=\det(\Gamma)$ and scaling $w_{2} \to w_{2}\sqrt{\Gamma_{22}}$ shows that
\begin{equation}
I = \Gamma_{22}^{h}e^{\frac{s^{2}D}{\Gamma_{22}}}\frac{1}{\pi}\,\int_{\mathbb{C}}d^{2}w_{2}\,|w_{2}|^{2h}e^{-|w_{2}|^{2}+\frac{sw_{2}\Gamma_{12}}{\sqrt{\Gamma_{22}}}+\frac{s\overline{w_{2}\Gamma_{12}}}{\sqrt{\Gamma_{22}}}}.
\end{equation}
To compute the integral above, let $\xi = \frac{\Gamma_{12}}{\sqrt{\Gamma_{22}}}$ and expand the exponential as
\begin{equation}
\begin{split}
&\frac{1}{\pi}\,\int_{\mathbb{C}}d^{2}w_{2}|w_{2}|^{2h}e^{-|w_2|^{2}+s\xi w_{2}+s\overline{\xi w_{2}}}\\
&=\sum_{k_1=0}^{\infty}\sum_{k_2=0}^{\infty}\frac{(s\xi)^{k_1}(s\overline{\xi})^{k_2}}{(k_1)!(k_2)!}\frac{1}{\pi}\int_{\mathbb{C}}d^{2}w_{2}|w_{2}|^{2h}e^{-|w_2|^{2}}(w_{2})^{k_1}(\overline{w_2})^{k_2}\\
&=\sum_{k_1=0}^{\infty}\sum_{k_2=0}^{\infty}\frac{(s\xi)^{k_1}(s\overline{\xi})^{k_2}}{(k_1)!(k_2)!}\delta_{k_1,k_2}\Gamma\left(h+\frac{k_1+k_2}{2}+1\right)\\
&=\sum_{k=0}^{\infty}\frac{\Gamma(h+k+1)}{(k!)^{2}}\,s^{2k}|\xi|^{2k}\\
&=\Gamma(h+1)\,{}_1F_1(h+1,1;s^{2}|\xi|^{2}). \label{steps}
\end{split}
\end{equation}
Therefore
\begin{equation}
I = \Gamma_{22}^{h}e^{\frac{s^{2}D}{\Gamma_{22}}}\Gamma(h+1)\,{}_1F_1\left(h+1,1;s^{2}\frac{|\Gamma_{12}|^{2}}{\Gamma_{22}}\right) \label{two-point-result}.
\end{equation}
From the matrix elements of $\Gamma$ in \eqref{gamz1z2-intro} we have 
\begin{equation}
\begin{split}
\Gamma_{22}^{h}e^{\frac{s^{2}D}{\Gamma_{22}}}= \frac{1}{(1-|z_{2}|^{2})^{2h}(1-|z_{1}|^{2})^{s^{2}}}\,e^{-s^{2}\frac{|\Gamma_{12}|^{2}}{\Gamma_{22}}} \label{eqn1gcomp}
\end{split}
\end{equation}
where
\begin{equation}
\frac{|\Gamma_{12}|^{2}}{\Gamma_{22}} = \frac{|z_{1}|^{2}(1-|z_{2}|^{2})^{2}}{|1-z_{1}\overline{z_2}|^{2}}. \label{eqn2gcomp}
\end{equation}
Inserting \eqref{eqn1gcomp} and \eqref{eqn2gcomp} into \eqref{two-point-result} completes the proof.
\end{proof}
We now consider an alternative derivation in the special case $h=s$ and $s \in \mathbb{N}$. We first note that in this case, the formula obtained in \eqref{jmoms} reduces to \eqref{derivmoms2}. This follows immediately from the identity
\begin{equation}
s!\,{}_1F_1(s+1,1;x) =s!e^{x}L_{s}(-x). \label{1f1tolag}
\end{equation}
In the positive integer setting, a simpler proof can be given based on the identity
\beas
    &&\mathbb{E}(|\Lambda'(z)|^{2s}) \\
    &=&\frac{\partial}{\partial z_1}\ldots\frac{\partial}{\partial z_s}\frac{\partial}{\partial w_1}\ldots \frac{\partial}{\partial w_s}\mathbb{E}\left(\Lambda(z_1)\ldots \Lambda(z_s)\overline{\Lambda(\overline{w_1})}\ldots \overline{\Lambda(\overline{w_s})}\right)\bigg{|}_{\bm z=z, \bm w =\bar{z}},
\eeas
where the notation $\bm{z}=z$ means setting $z_{j}=z$ for each $j=1,\ldots,s$. The latter expectation can be computed by exploiting the logarithmic correlations of the Gaussian field $G$ in \eqref{gfield}. This leads to the identity
\begin{equation}
    \mathbb{E}\left(\Lambda(z_1)\ldots \Lambda(z_s)\overline{\Lambda(\overline{w_1})}\ldots \overline{\Lambda(\overline{w_s})}\right) = \prod_{i=1}^{s}\prod_{j=1}^{s}\frac{1}{1-z_iw_j}. \label{defofF}
\end{equation}
It turns out the needed derivatives of this can be performed and we find the following.
\begin{lemma}
\label{th:derivmomsint}
Let $s \in \mathbb{N}$. For any fixed $z$ with $|z|<1$, we have
\begin{equation}
    \mathbb{E}\left(|\Lambda'(z)|^{2s}\right) = \frac{1}{(1-|z|^{2})^{s^{2}+2s}}\,\sum_{k=0}^{s}\binom{s}{k}^{2}(s-k)!\,(|z|s)^{2k}. \label{derivmoms}
\end{equation}
\end{lemma}
\begin{proof}
Let
\bea\label{definitionnewadded}
F(\z,\w)=\prod_{i=1}^{s}\prod_{j=1}^{s}\frac{1}{1-z_iw_j}.
\eea
Derivatives are easily performed in $z_1,\ldots,z_s$ and merged yielding the identity
    \begin{equation}
\frac{\partial^{s}}{\partial z_1 \ldots \partial z_s} F(\bm z, \bm w)\bigg{|}_{\bm{z}=z} = F(z,\bm w)H(z,\bm w) \label{prod2terms}
\end{equation}
where
\begin{equation}
    F(z,\bm w) = \prod_{j=1}^{s}\frac{1}{(1-zw_j)^{s}},\quad \textrm{and} \quad H(z,\bm w) = \left(\sum_{l=1}^{s}\frac{w_{l}}{1-zw_l}\right)^{s}.
\end{equation}
To further differentiate with respect to $w_1,\ldots,w_s$ we use the product rule repeatedly on the right-hand side of \eqref{prod2terms}. After doing this for each of the $s$ derivatives, we obtain an expansion with $2^{s}$ different terms. When we merge variables $w_1,\ldots,w_s$, since $F$ and $H$ are both symmetric functions, we only need to keep track of the total number of derivatives acted either on $F$ or on $H$. We therefore obtain the binomial expansion
\begin{equation}
    \mathbb{E}\left(|\Lambda'(z)|^{2s}\right) = \sum_{k=0}^{s}\binom{s}{k}F_{w_1,\ldots,w_k}\bigg{|}_{\bm{w}=\bar{z}}H_{w_1,\ldots,w_{s-k}}\bigg{|}_{\bm{w}=\bar{z}}. \label{binexpan}
\end{equation}
Here the subscripts denote partial derivatives in the indicated variables. We have,
\begin{equation}
    \begin{split}
F_{w_1,\ldots,w_k}\bigg{|}_{\bm{w}=\bar{z}} &= \prod_{j=1}^{k}\frac{sz}{(1-zw_j)^{s+1}}\prod_{j=k+1}^{s}\frac{1}{(1-zw_j)^{s}}\bigg{|}_{\bm{w}=\bar{z}}\\
&=\frac{(sz)^{k}}{(1-|z|^{2})^{k+s^{2}}}. \label{Fderiv}
\end{split}
\end{equation}
Similarly, 
\begin{equation}
    \begin{split}
H_{w_1,\ldots,w_{s-k}}\bigg{|}_{\bm{w}=\bar{z}} &= \prod_{j=1}^{s-k}\frac{1}{(1-zw_j)^{2}}\left(\sum_{l=1}^{s}\frac{w_l}{1-zw_l}\right)^{k}\,\frac{s!}{k!}\bigg{|}_{\bm{w}=\bar{z}}\\
&= (1-|z|^{2})^{-2(s-k)}\left(\frac{s\bar{z}}{1-|z|^{2}}\right)^{k}\binom{s}{k}(s-k)!. \label{Hderiv}    
    \end{split}
\end{equation}
Now inserting \eqref{Hderiv} and \eqref{Fderiv} into \eqref{binexpan} completes the proof.
\end{proof}
The strategy to prove Theorems \ref{th:non-integer-intro} and \ref{th:jointmoms-intro} is now the following. By Lemma \ref{le:conv-law}, we know that $G_{N}(z_{2}) := \log\Lambda_{N}(z_2)$ and $G_{N}'(z_{1})$ converge jointly in distribution to $G(z_{2})$ and $G'(z_{1})$. Then it remains to verify a uniform integrability criterion to conclude convergence of the required expectations. To establish the latter in the negative moment case $-1<h<0$, a separate Fourier analytic technique is applied in the following subsection. Uniform integrability in the other ranges and the final proofs will be given in Section \ref{se:mainpfs}.

\subsection{Bounds on negative moments}
The goal of this section is to prove the following.
\begin{theorem}
\label{thm:negmomsbnd}
For any $a$ with $0\leq a<2$, let $r=|z|<1$ be fixed and $N > 4$. Then there is a constant $C$ depending only on $a$ and $r$ such that
\begin{equation}
\mathbb{E}(|G_{N}'(r)|^{-a}) \leq C \label{log-deriv-bnd}
\end{equation}
and
\begin{equation}
\mathbb{E}(|\Lambda'_{N}(r)|^{-a}) \leq C. \label{neg-mom-bnd}
\end{equation}
\end{theorem}
We begin by describing the general strategy, prove some preliminary lemmas and then we give the proof of Theorem \ref{thm:negmomsbnd} at the end of the section. We note that $G_{N}'(r)$ has the form of a linear statistic
\begin{equation}
G_{N}'(r) = \frac{\Lambda_{N}'(r)}{\Lambda_{N}(r)} = X_{N}(f)
\end{equation}
where $X_{N}(f) = \sum_{j=1}^{N}f(\theta_j)$ and
\begin{equation}
    f(\theta) = -\frac{e^{-\i\theta}}{1-re^{-\i\theta}}. \label{f-form}
\end{equation}
We start from
\begin{equation}
    \mathbb{E}(|X_{N}(f)|^{-a}) = \int_{0}^{\infty}x^{-a}\,\frac{d}{dx}\mathbb{P}(|X_{N}(f)| < x)\,dx.
\end{equation}
Integrating by parts, we have
\begin{equation}
\begin{split}
\mathbb{E}(|X_{N}(f)|^{-a}) &= -\lim_{y \to 0}y^{-a}\mathbb{P}(|X_{N}(f)|<y) + a\int_{0}^{\infty}y^{-a-1}\mathbb{P}(|X_{N}(f)|<y)dy.
\end{split}
\end{equation}
We can drop the part of the integral where $y>1$ because $y^{-a-1}$ is integrable on $[1,\infty)$ for $a>0$. So 
\begin{equation}
    \begin{split}
        \mathbb{E}(|X_{N}(f)|^{-a}) \leq \lim_{y \to 0}y^{-a}\mathbb{P}(|X_{N}(f)|<y)+a\int_{0}^{1}y^{-a-1}\mathbb{P}(|X_{N}(f))|<y)dy+C. \label{smalldev}
    \end{split}
\end{equation}
To bound the integral in \eqref{smalldev} for any $a<2$, we need a bound on $\mathbb{P}(|X_{N}(f))|<y)$ on the order $y^{2}$. This also implies that $\lim_{y \to 0}y^{-a}\mathbb{P}(|X_{N}(f)|<y) = 0$. To obtain the bound, we use the following `small ball' inequality
\begin{equation}
    \mathbb{P}(|X_{N}(f)| \leq y) \leq y^{2}\int_{|\xi_{1}|<y^{-1}}\int_{|\xi_{2}|<y^{-1}}d\xi_{1}\,d\xi_{2}\,|\mathbb{E}(e^{\i\xi_1 \mathrm{Re}(X_{N}(f))+\i\xi_{2}\mathrm{Im}(X_{N}(f))})|, \label{halasz}
    \end{equation}
see \cite[Sec 5]{H77}. For $\xi_{1},\xi_{2} \in \mathbb{R}$, define
\begin{equation}\label{defofpsi}
\psi_{N}(\xi_{1},\xi_{2}) := \mathbb{E}(e^{\i\xi X_{N}(g)})
\end{equation}
where $\xi = |\xi_{1}|+|\xi_{2}|$ and
\begin{equation}
    g(\theta) = \frac{\xi_{1}\mathrm{Re}(f(\theta))+\xi_{2}\mathrm{Im}(f(\theta))}{|\xi_1|+|\xi_2|}. \label{gdef}
\end{equation}
Although the above definition of $g(\theta)$ depends on $\xi_{1}$, $\xi_{2}$, for convenience of notation, in the following we still use $g(\theta)$. Then \eqref{halasz} becomes
\begin{equation}
\label{fourierT}
\mathbb{P}(|X_{N}(f)| \leq y) \leq y^{2}\int_{|\xi_{1}|<y^{-1}}\int_{|\xi_{2}|<y^{-1}}d\xi_{1}\,d\xi_{2}\,|\psi_{N}(\xi_{1},\xi_{2})|.
\end{equation}
Due to the factor $y^{2}$ in \eqref{fourierT}, it will be sufficient to control the double integral uniformly in $N$ and $y$. To do this we adopt the strategy in Johansson \cite{K97} where precise bounds are obtained for the case of trigonometric polynomial $g$. Since our $g$ defined in (\ref{gdef}) is not a trigonometric polynomial, we discuss here some of the similarities and differences in the proof.

We start by proving a similar result to \cite[Lemma 2.14]{K97} which bounds the high frequency contributions to the characteristic function. If $g$ is a polynomial of fixed degree, $g'$ can at worst have a fixed number of zeros. Then the result follows from a standard stationary phase approximation given in \cite[Chapter 6]{BH86}. As our function $g$ is not a polynomial, the behaviour of zeros of $g'$ is less immediate, so we will carefully check that it satisfies the conditions needed in the stationary phase approximation. We also take extra care because our $g$ depends on $\xi_{1},\xi_{2}$ and our bounds need to be uniform in these parameters.

\begin{lemma}
\label{lem:high-fourier}
Let $0<r<1$. If $\xi>N^{8}$, then there is a constant $C>0$ depending only on $r$ such that for all $N$ we have
\begin{equation}
    |\psi_{N}(\xi_{1},\xi_{2})| \leq C^{N}N^{-N/2}\xi^{-N/4}.
\end{equation}
\end{lemma}
\begin{proof}
Let $h(\theta) = e^{\i \xi g(\theta)}$. We write the characteristic function as the Toeplitz determinant
\begin{equation}
\psi_{N}(\xi_{1},\xi_{2})= \mathrm{det}\bigg\{\hat{h}_{j-k}\bigg\}_{j,k=0}^{N-1}
\end{equation}
where $\hat{h}_{k} = \frac{1}{2\pi}\int_{0}^{2\pi}h(\theta)e^{-\i k\theta}\,d\theta$. By Hadamard's inequality
\begin{equation}
    |\psi_{N}(\xi_{1},\xi_{2})| \leq \prod_{j=1}^{N}\left(\sum_{k=1}^{N}|\hat{h}_{j-k}|^{2}\right)^{\frac{1}{2}}. \label{had}
\end{equation}
Now we apply the stationary phase approximation to the Fourier integral $\hat{h}_{k}$. Using \eqref{gdef}, we write out $g(\theta)$ explicitly as
\begin{equation*}
    g(\theta) = \frac{\eta_{1}(r-\cos(\theta))+\eta_{2}\sin(\theta)}{r^{2}+1-2r\cos(\theta)} \label{gexplicit}
\end{equation*}
where 
\begin{equation}
    \eta_{1} = \frac{\xi_{1}}{|\xi_{1}|+|\xi_{2}|}, \qquad \eta_{2} = \frac{\xi_{2}}{|\xi_{1}|+|\xi_{2}|}.
\end{equation}
We have
\begin{equation}
\label{gprimeform}
    g'(\theta) = \frac{\eta_{2}(1+r^{2})\cos(\theta)+\eta_{1}(1-r^{2})\sin(\theta)-2\eta_{2}r}{(r^{2}+1-2r\cos(\theta))^{2}}.
\end{equation}
To find stationary points, we look for values $\theta^{*}$ such that the numerator above vanishes. Since $\xi_{2}$ appears as an integration variable in \eqref{fourierT}, we may assume that $\eta_{2} \neq 0$ in what follows. Then $\sin(\theta^*)$ satisfies the following equation:
\bea\label{equationforthezero}
\Big(\frac{\eta_{1}^2}{\eta_{2}^2}(1-r^2)^2+(1+r^2)^2\Big)(\sin (\theta^*))^2-4\frac{\eta_{1}}{\eta_{2}}r(1-r^2)\sin(\theta^*)-(1-r^2)^2=0.
\eea
Combining with
\beas
(1+r^2)\cos(\theta^*)+\frac{\eta_{1}}{\eta_{2}}(1-r^2)\sin(\theta^*)=2r,
\eeas
we conclude that $g'(\theta)=0$ has at most two isolated zeros. Next we check the order of the zero, i.e. that the second derivative stays non-zero. It is straightforward algebra to check that
\begin{equation*}
    g''(\theta^{*}) = -\frac{\sin(\theta^{*})(1-r^{2})^{2}}{(r^2+1-2r\cos(\theta^*))^3}\frac{(\eta_{1}^{2}+\eta_{2}^{2})}{\eta_{2}}.
\end{equation*}
By (\ref{equationforthezero}), it is not hard to check that $|g''(\theta^{*})|>c$ for some absolute constant $c>0$ depending only on $r$.

Hence the conditions for applying the stationary phase approximation are satisfied. Applying standard results from \cite[Chapter 6]{BH86}, we obtain the bound
\begin{equation}
    |\hat{h}_{k}| \leq C\frac{|k|+1}{\xi^{1/2}}.
\end{equation}
Inserting this into Hadamard's inequality \eqref{had} we find
\begin{equation}
\label{xibnd}
\begin{split}
    |\psi_{N}(\xi_{1},\xi_{2})| &\leq C^{N}\prod_{j=1}^{N}\left(\sum_{k=1}^{N}(|k-j|+1)^{2}/\xi\right)^{\frac{1}{2}}\\
    &\leq C^{N}\left(\sum_{k=1}^{N}(k+1)^{2}/\xi\right)^{\frac{N}{2}}\\
    &\leq C^{N}N^{-N/2}\xi^{-N/4}\,\left(N^{2N}\xi^{-N/4}\right).
    \end{split}
\end{equation}
Since $\xi > N^{8}$, we have $N^{2N}\xi^{-N/4} \leq 1$. This completes the proof of the Lemma.
\end{proof}

\begin{lemma}
\label{lem:klem}
Let $g : [0,2\pi] \to \mathbb{R}$ be a smooth, periodic function with period $2\pi$. Introduce the quantities
\begin{equation}
\begin{split}
\label{quantities}
b &= \frac{1}{2\pi}\int_{0}^{2\pi}(g'(\theta))^{2}d\theta,\\
h(\theta) &= g'(\theta)-g'(0),\\
k(\theta) &=  (g'(\theta))^{2}-g'(\theta)g'(0)-b,\\
A(k) &= \sum_{n=1}^{\infty}n|\hat{k}_{n}|^{2}, \qquad \hat{k}_{n} = \frac{1}{2\pi}\int_{0}^{2\pi}k(\theta)e^{-\i n\theta}d\theta,
\end{split}
\end{equation}
and $d_{1} = \left\lVert(h')^{2}\right\rVert_{\infty} := \sup_{\theta \in [0,2\pi]}|h'(\theta)|^{2}$. Then for any $\xi \in \mathbb{R}$, we have
\begin{equation}
|\mathbb{E}(e^{i\xi \sum_{j=1}^{N}g(\theta_{j})})| \leq \mathrm{exp}\left(-\frac{b^{2}\xi^{2}}{4(A(k)\xi^{2}/N^{2}+d_{1}(1+\xi/N))}\right) \label{k-bound}
\end{equation}
where the expectation taken with respect to the eigenvalues $\{e^{i\theta_{j}}\}_{j=1}^{N}$ of the CUE.
\end{lemma}

\begin{proof}
This follows directly from \cite[Proof of Proposition 2.8]{K97}.
\end{proof}

\begin{corollary}
\label{lem:intermediate-fourier}
Let $0 < r < 1$ be fixed and $\xi = |\xi_{1}|+|\xi_{2}|$. If $\xi < N$, there is a constant $c>0$ depending only on $r$ such that
\begin{equation}
|\psi_{N}(\xi_{1},\xi_{2})| \leq e^{-c\xi_{1}^{2}-c\xi_{2}^{2}} \label{low-fourier}
\end{equation}
while if $\xi \geq N$, we have
\begin{equation}
|\psi_{N}(\xi_{1},\xi_{2})| \leq e^{-cN^{2}}. \label{mid-fourier}
\end{equation}
\end{corollary}

\begin{proof}
We apply Lemma \ref{lem:klem} with the choice of smooth function $g$ in \eqref{gdef}. We begin by making explicit the quantities appearing in \eqref{k-bound}. A direct computation using \eqref{gprimeform} shows that
\begin{equation}
b = \frac{(\eta_{1}^{2}+\eta_{2}^{2})(r^{2}+1)}{2(1-r)^{3}(1+r)^{3}}.
\end{equation}
Using $\xi^{2}(\eta_{1}^{2}+\eta_{2}^{2}) = \xi_{1}^{2}+\xi_{2}^{2}$ and the bound $\eta_{1}^{2}+\eta_{2}^{2} \geq \frac{1}{2}$ shows that $b^{2}\xi^{2} \geq c(\xi_{1}^{2}+\xi_{2}^{2})$ for some constant $c>0$. Since $g$ is a smooth function on the unit circle, the Fourier coefficients $\hat{k}_{n} = O(n^{-\ell})$ as $n \to \infty$, for any $\ell>0$. Hence, the factor $A(k)$ is finite. In terms of the quantities given in \eqref{quantities}, it is not hard to check that  
\be
d_{1} = \left\lVert(h')^{2}\right\rVert_{\infty} = \left\lVert(g'')^{2}\right\rVert_{\infty} \leq C(1-r)^{-12}
\ee
for some $C$ independent of $\xi_{1}$, $\xi_{2}$ and $r$. Using these bounds in \eqref{k-bound} with the assumption $\xi < N$ immediately gives \eqref{low-fourier}. For the regime $\xi \geq N$, rewrite \eqref{k-bound} as 
\begin{equation}
|\psi_{N}(\xi_{1},\xi_{2})| \leq \mathrm{exp}\left(-\frac{b^{2}N^{2}}{4(A(k)+d_{1}N^{2}/\xi^{2}+d_{1}N/\xi)}\right).
\end{equation}
Using that $b$ is bounded from below and $\xi \geq N$ completes the proof of \eqref{mid-fourier}.
\end{proof}

\begin{proof}[Proof of Theorem \ref{thm:negmomsbnd}]
Since our bounds for $\psi_{N}(\xi_1,\xi_2)$ only depend on $|\xi_{1}|+|\xi_{2}|$, it suffices to assume $\xi_{1}\geq0$ and $\xi_{2} \geq 0$. We split the integral in \eqref{fourierT} over the three regions $R_{1} = \{(\xi_1,\xi_2) \mid \xi < N\}, R_{2} = \{(\xi_1,\xi_2) \mid N \leq \xi < N^{8}\}$ and $R_{3} = \{(\xi_1,\xi_2) \mid \xi > N^{8}\}$ and denote the corresponding integrals $I_{1}$, $I_{2}$ and $I_{3}$. For $I_{3}$, we employ the bound of Lemma \ref{lem:high-fourier} and obtain $I_{3} \leq C^{N}N^{-N/2}(I_{3,1}+I_{3,2})$ where
\begin{equation}
\begin{split}
I_{3,1} &= \int_{0}^{N^{8}}\int_{N^{8}-\xi_{2}}^{\infty}(\xi_{1}+\xi_{2})^{-N/4}d\xi_{1}d\xi_{2}\\
I_{3,2} &= \int_{0}^{\infty}\int_{N^{8}}^{\infty}(\xi_{1}+\xi_{2})^{-N/4}d\xi_{1}d\xi_{2}.
\end{split}
\end{equation}
Both integrals above are bounded by $N^{15}N^{-2N}$ provided $N>8$ and consequently $I_{3} \leq C^{N}N^{-5N/2}$ for some (possibly adjusted) constant $C>0$.

The integral $I_{2}$ is also exponentially small because $|R_{2}| \leq N^{16}$ and by Lemma \ref{lem:intermediate-fourier} we have $|\psi_{N}(\xi_1,\xi_2)| \leq e^{-cN^{2}}$ on $R_{2}$. By the same Lemma
\begin{equation}
I_{1} \leq \int_{0}^{\infty}\int_{0}^{\infty}e^{-c(\xi_{1}^{2}+\xi_{2}^{2})}d\xi_{1}d\xi_{2} = \frac{\pi}{4c}.
\end{equation}
Putting all the bounds together, \eqref{fourierT} implies that $\mathbb{P}(|X_{N}(f)|<y) \leq \frac{1}{c}y^{2}$. Plugging this estimate into \eqref{smalldev} completes the proof of \eqref{log-deriv-bnd}. To show \eqref{neg-mom-bnd} H\"older's inequality gives
\begin{equation}
\mathbb{E}(|\Lambda_{N}'(r)|^{-a}) \leq \mathbb{E}\left(|G_{N}'(r)|^{-aq}\right)^{1/q}\mathbb{E}\left(|\Lambda_{N}(r)|^{-a\ell}\right)^{1/\ell}
\end{equation}
where $q = \frac{\ell}{\ell-1}>1$ and we choose $\ell$ large enough such that $aq < 2$. Then the first term above is bounded by \eqref{log-deriv-bnd}. By \cite[Corollary 2]{FK04}, the negative moments $\mathbb{E}\left(|\Lambda_{N}(r)|^{-a\ell}\right)$ converge as $N \to \infty$, hence they are bounded for $N$ sufficiently large, say $N>N_{0}$. Boundedness for any $N$ with $4<N\leq N_{0}$ can be obtained as follows. By (\ref{xibnd}), we have $|\psi_{N}(\xi_{1},\xi_{2})|$ is bounded by $C_{N_{0}}\xi^{-N/2}$ for any $(\xi_{1},\xi_{2})$ for some constant depending only on $N_{0}$. This can be used to bound the right-hand side of \eqref{fourierT} on the domain $(\xi_{1},\xi_{2}) \in \mathbb{R}_{+}^{2}\setminus [0,1]^{2}$. On the domain $(\xi_{1},\xi_{2}) \in [0,1]^{2}$ we use the trivial bound $|\psi_{N}(\xi_{1},\xi_{2})|\leq 1$. Then by (\ref{smalldev}) and (\ref{fourierT}), we conclude the boundedness of $\mathbb{E}(|G_{N}'(r)|^{-a})$ in the range $4<N\leq N_{0}$. By the above boundedness and the fact that $\Lambda_{N}(r)$ is bounded away from zero in this range, we conclude that $\mathbb{E}(|\Lambda_{N}'(r)|^{-a}) = \mathbb{E}(|G_{N}'(r)\Lambda_{N}(r)|^{-a})$ is bounded.
\end{proof}

\section{Moments of the derivative of the Riemann zeta function off the critical line}\label{derivativeoftheRiemannzetafunction}

The goal of this section is to verify Conjecture \ref{conjectureonthemomentsoffthecriticalline} for some special cases, i.e., Theorems \ref{checkthecasek=2} and \ref{assumetheLindelofconjecture}.

\begin{lemma}\label{shifted derivative}
Let $\alpha_{1},\alpha_{2}$ and $\beta_{1},\beta_{2}$ are sufficiently small complex numbers.
Then for any $\sigma> 1/2$,
\beas
&&\lim_{T\rightarrow \infty}\frac{1}{T}\int_{1}^{T}|\zeta'(\sigma+\i t)|^4dt=\frac{\partial} {\partial\alpha_{1}}\frac{\partial} {\partial\alpha_{2}}\frac{\partial }{\partial \beta_{1}}\frac{\partial }{\partial \beta_{2}}\\
&&\frac{\zeta(2\sigma+\alpha_{1}+\beta_{1})\zeta(2\sigma+\alpha_{1}+\beta_{2})\zeta(2\sigma+\alpha_{2}+\beta_{1})\zeta(2\sigma+\alpha_{2}+\beta_{2})}{\zeta(4\sigma+\alpha_{1}+\alpha_{2}+\beta_{1}+\beta_{2})}\Bigg|_{\alpha_{1}=\alpha_{2}=\beta_{1}=\beta_{2}=0}.
\eeas
\end{lemma}
\begin{proof}
By the approximate functional equation for the Riemann zeta function, 
\bea\label{approximate functional equation}
\zeta(s)=\sum_{n \leq x}n^{-s}+\chi(s)\sum_{n\leq y}n^{s-1}+O(x^{-\Re(s)}+|t|^{\frac{1}{2}-\Re(s)}y^{\Re(s)-1})
\eea
uniformly for $x,y\geq 1$ and $xy=\mathrm{Im}(s)/2\pi$, and $0<\Re(s)<1$.
Then by Cauchy's theorem, for $\sigma\geq 1/2$,
\bea\label{approximate functional equationforderivative}
\frac{d\zeta(s)}{ds}\Big|_{s=\sigma+\i t}&=&\sum_{n\leq x}\frac{d n^{-s}}{ds}\Big|_{s=\sigma+\i t}+\frac{d}{ds}\Big(\chi(s)\sum_{n\leq y}n^{s-1}\Big)\Big|_{s=\sigma+\i t}\nonumber\\
&&+O\Big((x^{-\frac{1}{2}}+y^{-\frac{1}{2}})\log |t|\Big).
\eea
Take $x=y=\sqrt{\frac{t}{2\pi}}$ for $t\geq 1$. Note that for $t\geq 1$,
\be\label{derivativeofchi}
\frac{\chi'(\sigma+\i t)}{\chi(\sigma+\i t)}=-\log \frac{s}{2\pi \i}+O(\frac{1}{|s|})=-\log\frac{t}{2\pi}+O(\frac{1}{t}).
\ee
So for $\sigma > 1/2$, by a similar argument to that of \cite[Theorem 7.5]{titchmarsh1986theory}, we have
\beas
&&\frac{1}{T}\int_{1}^{T}|\zeta'(\sigma+\i t)|^4dt\\
&=&\frac{1}{T}
\int_{1}^{T}\zeta'(\sigma+\alpha_{1}+\i t)\zeta'(\sigma+\alpha_{2}+\i t)\zeta'(\sigma+\beta_{1}-\i t)\zeta'(\sigma+\beta_{2}-\i t)dt\Big|_{\alpha_{1}=\alpha_{2}=\beta_{1}=\beta_{2}=0}\\
&=&\frac{\partial} {\partial\alpha_{1}}\frac{\partial} {\partial\alpha_{2}}\frac{\partial }{\partial \beta_{1}}\frac{\partial }{\partial \beta_{2}}\sum_{n_{1}n_{2}=n_{3}n_{4}}\frac{1}{n_{1}^{\sigma+\alpha_{1}}n_{2}^{\sigma+\alpha_{2}}n_{3}^{\sigma+\beta_{1}}n_{4}^{\sigma+\beta_{2}}}\Bigg|_{\alpha_{1}=\alpha_{2}=\beta_{1}=\beta_{2}=0}+O_{\epsilon}(T^{1-2\sigma+\epsilon}).
\eeas
Let
\beas
\sigma_{\alpha,\beta}(n)=\sum_{n_{1}n_{2}=n}n_{1}^{-\alpha}n_{2}^{-\beta}=n^{-\alpha-\beta}\sum_{n_{1}n_{2}=n}n_{1}^{\alpha}n_{2}^{\beta}=n^{-\alpha}\sum_{d|n}d^{\alpha-\beta}.
\eeas
Then
\beas
\sum_{\substack{n_{1},n_{2},n_{3},n_{4}=1\\n_{1}n_{2}=n_{3}n_{4}}}^{\infty}\frac{1}{n_{1}^{\sigma+\alpha_{1}}n_{2}^{\sigma+\alpha_{2}}n_{3}^{\sigma+\beta_{1}}n_{4}^{\sigma+\beta_{2}}}&=&\sum_{n=1}^{\infty}\frac{\sigma_{\alpha_{1},\alpha_{2}}(n)\sigma_{\beta_{1},\beta_{2}}(n)}{n^{2\sigma}}.
\eeas
Note that when $\sigma>1/2$, the series $\sum_{n=1}^{\infty}\frac{\sigma_{\alpha_{1},\alpha_{2}}(n)\sigma_{\beta_{1},\beta_{2}}(n)}{n^{2\sigma}}$ is convergent when $\alpha_{1},\alpha_{2},\beta_{1},\beta_{2}$ are in a sufficiently small neighborhood of $0$.  In the following, we shall use the Euler product to study the series. To do that, we mainly use the following Ramanujan's identity that
\beas
\sum_{n=1}^{\infty}\Big(\frac{\alpha^{n+1}-\beta^{n+1}}{\alpha-\beta}\Big)\Big(\frac{\gamma^{n+1}-\delta^{n+1}}{\gamma-\delta}\Big)x^{n}=\frac{1-\alpha\beta\gamma\delta x^2}{(1-\alpha\gamma x)(1-\alpha\delta x)(1-\beta\gamma x)(1-\beta\delta x)}.
\eeas
So we have
\beas
&&\sum_{n=1}^{\infty}\frac{\sigma_{\alpha_{1},\alpha_{2}}(n)\sigma_{\beta_{1},\beta_{2}}(n)}{n^{2\sigma}}
=\prod_{p}\Bigg(1+\sum_{m=1}^{\infty}\sigma_{\alpha_{1},\alpha_{2}}(p^m)\sigma_{\beta_{1},\beta_{2}}(p^m)p^{-2m\sigma}\Bigg)\\
&&=\frac{\zeta(2\sigma+\alpha_{1}+\beta_{1})\zeta(2\sigma+\alpha_{1}+\beta_{2})\zeta(2\sigma+\alpha_{2}+\beta_{1})\zeta(2\sigma+\alpha_{2}+\beta_{2})}{\zeta(4\sigma+\alpha_{1}+\alpha_{2}+\beta_{1}+\beta_{2})}.
\eeas
Combining the above, we obtain the desired result stated in this lemma.
\end{proof}

\begin{proof}[Proof of Theorem \ref{checkthecasek=2}]
Note that $\zeta(w)\sim \frac{1}{w-1}$ as $w\rightarrow 1$ and $\zeta(2)=\frac{\pi^2}{6}$.
Observe that $\zeta'(w)\sim -\frac{1}{(w-1)^2}$ as $w\rightarrow 1$ and $\zeta''(w)\sim \frac{2}{(w-1)^3}$ as $w\rightarrow 1$. The asymptotic formula (\ref{verifyk=2moment}) follows from Lemma \ref{shifted derivative}.
\end{proof}

\begin{lemma}\label{underlindelof}
Assume the truth of the Lindel\"of hypothesis, then for any integer $s\geq 3$,

\bea 
\lim_{T\rightarrow \infty} \frac{1}{T} \int_1^T |\zeta'(\sigma+\i t)|^{2s} dt
=\sum_{n=1}^\infty \frac{((\log*\cdots*\log)(n))^2}{n^{2\sigma}}
\eea
for any $\sigma>1/2$.
\end{lemma}
\begin{proof}
It follows from \cite[Theorem 13.2]{titchmarsh1986theory} that 
\bea \label{restrictregionforintegers} 
 \lim_{T\rightarrow \infty} \frac{1}{T} \int_1^T |\zeta(\sigma+\i t)|^{2s} dt = \sum_{n=1}^\infty \frac{(d_s(n))^2}{n^{2\sigma}}
\eea
holds for any $\sigma>1/2$ and any integer $s\geq 1$ under the Lindel\"of hypothesis. For the validity of (\ref{restrictregionforintegersderivative}) for any $\sigma>1/2$, it follows from a similar argument to that of (\ref{restrictregionforintegers}) by doing the following modifications. Firstly, replacing
\beas
\sum_{n=1}^{\infty}\frac{d_{s}(n)}{n^{z}}e^{-\delta n}=\frac{1}{2\pi i}\int_{2-i\infty}^{2+i\infty}\Gamma(w-z)(\zeta(w))^{s}\delta^{z-w}dw,
\eeas
which is in the paragraph below \cite[Theorem 7.9]{titchmarsh1986theory}, by
\bea\label{weight}
(-1)^{s}\sum_{n=1}^{\infty}\frac{(\log*\cdots*\log)(n)}{n^{z}}e^{-\delta n}&=&\frac{1}{2\pi i}\int_{2-i\infty}^{2+i\infty}\Gamma(w-z)(\zeta'(w))^{s}\delta^{z-w}dw,
\eea
for $\delta > 0$ and $1<\Re(z)<2$, where $(\log*\cdots*\log)(n)$ is $s$-convolutions of $\log n$. Here we use the estimate that the average order of $(\log*\cdots*\log)(n)$ is bounded by that of $d_{s}(n)$ times $(\log n)^{s}$. This ensures that the series $\sum_{n=1}^{\infty}\frac{(\log*\cdots*\log)(n)}{n^{z}}$ is absolutely  convergent for $\Re(z)>1$, which is required in the  lemma below \cite[Theorem 7.9]{titchmarsh1986theory}. Secondly, assume the Lindel\"of hypothesis, that is $\zeta(\sigma+\i t)=O_{\epsilon}((|t|+1)^{\epsilon})$ for any $\sigma >1/2$. By Cauchy's theorem,
\bea
\zeta'(\sigma+\i t)&=&\frac{1}{2\pi \i}\int_{\mathcal{C}}\frac{\zeta(w)}{(w-\sigma-\i t)^2}dw,
\eea
where $\mathcal{C}$ is the positively oriented circle $|w-\sigma-\i t|=\epsilon$, we have
$\zeta'(\sigma+\i t)=O_{\epsilon}((|t|+1)^{\epsilon})$ for any $\sigma > 1/2$. Then we have for any $\sigma>1/2$ and any integer $s\geq 0$,
\beas
\frac{1}{T}\int_{1}^{T}|\zeta'(\sigma+\i t)|^{2s}dt=O_{\epsilon}(T^{\epsilon}).
\eeas
We use this result to replace \cite[formula 7.9.2]{titchmarsh1986theory}.
\end{proof}

\begin{proof}[Proof of Theorem \ref{assumetheLindelofconjecture}]
By Lemma \ref{underlindelof} and assuming the Lindel\"of hypothesis, for any  $\sigma>1/2$,
\bea\label{restrictregionforintegersderivative2}
\lim_{T\rightarrow \infty} \frac{1}{T} \int_1^T |\zeta'(\sigma+\i t)|^{2s} dt
=\sum_{n=1}^\infty \frac{((\log*\cdots*\log)(n))^2}{n^{2\sigma}},
\eea
where $(\log*\cdots*\log)(n)$ is $s$-convolutions of $\log n$. 
Let $\alpha_{1},\ldots,\alpha_{s}$ and $\beta_{1},\ldots,\beta_{s}$ are sufficiently small complex numbers. We now consider the following series.
\beas
\sum_{m=1}^{\infty}\frac{\sigma_{\alpha_{1},\ldots,\alpha_{s}}(m)\sigma_{\beta_{1},\ldots,\beta_{s}}(m)}{m^{2\sigma}},
\eeas
where
\beas
\sigma_{\alpha_{1},\ldots,\alpha_{s}}(m)=\sum_{n_{1}\cdots n_{s}=m}n_{1}^{-\alpha_{1}}\cdots n_{s}^{-\alpha_{s}}.
\eeas
Observe  that for $\delta>0$ sufficiently small and $|\alpha_{i}|\leq \delta$ and $|\beta_{i}|\leq \delta$ for $i=1,\ldots,s$, $|\sigma_{\alpha_{1},\ldots,\alpha_{s}}(m)|\leq m^{\delta}d_{s}(m)$ and $d_{s}(m)$ is the $s$-fold divisor function. So in this region, the series is absolutely convergent uniformly with respect to $\alpha_{1},\ldots,\alpha_{s},\beta_{1},\ldots,\beta_{s}$. By (\ref{restrictregionforintegersderivative2}), we have
\bea\label{identity}
&&\lim_{T\rightarrow \infty} \frac{1}{T} \int_1^T |\zeta'(\sigma+\i t)|^{2s} dt\nonumber \\
&=&\prod_{i=1}^{s}\frac{\partial}{\partial \alpha_{i}}\prod_{j=1}^{s}\frac{\partial}{\partial \beta_{j}}\sum_{m=1}^{\infty}\frac{\sigma_{\alpha_{1},\ldots,\alpha_{s}}(m)\sigma_{\beta_{1},\ldots,\beta_{s}}(m)}{m^{2\sigma}}\Bigg|_{\alpha_{1}=\ldots=\alpha_{s}=0,\beta_{1}=\ldots=\beta_{s}=0}.
\eea
By the Euler product, we know that
\bea
&&\sum_{m=1}^{\infty}\frac{\sigma_{\alpha_{1},\ldots,\alpha_{s}}(m)\sigma_{\beta_{1},\ldots,\beta_{s}}(m)}{m^{2\sigma}}\nonumber\\ 
&=& \prod_{p} \Bigg( 1 + \sum_{n=1}^{\infty}\frac{1}{p^{2n\sigma}}  \sum_{\substack{a_1+\cdots+a_s=n \\ b_1+\cdots+b_s=n \\ a_i,b_j \geq 0}}
\frac{1}{p^{a_1\alpha_1}\cdots p^{a_s\alpha_s} p^{b_1\beta_1}\cdots p^{b_s\beta_s} } \Bigg) \label{statethereason}\\
&=&H(\alpha_1,\ldots,\alpha_s, \beta_1,\ldots, \beta_s) \prod_{i,j=1}^s \zeta(2\sigma + \alpha_i+\beta_j)\nonumber.
\eea
Note that (\ref{statethereason})$=\prod_{p}\Bigg(1+\sum_{i,j=1}^{s}p^{-2\sigma-\alpha_{i}-\beta_{j}}+O_{s}(p^{-4\sigma-4\min_{i,j}\{\alpha_{i},\beta_{j}\}+\epsilon})\Bigg)$ and \beas
\prod_{i,j=1}^s \zeta(2\sigma + \alpha_i+\beta_j)^{-1}=\Bigg(1-\sum_{i,j=1}^{s}p^{-2\sigma-\alpha_{i}-\beta_{j}}+O_{s}(p^{-4\sigma-4\min_{i,j}\{\alpha_{i},\beta_{j}\}})\Bigg).
\eeas
Then we have that $H(\alpha_1,\ldots,\alpha_s, \beta_1,\ldots, \beta_s)$ is absolutely convergent when $|\alpha_{i}|<\frac{1}{4}$ and $|\beta_{j}|<\frac{1}{4}$ for all $i,j=1,\ldots,s$.
Let
\bea\label{analytic in a wide region}
f(w)=\prod_{p}\Bigg((1-p^{-w})^{s^2}\sum_{m=0}^{\infty}d_{s}(p^m)^2p^{-mw}\Bigg).
\eea
Then $f(w)$ is holomorphic and bounded for $\{w:\Re(w)> \frac{1}{2}\}$. Let $\delta$ be a sufficiently small complex number. Observe that 
\bea\label{December17}
H(\delta,\ldots,\delta,\delta,\ldots, \delta)
=\left(\zeta(2\sigma+2\delta) \right)^{-s^2} \sum_{n=0}^\infty \frac{d_s(n)^2}{n^{2\sigma+2\delta}}=f(2\sigma+2\delta)
\eea
By (\ref{identity}) and (\ref{statethereason}),
\bea
&&  \hspace{-1cm}  \lim_{T\rightarrow \infty} \frac{1}{T} \int_1^T |\zeta'(\sigma+\i t)|^{2s}dt \nonumber
\\
&& \hspace{-1.4cm} =\prod_{i=1}^{s}\frac{\partial}{\partial \alpha_{i}}\prod_{j=1}^{s}\frac{\partial}{\partial \beta_{j}}H(\alpha_1,\ldots,\alpha_s, \beta_1,\ldots, \beta_s) \prod_{i,j=1}^s \zeta(2\sigma + \alpha_i+\beta_j)\Bigg|_{\substack{ \alpha_{1}=\ldots=\alpha_{s}=0 \\ \beta_{1}=\ldots=\beta_{s}=0}}.
\label{1002eq1}
\eea
On the other hand, for integers $0\leq m,n\leq s$,
\beas
&&\prod_{i=1}^{m}\frac{\partial}{\partial \alpha_{i}}\prod_{j=1}^{n}\frac{\partial}{\partial \beta_{j}}
\prod_{i,j=1}^s \zeta(2\sigma + \alpha_i+\beta_j)\Bigg|_{\substack{ \alpha_{1}=\ldots=\alpha_{s}=0 \\ \beta_{1}=\ldots=\beta_{s}=0}}\\
&=& \prod_{i=1}^m \frac{\partial}{\partial \alpha_{i}} 
\left(\prod_{j=1}^s \zeta(2\sigma+\alpha_j)\right)^n
\left(\sum_{k=1}^s \frac{\zeta'(2\sigma+\alpha_k)}{\zeta(2\sigma+\alpha_k)} \right)^n
\Bigg|_{\alpha_{1}=\ldots=\alpha_{s}=0} \\
&=&\sum_{j=0}^m \binom{m}{j} 
n^j \frac{n!}{j!} 
\left(\prod_{i=1}^j \zeta(2\sigma)^{n-1} \zeta'(2\sigma)\right)
\left(\prod_{i=j+1}^s \zeta(2\sigma)^{n} \right) \\
&& \times \, \left(\sum_{l=1}^s\frac{\zeta'(2\sigma)}{\zeta(2\sigma)}\right)^{n-m+j} \,
\prod_{l=1}^{m-j} \left( -\left(\frac{\zeta'(2\sigma)}{\zeta(2\sigma)} \right)^2 + 
\frac{\zeta''(2\sigma)}{\zeta(2\sigma)} \right).
\eeas
Note that $\zeta(s)$ only has a pole at $s=1$ of order 1 and the residue 1. So when $\sigma \rightarrow 1/2$, then
\bea
&&\prod_{i=1}^{m}\frac{\partial}{\partial \alpha_{i}}\prod_{j=1}^{n}\frac{\partial}{\partial \beta_{j}}
\prod_{i,j=1}^s \zeta(2\sigma + \alpha_i+\beta_j)\Bigg|_{\substack{ \alpha_{1}=\ldots=\alpha_{s}=0 \\ \beta_{1}=\ldots=\beta_{s}=0}} \nonumber \\
&&\sim \sum_{j=0}^{m}\binom{m}{j} n^j \frac{n!}{j!} s^{n-m+j} \frac{(-1)^{n-m}}{(2\sigma-1)^{ns+n+m}}.
\label{1002eq2}
\eea
By the holomorphic property of $H(\alpha_{1},\ldots,\alpha_{s},\beta_{1},\ldots,\beta_{s})$ in the neighborhood of $(0,\ldots,0)$, combining with the Cauchy integral formula in several complex variables, we conclude that any order of the partial derivatives of $H(\alpha_{1},\ldots,\alpha_{s},\beta_{1},\ldots,\beta_{s})$ with respect to $\alpha_{1},\ldots,\alpha_{s}$ and $\beta_{1},\ldots,\beta_{s}$ at $(0,\ldots,0)$ is a constant as $\sigma\rightarrow 1/2$.  
Hence, by \eqref{1002eq2}, in \eqref{1002eq1} the contribution for the leading order and the leading coefficient as $\sigma\rightarrow 1/2$ comes from
\beas
&&H(\alpha_1,\ldots,\alpha_s, \beta_1,\ldots, \beta_s)\Bigg|_{\substack{ \alpha_{1}=\ldots=\alpha_{s}=0 \\ \beta_{1}=\ldots=\beta_{s}=0}}\prod_{i=1}^{s}\frac{\partial}{\partial \alpha_{i}}\prod_{j=1}^{s}\frac{\partial}{\partial \beta_{j}}
\prod_{i,j=1}^s \zeta(2\sigma + \alpha_i+\beta_j)\Bigg|_{\substack{ \alpha_{1}=\ldots=\alpha_{s}=0 \\ \beta_{1}=\ldots=\beta_{s}=0}}.
\eeas
Note that $\lim_{\sigma\rightarrow \frac{1}{2}}H(0,\ldots,0)=a_{s}$ by (\ref{December17}), the arithmetic factor given as (\ref{definitionofthearithmeticfactor}). By \eqref{1002eq2},
\beas
&&\prod_{i=1}^{s}\frac{\partial}{\partial \alpha_{i}}\prod_{j=1}^{s}\frac{\partial}{\partial \beta_{j}}
\prod_{i,j=1}^s \zeta(2\sigma + \alpha_i+\beta_j)\Bigg|_{\substack{ \alpha_{1}=\ldots=\alpha_{s}=0 \\ \beta_{1}=\ldots=\beta_{s}=0}}\sim \sum_{j=0}^{s}\binom{s}{j}^2(s-j)!s^{2j} \frac{1}{(2\sigma-1)^{s^2+2s}},
\eeas
note that $\sum_{j=0}^{s}\binom{s}{j}^2(s-j)!s^{2j}=s!L_{s}(-s^2)$. This completes the proof of (\ref{June14formula2}).
\end{proof}

\section{Integer moments: finite matrix size properties}
\label{se:exact}
In this section, we shall obtain exact formulae for $\E\left[ |\Lambda_N'(z)|^{2s}  \right]$ as given by Theorem \ref{th:exact-intro} and prove Theorem \ref{structuretheorem}. Then in Section \ref{se:mainpfs} we apply Theorem \ref{th:exact-intro} to establish Theorems \ref{th:meso} and \ref{th:micro-intro}. We start with the identity for any $z \in \mathbb{C}$,
\begin{equation}
\E\left[ |\Lambda_N'(z)|^{2s}  \right] = \prod_{j=1}^{s}\frac{\partial}{\partial z_{j}}\frac{\partial}{\partial w_{j}}\mathbb{E}\left[\prod_{j=1}^{s}\det(I-z_{j}U)\det(I-w_{j}U^{\dagger})\right]\Bigg{|}_{\bm{z}=\bm{w}=|z|} \label{deriv-merge}
\end{equation}
where the notation refers to evaluating the vectors $\bm{z} = (z_{1},\ldots,z_{s})$ and $\bm{w} = (w_{1},\ldots,w_{s})$ at the point $w_{j}=z_{j}=|z|$ for all $j=1,\ldots,s$.  Here, we used the radial symmetry of the moments. The expectation in \eqref{deriv-merge} can be evaluated using the following result from \cite[Theorem 1]{akemann2003characteristic}.
\begin{lemma}\label{an expression for correlation}
For any given $\bm{z}, \bm{w} \in \mathbb{C}^{s}$, we have
\begin{equation}
\label{correlation}
\mathbb{E}\left[\prod_{j=1}^{s}\det(I-z_{j}U)\det(I-w_{j}U^{\dagger})\right]
=
\frac{\det\bigg\{K_{N}(z_i w_j)\bigg\}_{i,j=1}^{s}}{\Delta(\bm z)\Delta(\bm w)},
\end{equation}
where $K_{N}(r) = \sum_{\ell=0}^{N+s-1} r^\ell$ and
\begin{equation}
\label{vand}
    \Delta(\bm z) = \prod_{1 \leq j < k \leq s}(z_{k}-z_{j}) = \mathrm{det}\bigg\{z_{i}^{j-1}\bigg\}_{i,j=1}^{s}
\end{equation}
is the Vandermonde determinant.
\end{lemma}

\subsection{Low order moments}
For some clarity, consider $s=1$. We write
\begin{equation}
\begin{split}
\mathbb{E}(|\Lambda_{N}'(z)|^{2}) = \frac{\partial}{\partial z_1}\frac{\partial}{\partial w_{1}}\mathbb{E}\left[\det(I-z_{1}U)\det(I-w_{1}U^{\dagger})\right]\bigg{|}_{z_{1}=w_{1}=|z|}. \label{hs1}
\end{split}
\end{equation}
By Lemma \ref{an expression for correlation} this is
\begin{equation}
\mathbb{E}(|\Lambda_{N}'(z)|^{2}) = \sum_{j=1}^{N}j^{2}|z|^{2(j-1)}. \label{1stmom}
\end{equation}
Then for every fixed $z$ with $0 \leq |z|<1$ we have the limit
\begin{equation}
\lim_{N \to \infty}\mathbb{E}(|\Lambda_{N}'(z)|^{2}) = \sum_{j=1}^{\infty}j^{2}|z|^{2(j-1)} = \frac{|z|^{2}+1}{(1-|z|^{2})^{3}}. \label{interiors1}
\end{equation}
Note this is consistent with \eqref{derivmoms2}. We briefly point out that one can easily compute asymptotics using \eqref{1stmom}. Setting $|z|=1$, we obtain
\begin{equation}
\mathbb{E}(|\Lambda_{N}'(1)|^{2}) = \sum_{j=1}^{N}j^{2} = \frac{N(N+1)(2N+1)}{6} \sim \frac{N^{3}}{3}. \label{unitcircles1}
\end{equation}
In the microscopic scaling it is simple to show from \eqref{1stmom}, e.g.\ using a Riemann sum approximation, that if $|z|^{2}=1-\frac{c}{N}$, then
\begin{equation}
\mathbb{E}(|\Lambda_{N}'(z)|^{2}) \sim N^{3}\int_{0}^{1}x^{2}e^{-cx}dx.
\end{equation}
Note that when $c=0$ this recovers \eqref{unitcircles1}.
\subsection{General case}
We now extend the previous example to the case of a general positive integer $s$. From \eqref{deriv-merge} and Lemma \ref{an expression for correlation} we obtain
\begin{equation}
\E\left[ |\Lambda_N'(z)|^{2s}  \right] = \prod_{j=1}^{s}\frac{\partial}{\partial z_{j}}\frac{\partial}{\partial w_{j}}\frac{\det\bigg\{\sum_{\ell=0}^{N+s-1} (z_i w_j)^{\ell}\bigg\}_{i,j=1}^{s}}{\Delta(\bm z)
\Delta(\bm w)}\bigg{|}_{\bm{z}=\bm{w}=z} \label{genform}
\end{equation}
The challenging issue is to correctly deal with the denominator in \eqref{genform} in the merged limit $\bm{z}=\bm{w}=|z|$. To do that, we need some preparations.
\begin{lemma}\label{factortheorem}
Let $n\geq 1$ be an integer. Let $f(x_1,\ldots,x_n)$ be a multivariate anti-symmetric polynomial, that is, for any permutation
$\sigma$ of $\{1,2,\ldots,n\}$,
\bea
f(x_{\sigma(1)},\ldots,x_{\sigma(n)})
={\rm sign}(\sigma)
f(x_1,\ldots,x_n).
\label{0729}
\eea
Then
\[
\frac{f(x_{1},\ldots,x_{n})}{\prod_{i=1}^{n-1}(x_n-x_i)}
\]
is a polynomial of $x_1,\ldots,x_{n}$.
\end{lemma}

\begin{proof}
By the factor theorem, for any two polynomials  $f(x_1,\ldots,x_n), g(x_2,\ldots,x_n)$, we have that $x_1-g$ is a factor of $f$ if and only if $f(g,x_2,\ldots,x_n)$ is a zero polynomial. For us, if we set $\sigma$ as a permutation of $1$ and $n$, then \eqref{0729} indicates that $f(x_n,x_2,\ldots,x_n)$ is a zero polynomial. Thus $x_1-x_n$ is a factor of $f$. By a similar argument, we have that $x_i-x_n$ is a factor of $\frac{f(x_{1},\ldots,x_{n})}{\prod_{j=1}^{i-1}(x_{j}-x_{n})}$ for any $i=2,\ldots,n$.
\end{proof}

\begin{lemma}
\label{le:ordered-expan}
Let $n\geq 1$ be an integer. Let
\bea\label{defofPh}
P_n=\{(q_1,\ldots,q_n): q_1+\ldots+q_n=\frac{n(n+1)}{2},q_1>\cdots>q_n\geq0 \}.
\eea
Let $\omega_{n}$ be the function from $P_{n}$ to the positive integers defined inductively by
\beas
\omega_{n}\Big((n,\ldots,2,1)\Big)=1
\eeas
and for $(q_{1},q_{2},\ldots,q_{n})\neq (n,\ldots,2,1)$(implying $q_{n}=0$), 
\bea\label{definition of omegas}
\omega_{n}\Big((q_{1},\ldots,q_{n-1},0)\Big)=\sum_{\substack{j=1\\q_{j}-q_{j+1}\geq 2}}^{n-1}\omega_{n-1}\Big((q_{1}-1,\ldots,q_{j}-2,q_{j+1}-1,\ldots,q_{n-1}-1)\Big)
\eea
and $\omega_{1}\big((1)\big)=1$. Here we set the default value as $P_{0}=\{\emptyset\}$ and $\omega_{0}(\emptyset)=1$.
Let $f(x_1,\ldots,x_n)$ be a multivariate anti-symmetric polynomial. Then
\bea
&& 
\prod_{i=1}^{n} \frac{\partial}{\partial x_i}
\frac{f(x_{1},\ldots,x_{n})}{\prod_{1\leq i<j\leq n} (x_i-x_j)} \Big|_{x_1=\cdots=x_n=x}\nonumber\\
&=&
\sum_{ (q_1,\ldots,q_{n})\in P_{n} }
\frac{\omega_{n}(q_1,\ldots,q_{n}) }{\prod_{i=1}^n q_i!} 
\prod_{i=1}^{n} \frac{\partial^{q_i}}{\partial x_i^{q_i}}
f(x_1,\ldots,x_n)
\Big|_{\bm x=x}\label{0729split2}.
\eea
\end{lemma}
\begin{proof}
It is easy to see that (\ref{0729split2}) holds for $n=1$. 
Assume inductively that $(\ref{0729split2})$ holds for $n-1$ for any $n\geq 2$. We now prove the case $n$. For fixed $x_1, \ldots,x_{n-1}$, we treat $f$ (short for $f(x_{1},\ldots,x_{n})$) as a function of $x_n$ and Taylor expand at $x_n=x$. Introduce
\[
g :=\frac{\partial}{\partial x_n} \frac{f}{\prod_{i=1}^{n-1}(x_{i}-x_n)}\Big|_{x_n=x}.
\]
By the inductive assumption and the fact that $g$ is a multivariate polynomial by Lemma \ref{factortheorem}, we obtain that the left-hand side of (\ref{0729split2}) equals the following formula. 
\bea\label{computationaboutg}
&&\prod_{i=1}^{n-1}\frac{\partial}{\partial x_i} \frac{g}{\prod_{1\leq i<j\leq n-1}(x_i-x_j)}\Big|_{\bm x=x}\nonumber\\
&=&\sum_{ (p_1,\ldots,p_{n-1})\in P_{n-1} }
\frac{\omega_{n-1}(p_1,\ldots,p_{n-1}) }{\prod_{i=1}^{n-1} p_i!} 
\prod_{i=1}^{n-1} \frac{\partial^{p_i}}{\partial x_i^{p_i}}
g
\Big|_{\bm x=x}. \label{partialsg}
\eea
Now by definition of $g$
\bea
g = f\Big|_{x_n=x}
\left(\sum_{j=1}^{n-1}\frac{1}{(x_{j}-x)^2}\frac{1}{\prod_{i\neq j}(x_{i}-x)}\right) +\frac{\partial f}{\partial x_n}\Big|_{x_n=x}\frac{1}{\prod_{i=1}^{n}(x_{i}-x)}.
\label{07291012eq2}
\eea
Comparing the Taylor coefficients on both sides of \eqref{07291012eq2} gives the derivatives of $g$ in terms of $f$:
\begin{align}
&\frac{1}{\prod_{i=1}^{n-1}p_{i}!}\prod_{i=1}^{n-1}\frac{\partial^{p_i}}{\partial x_i^{p_i}}g\bigg{|}_{\bm x = x} = \sum_{j=1}^{n-1} \left(\prod_{i=1, i\neq j}^{n-1}\frac{1}{(p_{i}+1)!}\frac{\partial^{p_i+1}}{\partial x_i^{p_i+1}}\right)\frac{1}{(p_{j}+2)!}\frac{\partial^{p_j+2}}{\partial x_j^{p_j+2}}f\bigg{|}_{\bm x = x} \label{line1derivtaylor}\\
&+\prod_{i=1}^{n-1}\frac{1}{(p_{i}+1)!}\frac{\partial^{p_i+1}}{\partial x_i^{p_i+1}}\frac{\partial}{\partial x_{n}}f\bigg{|}_{\bm x = x}. \label{line2derivtaylor}
\end{align}
As $(p_{1},\ldots,p_{n-1})$ vary in $P_{n-1}$ then the orders of the derivatives on the right-hand side in \eqref{line1derivtaylor} and \eqref{line2derivtaylor} vary in the set $\mathcal{D}_{n} = \mathcal{D}_{n,0}\cup\mathcal{D}_{n,1}$ where 
\begin{equation}
\begin{split}
\mathcal{D}_{n,0} &= \cup_{j=1}^{n-1}\{(p_{1}+1,\ldots,p_{j}+2,\ldots,p_{n-1}+1,0) \mid (p_{1},\ldots,p_{n-1}) \in P_{n-1}\}\\
\mathcal{D}_{n,1} &= \{(p_{1}+1,\ldots,p_{n-1}+1,1) \mid (p_{1},\ldots,p_{n-1}) \in P_{n-1}\}.
\end{split}
\end{equation}
Because of the anti-symmetry of $f$, the non-zero contributions to \eqref{line1derivtaylor} and \eqref{line2derivtaylor} will only be a subset of $\mathcal{D}_{n}$ that we denote here $\tilde{\mathcal{D}}_{n}$. We now show that $\tilde{\mathcal{D}}_{n} = P_{n}$. From the inductive hypothesis, the elements of $\tilde{\mathcal{D}}_{n}$ partition the integer $n(n+1)/2$. When we consider $(q_{1},\ldots,q_{n})\in \mathcal{D}_{n}$, there are vectors whose two components are equal with the other vectors preserving the strictly decreasing order of the components. By anti-symmetry the terms with duplicate components do not contribute, so we just collect the latter vectors in $\tilde{\mathcal{D}}_{n}$. So we have $\tilde{\mathcal{D}}_{n}\subset P_{n}$. To see the converse, let $(q_{1},\ldots,q_{n})\in P_{n}$, if $q_{n}=1$, then $(n,n-1,\ldots,1)\in \mathcal{D}_{n,1}\cap\tilde{\mathcal{D}}_{n}$ and if $q_{n}=0$, there is some $j$ with $1\leq j\leq n-1$ such that $q_{j}-q_{j+1}\geq 2$, so $(q_{1},\ldots,q_{n})\in \mathcal{D}_{n,0}\cap\tilde{\mathcal{D}}_{n}$ and we conclude that $\tilde{\mathcal{D}}_{n} = P_{n}$. 

According to the above analysis and (\ref{computationaboutg}), we have 
\be
\prod_{i=1}^{n}\frac{\partial}{\partial x_i} \frac{f}{\prod_{1\leq i<j\leq n}(x_i-x_j)}\Big|_{\bm x=x} = \sum_{ (q_1,\ldots,q_{n})\in P_{n} }
\frac{c(q_1,\ldots,q_n)}{\prod_{i=1}^{n} q_i!} 
\prod_{i=1}^{n} \frac{\partial^{q_i}}{\partial x_i^{q_i}}
f
\Big|_{\bm x=x},
\ee
for some coefficients $c(q_1,\ldots,q_n)$. We can obtain the coefficients by comparing with \eqref{line1derivtaylor} and \eqref{line2derivtaylor} as follows. Suppose first that $q_{n}=1$. Then the only possible element of $P_{n}$ is $(q_{1},\ldots,q_{n})=(n,n-1,\ldots,1)$, corresponding exactly to $(p_{1},\ldots,p_{n-1})=(n-1,\ldots,1)$. So $c(q_1,\ldots,q_n)=\omega_{n-1}((n-1,n-2,\ldots,1))=1$. If $q_{n}\neq 1$, then $q_{n}=0$ since $(q_{1},\ldots,q_{n})\in P_{n}$. For such fixed $(q_{1},\ldots,q_{n-1},q_{n})$, by our relation between $\tilde{\mathcal{D}}_{n}$ and $P_{n}$, we have $(q_{1},\ldots,q_{n}) = (p_{1}+1,\ldots,p_{j}+2,\ldots,p_{n-1}+1,0)$ for some $j=1,\ldots,n-1$. Solving this for $p_{1},\ldots,p_{n-1}$ and inserting it into \eqref{partialsg} gives
\be
c(q_1,\ldots,q_{n})=\sum_{\substack{j=1\\q_{j}-q_{j+1}\geq 2}}^{n-1}\omega_{n-1}\Big((q_{1}-1,\ldots,q_{j}-2,q_{j+1}-1,\ldots,q_{n-1}-1)\Big),
\ee
which is the recursion given in (\ref{definition of omegas}).
\end{proof}
Next, we present a general form of Lemma \ref{le:ordered-expan} and express it in terms of quantities involved in Young diagrams.
\begin{lemma}\label{le:young}
With the same notation as Lemma \ref{le:ordered-expan}, we have for any $m=0,1,\ldots,n$,
\begin{equation}\label{one expression}
\begin{split}
&\prod_{i=1}^{m} \frac{\partial}{\partial x_i}
\frac{f(x_{1},\ldots,x_{n})}{\prod_{1\leq i<j\leq n} (x_i-x_j)} \Big|_{_{x_1=\cdots=x_n=x}}\\
&=\sum_{\lambda \in Y_{m}}
    \frac{f_\lambda}{\prod_{i=1}^{n}(\lambda_{i}+n-i)!}\prod_{i=1}^{n} \frac{\partial^{\lambda_i+n-i}}{\partial x_i^{\lambda_i+n-i}}
f(x_1,\ldots,x_n)\Big|_{x_1=\cdots=x_n=x}.
\end{split}
\end{equation}
Here we set $\lambda_{m+1}=\cdots=\lambda_{n}=0$ when $m\leq n-1$.
\end{lemma}
\begin{proof}
Using the Taylor expansion of $f$ at $(x_{1},\ldots,x_{m},x\ldots,x)$, we have 
\beas
\frac{f(\x)}{\Delta(\x)} \Big|_{x_{m+1}=\cdots=x_n=x}
=\prod_{i=m+1}^{n} \frac{1}{(n-i)!} 
\frac{g(\tilde{\x})}{\Delta(\tilde{\x})},
\eeas
where $\tilde{\x}=(x_{1},\ldots,x_{m})$, and
\beas
g(\tilde{\x}) =
\frac{\prod_{i=m+1}^{n}
\big(\frac{\partial}{\partial x_i}\big)^{n-i} f(\x)\Big|_{x_{m+1}=\cdots=x_n=x}} {\prod_{1\leq i\leq m}(x_{i}-x)^{n-m}}.
\eeas
Then $g(\tilde{\x})$ is a polynomial. In fact, we just need to prove 
\bea\label{901formula1}
\prod_{i=1}^{m}(\frac{\partial}{\partial x_i})^{a_{i}}\prod_{i=m+1}^{n}\big(\frac{\partial}{\partial x_i}\big)^{n-i}f(\x)\Big|_{\x=x}
\eea
is zero for any $(a_{1},\ldots,a_{m})$ with at least one $a_{i}$ less than $n-m$. Without loss of generality, we suppose that $a_{1}\leq n-m-1$, then 
\beas
&&(\ref{901formula1})=(\frac{\partial }{\partial x_{1}})^{a_{1}}(\frac{\partial}{\partial x_{n-a_{1}}})^{a_{1}}\Bigg(\prod_{i=2}^{m}(\frac{\partial}{\partial x_i})^{a_{i}}\prod_{\substack{i=m+1\\i\neq n-a_{1}}}^{n}\big(\frac{\partial}{\partial x_i}\big)^{n-i} f(\x)\Bigg)\Bigg|_{\x=x}.
\eeas
It is not hard to check that the function inside the bracket satisfies (\ref{0729}) for the permutation $x_{1}\leftrightarrow x_{n-a_{1}}$. So $(\ref{901formula1})=0$.
It is easy to see that $g(\tilde{\x})$ satisfies (\ref{0729}). We apply Lemma \ref{le:ordered-expan} to $g(\tilde{\x})$ to obtain
\beas
\prod_{i=1}^{m} \frac{\partial}{\partial x_i}
\frac{f(\x)}{\Delta(\x)} \Big|_{\x=x}
&=&\prod_{i=m+1}^{n}\frac{1}{(n-i)!}\sum_{(q_1,\ldots,q_m)\in P_m}
\frac{\omega_{m}(q_1,\ldots,p_{m})}{\prod_{i=1}^{m}q_{i}!}
\times \,
\prod_{i=1}^{m} \frac{\partial^{q_i}}{\partial x_i^{q_i}}
g(\tilde{\x})
\Big|_{\tilde{\x}=x}
\eeas
Combining above and the Taylor expansion of $g(\tilde{\x})$ at $\tilde{\x}=x$, 
we have
\bea\label{901formula3}
\prod_{i=1}^{m} \frac{\partial}{\partial x_i}
\frac{f(\x)}{\Delta(\x)} \Big|_{\x=x}
&=&\sum_{(q_1,\ldots,q_m)\in P_m}
\frac{\omega_{m}(q_1,\ldots,q_{m})}{\prod_{i=1}^{n}(q_{i}+n-m)!}
\prod_{i=1}^{n} \frac{\partial^{q_i+n-m}}{\partial x_i^{q_i+n-m}}
f(\x)
\Big|_{\x=x},
\eea
where $q_{i}=m-i$ for $i=m+1,\ldots,n$.
Consider the sum over $(q_{1},\ldots,q_{m}) \in P_{m}$ above  and to each term define $\{\lambda_{j}\}_{j=1}^{m}$ by
\bea\label{901formula4}
q_{j} = \lambda_{j}+m-j, \quad j=1,\ldots,m.
\eea
Then the strictly decreasing property of $\{q_{j}\}_{j=1}^{m}$ implies that $\{\lambda_{j}\}_{j=1}^{m}$ are a weakly decreasing sequence of non-negative integers. The condition $\sum_{j=1}^{m}q_{j} = m(m+1)/2$ becomes $\sum_{j=1}^{m}\lambda_{j}=m$.
Let
\beas
Q_m &=& \{(\lambda_1,\lambda_2,\ldots,\lambda_m): \lambda_1+\cdots+\lambda_m =m, \lambda_1\geq \cdots \geq \lambda_m \geq 0\}.
\eeas
So the weight function $\omega_m: P_m \rightarrow \{1,2,\ldots\}$ is correspondingly defined in $Q_m$. Then the recursion relation on $Q_m$ is
\[
\omega_m((1,1,\ldots,1))=1,
\]
and
\be
\omega_m((\lambda_1,\ldots,\lambda_{m-1},0))
=\sum_{\substack{j=1\\\lambda_{j} - \lambda_{j+1}\geq 1}}^{m-1} \omega_{m-1}((\lambda_1,\ldots,\lambda_{j}-1,\lambda_{j+1},\ldots,\lambda_{m-1})). \label{rr-explain}
\ee
Let $c$ be a corner of a Young diagram, that is a box with no neighbours below or to the right. Let $\lambda \setminus c$ be the Young diagram obtained from $\lambda$ by erasing $c$. Then \eqref{rr-explain} states that
\begin{equation}
\omega_{m}(\lambda) = \sum_{\text{c is a corner of $\lambda$}}\omega_{m-1}(\lambda \setminus c).
\end{equation}
This is precisely the recurrence relation obeyed by $f_{\lambda}$, the number of standard Young tableaux of type $\lambda$. We conclude that 
\be
\omega_{m}(q_{1},\ldots,q_{m}) = f_{\lambda}.
\label{0819eq1}
\ee
Plugging the above and (\ref{901formula4}) into (\ref{901formula3}), we obtain the claim in this lemma. 
This completes the proof.
\end{proof}

Now we complete the proof of Theorem \ref{th:exact-intro}.
\begin{theorem}[Restatement of Theorem \ref{th:exact-intro}]
\label{th:exact}
We have the following exact formula, valid for any $z\in \mathbb{C}$ and any positive integers $N,s$,
\begin{equation}
\E[|\Lambda_{N}'(z)|^{2s}] = \sum_{\lambda, \mu \in Y_{s}}\frac{f_\lambda f_\mu}{\lambda! \mu!}
\det\bigg\{(u^{\lambda_i+s-i} K_{N}^{(\lambda_i+s-i)}(u))^{(\mu_j+s-j)}\bigg\}_{i,j=1}^{s} \label{exact}
\end{equation}
where $u=|z|^{2}$ and
\begin{equation}\label{kndef}
K_{N}(u) = \sum_{j=0}^{N+s-1}u^{j}.
\end{equation}
The bracketed superscripts on the right-hand side in \eqref{e2sintro} denote repeated differentiations $f^{(n)}(u) = \frac{d^{n}f}{du^{n}}$.
\end{theorem}
\begin{proof}
We have
\begin{equation}
\E\left[ |\Lambda_N'(z)|^{2s}  \right] = \prod_{j=1}^{s}\frac{\partial}{\partial w_{j}}\frac{\partial}{\partial z_{j}}\frac{\det\bigg\{K_{N}(z_i w_j)\bigg\}_{i,j=1}^{s}}{\prod_{1\leq i<j \leq s}(z_j-z_i) 
\prod_{1\leq i<j \leq s}(w_{j}-w_{i})}\bigg{|}_{\bm{z}=\bm{w}=z} \label{genform2}.
\end{equation}
First note that the polynomial $\det\bigg\{K_{N}(z_i w_j)\bigg\}_{i,j=1}^{s}$ satisfies the anti-symmetry properties required in the hypothesis of Lemmas \ref{le:ordered-expan} and \ref{le:young}. Hence applying Lemma \ref{le:young} to  expression \eqref{genform2}, first in the variables $\bm z$ and once more in the variables $\bm w$ gives
\begin{equation}
\E[|\Lambda_{N}'(z)|^{2s}] = \sum_{\lambda, \mu \in Y_{s}}\frac{f_\lambda f_\mu}{\lambda! \mu!}\prod_{j=1}^{s}\frac{\partial^{\mu_{j}+s-j}}{\partial w_{j}^{\mu_{j}+s-j}}\frac{\partial^{\lambda_{j}+s-j}}{\partial z_{j}^{\lambda_{j}+s-j}}\det\bigg\{K_{N}(z_i w_j)\bigg\}_{i,j=1}^{s}\bigg{|}_{\bm z=\bm w = z} \label{2smom}
\end{equation}
The derivatives here are now straightforward to compute. By writing the determinant as a sum over permutations and performing the derivatives, we get
\begin{equation}
\begin{split}
&\prod_{j=1}^{s}\frac{\partial^{\mu_{j}+s-j}}{\partial w_{j}^{\mu_{j}+s-j}}\frac{\partial^{\lambda_{j}+s-j}}{\partial z_{j}^{\lambda_{j}+s-j}}\det\bigg\{K_{N}(z_i w_j)\bigg\}_{i,j=1}^{s}\bigg{|}_{\bm z=\bm w = z}\\
&= \det\bigg\{(u^{\lambda_i+s-i} K_{N}^{(\lambda_i+s-i)}(u))^{(\mu_j+s-j)}\bigg\}_{i,j=1}^{s}.
\end{split}
\end{equation}
Inserting this into \eqref{2smom} completes the proof.
\end{proof}

\subsection{Proof of Theorem \ref{structuretheorem}}\label{proof the second main result}
In the remaining part of this section, we shall prove Theorem \ref{structuretheorem}. Before it, we need some preparations.
We start from an alternative expression for the  left-hand side of (\ref{correlation}) in Lemma \ref{an expression for correlation}.

\begin{lemma} \label{alternative expression2}
Let $s\geq 1$ be an integer. Suppose $\bm{z}, \bm{w} \in \mathbb{C}^{s}$. Let $F(\bm z, \bm w)$ and $G(\bm z, \bm w)$ be given in \eqref{definitionnewadded} and \eqref{defofG}, respectively. 
\begin{equation}
\label{901correlation}
\mathbb{E}\left[\prod_{j=1}^{s}\det(I-z_{j}U)\det(I-w_{j}U^{\dagger})\right]
=F(-\w,-\z)G(-\w,-\z).
\end{equation}
\end{lemma}
\begin{proof}
It follows from formula (2.11) of \cite{conrey2003autocorrelation} that
\be
\E \Bigg[ \prod_{i=1}^s \prod_{j=1}^N (1+z_i^{-1} e^{-\i\theta_j}) (1 + w_i e^{\i\theta_j}) \Bigg]
= \frac{S_{\langle N^s\rangle} (\z,\w)}{\prod_{i=1}^s z_i^N},
\label{0901-eq1}
\ee
where $e^{\i \theta_{j}}$ are eigenvalues of a Haar-unitary random matrix, and the Schur polynomial
\beas
&& S_{\langle N^s\rangle} (\alpha_1,\ldots,\alpha_s, \alpha_{s+1},\ldots,\alpha_{2s}) \\
&=& \frac{1}{\prod_{1\leq i<j\leq 2s} (\alpha_i-\alpha_j)} \det
\begin{pmatrix}
\alpha_1^{N+2s-1}  & \cdots & \alpha_1^{N+s} & \alpha_1^{s-1} & \cdots & 1 \\
\vdots & \vdots & \vdots & \vdots & \vdots & \vdots\\
\alpha_{2s}^{N+2s-1} & \cdots & \alpha_{2s}^{N+s} & \alpha_{2s}^{s-1} & \cdots & 1 \\
\end{pmatrix}.
\eeas
From the above
and the decomposition $\Delta({\bm{z}},{\bm{w}})=\Delta({\bm{z}})\Delta({\bm{w}})\prod_{i=1}^{s}\prod_{j=1}^{s}(z_{i} -w_{j})$, we have $\eqref{0901-eq1}=F(\z^{-1},\w) G(\z^{-1},\w)$. Hence we have \eqref{901correlation}.
\end{proof}

\begin{lemma}\label{the expression for the expectation of characteristic polynomials}
Let $s\geq 1$ be an integer, and $z\in \mathbb{C}$ with $|z|\neq 1$.
 Let $F(\bm z, \bm w)$  and $G(\bm z, \bm w)$ be as in \eqref{definitionnewadded} and \eqref{defofG}, respectively.  We have
\bea
\label{structureformula3}
\E[|\Lambda_N'(z)|^{2s}]
&=& 2\sum_{\substack{h_{1},h_{2}\\0\leq h_{1}< h_{2}\leq s}} 
\binom{s}{h_1}\binom{s}{h_2}
\prod_{j=1}^{s-h_2} \frac{\partial}{\partial w_j}  \prod_{i=1}^{s-h_1} 
\frac{\partial}{\partial z_i} F(\z,\w)\Big|_{\substack{\z= \w=-|z|}}
\nonumber \\
&& \times \, \prod_{j=1}^{h_2}
 \frac{\partial}{\partial w_j}  
\prod_{i=1}^{h_1}\frac{\partial}{\partial z_i} G(\z,\w)
\Big|_{\substack{\z= \w=-|z| }} \nonumber \\
&+& \sum_{h_{2}=0}^{s}
\binom{s}{h_2}^2\prod_{j=1}^{s-h_2} \frac{\partial}{\partial w_j}  \prod_{i=1}^{s-h_2} \frac{\partial}{\partial z_i} F(\z,\w)\Big|_{\substack{\z= \w=-|z|}}\nonumber\\
&& \times \, \prod_{j=1}^{h_2}
 \frac{\partial}{\partial w_j}  
\prod_{i=1}^{h_2}\frac{\partial}{\partial z_i} G(\z,\w)
\Big|_{\substack{\z= \w=-|z|}}.
\eea
\end{lemma}

\begin{proof}
By the radial property of the moments we have $\E[|\Lambda_N'(z)|^{2s}]=\E[|\Lambda_N'(|z|)|^{2s}]$. Furthermore, by \eqref{deriv-merge} and Lemma \ref{alternative expression2}, we have
\[
\E\left[ |\Lambda_N'(z)|^{2s}  \right] = \prod_{j=1}^{s}\frac{\partial}{\partial z_{j}}\frac{\partial}{\partial w_{j}} F(\z,\w)G(\z,\w) \Big|_{\substack{\z= \w=-|z| }}.
\]
Notice that $F(\z,\w), G(\z,\w)$ are invariant under permutations of $\z$, and also invariant under permutations of $\w$. Moreover, $F(\z,\w)=F(\w,\z), G(\z,\w)=G(\w,\z)$. The claim now follows by applying these properties to the above equality.
\end{proof}

From (\ref{structureformula3}), to obtain an explicit formula or study the structure of $\E[|\Lambda_N'(z)|^{2s}]$, it is sufficient to study the partial derivatives of $F$ and $G$. For the partial derivative of $F$, we can express it as an analytic function of $s$ as follows.

\begin{lemma}\label{mix of F}
Let $z\in \mathbb{C}$ with $|z|\neq 1$. Let $s\geq 1, h_{1},h_{2}$ be integers with $0\leq h_{1}\leq h_{2}\leq s$. Then
\bea
&& \binom{s}{h_{1}}\binom{s}{h_{2}}\prod_{j=1}^{s-h_2} \frac{\partial}{\partial w_j} 
\prod_{i=1}^{s-h_1} \frac{\partial}{\partial z_i} F(\bm{z},\bm{w})\Big|_{\substack{\z=\w=-|z|}} \nonumber\\
&=&\frac{1}{h_{1}!h_{2}!}\frac{(-s|z|)^{h_{2}-h_{1}}}{(1-|z|^2)^{s^2+2s-h_1-h_2}} \frac{1}{\Gamma(h_{2}-h_{1}+1)}\frac{(\Gamma(s+1))^2}{\Gamma(s-h_{2}+1)}e^{-s^2 |z|^2} \nonumber\\
&& \times \, {}_1F_1(s+1-h_{1},h_{2}-h_{1}+1;s^{2} |z|^2)
\label{analytic extentension},
\eea
where ${}_1F_1(a,b;z)$ is the confluent hypergeometric function of the first kind.
Moreover, the expression in the right hand-side of \eqref{analytic extentension}  is analytic when $s\in \mathbb{C}$ and $\Re(s)> -1$.
\end{lemma}

\begin{proof}
A similar argument to that of Lemma \ref{th:derivmomsint} leads to the following result
\bea\label{structureofderivativeofF}
&&\prod_{j=1}^{s-h_2} \frac{\partial}{\partial w_j} 
\prod_{i=1}^{s-h_1} \frac{\partial}{\partial z_i} F(\bm{z},\bm{w})\Big|_{\substack{\z=\w=-|z|}} \nonumber\\
&=&\frac{(-s|z|)^{h_{2}-h_{1}}}{(1-|z|^2)^{s^2+2s-h_1-h_2}}\sum_{k=0}^{s-h_2} \binom{s-h_2}{k} \binom{s-h_1}{h_2-h_1+k} (s-k-h_2)!(s^2|z|^2)^k.
\eea
Note that the generalized Laguerre polynomial 
$
L_{s}^{(\alpha)}(x)=\frac{x^{-\alpha}e^{x}}{s!}\frac{d^{s}}{dx^{s}}(e^{-x}x^{s+\alpha}),
$
for $\alpha \in \mathbb{C}$.
So by Kummer's transformation to ${}_1F_1(a,b;z)$, for any non-negative integer $s$,
\bea
&&\frac{\Gamma(s+1-h_{1})}{\Gamma(h_{2}-h_{1}+1)}e^{-s^2|z|^2}{}_1F_1(s+1-h_{1},h_{2}-h_{1}+1;s^{2}|z|^2)\nonumber\\
&=&(s-h_{2})!L_{s-h_{2}}^{(h_{2}-h_{1})}(-s^2|z|^2)\nonumber\\
&=&\sum_{k=0}^{s-h_2} \binom{s-h_2}{k} \binom{s-h_1}{h_2-h_1+k} (s-k-h_2)!(s^2|z|^2)^k \nonumber,
\eea
Note that $\frac{1}{\Gamma(h_{2}-h_{1}+1)}{}_1F_1(s+1-h_{1},h_{2}-h_{1}+1;s^{2}|z|^2)$ is an analytic function of $s$ with respect to $h_{1},h_{2}\in \mathbb{N}$, so we have the claim stated in this lemma.
\end{proof}
By (\ref{structureofderivativeofF}), we see that the derivative of $F$ at some point is almost a polynomial of $|z|^2$ except a factor of $|z|^{h_{2}-h_{1}}$. Now we shall evaluate 
\bea\label{beforeevaluation}
\prod_{j=1}^{h_2} \frac{\partial}{\partial w_j}  \prod_{i=1}^{h_1} \frac{\partial}{\partial z_i} G(\z,\w)
\eea
at the point $\z=-|z|$ and $\w=-|z|$.
In addition, we note that there is a factor $\frac{1}{\Gamma(s-h_{2}+1)}$ on the right hand-side of (\ref{analytic extentension}), this factor is zero when $h_{2}\geq s+1$ for any non-negative integer $s$. This leads to that the summation over $h_{1}$ and $h_{2}$ from $0$ to $s$ in (\ref{structureformula3}) can be written as from $0$ and $\infty$. The advantage of doing this is reducing the extension of $\E[|\Lambda_N'(z)|^{2s}]$ (from integer to real $s> -1$) to that of (\ref{beforeevaluation})
for integers $h_{1},h_{2}\geq 0$.
As a corollary of Lemma \ref{le:young}, we have the following representation of  
(\ref{beforeevaluation}) in terms of partitions and partial derivatives of $A({\bm{z}},{\bm{w}})$. 
\begin{proposition}
\label{expressionformixedmomentsofcharacterisricpolynomials}
Let $s\geq 1$ be an integer. Let $h_{1},h_{2}$ be integers with $0\leq h_{1}\leq s$, $0\leq h_{2}\leq s$.  Let $A(\bm z, \bm w)$  and $G(\bm z, \bm w)$ be as in \eqref{defofA} and \eqref{defofG}, respectively.
Then
\beas
&& 
\prod_{j=1}^{h_2} \frac{\partial}{\partial w_j}  \prod_{i=1}^{h_1} \frac{\partial}{\partial z_i} G(\z,\w)\Big|_{\z=\w=-|z|}\\
&=&
\sum_{ \lambda \in Y_{h_1}, \mu \in Y_{h_2} }
\frac{f_{\lambda} f_{\mu}}{\prod_{i=1}^{s}(\lambda_{i}+s-i)!(\mu_{i}+s-i)!}
\prod_{j=1}^{s} \frac{\partial^{\mu_i+s-i}}{\partial {w_i^{\mu_i+s-i}}}
\frac{\partial^{\lambda_i+s-i}}{\partial {z_i^{\lambda_i+s-i}}}
A(\z,\w)
\Big|_{\z=\w=-|z|}.
\eeas
Here we set $\lambda_{h_{1}+1}=\cdots=\lambda_{s}=0$ and $\mu_{h_{2}+1}=\cdots=\mu_{s}=0$.
\end{proposition}

In the following, we shall study (\ref{beforeevaluation}) based on Proposition \ref{expressionformixedmomentsofcharacterisricpolynomials}. We start from $h_{1}=h_{2}=0$.

\begin{proposition}\label{proaboutb00}
Let $z\in \mathbb{C}^*$ with $|z|\neq 1$.  Let $s\geq 1$ and $N\geq 1$ be integers. Then
\bea\label{an expression for b00new}
G(\z,\w)
\Big|_{\z=\w=-|z|}=\sum_{m=0}^{\infty}
\sum_{l=0}^s 
d_{m,l,s}(0,0) |z|^{2Nl- s^2+s +2m},
\eea
where $d_{m,l,s}(0,0)$ is independent of $|z|$, which can be given in a combinatorial sum (see \eqref{defineofdbarml} in Appendix \ref{appendix: two combinatorial formulas} for an explicit formula).
\end{proposition}

\begin{proof}
By Proposition \ref{expressionformixedmomentsofcharacterisricpolynomials},
\bea\label{03091}
G(\z,\w)
\Big|_{\substack{\z=\w=-|z|}}
=\prod_{i=1}^{s-1}\frac{1}{(i!)^2}\prod_{j=1}^s \frac{\partial^{j-1}}{\partial w_j^{j-1}} 
\prod_{i=1}^s \frac{\partial^{i-1}}{\partial z_i^{i-1}} A(\bm{z},\bm{w})
\Big|_{\substack{\z=\w=-|z|}},
\eea
with $A(\bm{z},\bm{w})$ defined in (\ref{defofA}). Then it is not hard to see that there are two matrices $B,C$, whose entries are powers of $-r$, such that
\beas
\prod_{j=1}^s \frac{\partial^{j-1}}{\partial w_j^{j-1}} 
\prod_{i=1}^s \frac{\partial^{i-1}}{\partial z_i^{i-1}} A(\bm{z},\bm{w})
\Big|_{\substack{\z=\w=-|z|}}=
\det\begin{pmatrix}
B \\
C
\end{pmatrix}.
\eeas
By applying Laplacian expansion to the first $s$ rows, we have
\bes 
\det\begin{pmatrix}
B \\
C
\end{pmatrix}= \sum_{S \subseteq \binom{[2s]}{s}} \text{sign}(S) \det(B_S) \det(C_{\overline{S}}),
\ees
where $\binom{[2s]}{s}$ is the set of all $s$-subset of $\{1,\ldots,2s\}$ and $\text{sign}(S)\in \{\pm 1\}$ is determined by $S$. Here $B_S$ is the $s\times s$ submatrix of $B$ whose columns are indexed by $S$, and similarly $C_{\bar{S}}$ is the $s\times s$ submatrix of $C$ whose columns are indexed by $\overline{S}$, the complement of $S$. Then direct computation shows that, up to a factor of $(-|z|)^{lN+i_1+\cdots+i_k+j_1+\cdots+j_l-\frac{s(s-1)}{2}}$, $\det(B_{S})$ and $\det(C_{\overline{S}})$ are Vandermonde determinants. Putting the above together, we obtain the claimed result.
\end{proof}
\begin{remark}
{\rm{
In (\ref{an expression for b00new}), the summation over $m$ is from $0$ to $\infty$. Actually from the definition of $d_{m,l,s}(0,0)$, for integer $s\geq 0$, we know that this summation is restricted on $m$ with $s(s-1)/2\leq m\leq (3s-1)s/2$.
}}
\end{remark}

Proposition \ref{proaboutb00} gives an explicit formula for (\ref{beforeevaluation}) in terms of partition sums when $h_{1}=h_{2}=0$. 
For general $h_{1},h_{2}$,
the method used above does not work well since we do not have Vandermonde determinants. 
Instead, there will appear determinants that are closely related to the generalized Vandermonde determinant introduced in \cite{generalized}. 
Unfortunately, it is hard to obtain neat formulae for those determinants. To overcome this, we classify these determinants as coefficients of powers of $|z|$ in the polynomial expression of the partial derivatives of $A(\bm{z},\bm{w})$ with respect to $\bm{z}$ and $\bm{w}$ at $\bm{z}=\bm{w}=-|z|$. Then we use the property of the latter one to find its derivative with respect to $|z|$ at zero. This would give the corresponding coefficients of the powers of $|z|$.

\begin{proposition}\label{forgeneralh1h2}
Let $z\in \mathbb{C}^*$ with $|z|\neq 1$. Let $s\geq 1$ be an integer. Let $(h_1,h_2)\neq(0,0)$ be non-negative integers with $h_{1},h_{2}\leq s$. Then
\bea\label{an expression for bh1h2new}
&&\prod_{j=1}^{h_2}
 \frac{\partial}{\partial w_j}  
\prod_{i=1}^{h_1}\frac{\partial}{\partial z_i} G(\z,\w)
\Big|_{\substack{\z=\w=-|z|}}\label{an expression for bh1h2} \nonumber\\
&=&(-1)^{h_{1}+h_{2}}\sum_{m=0}^{\infty}\sum_{l=1}^s
d_{m,l,s}(h_{1},h_{2})|z|^{2Nl-(s^2-s+h_{1}+h_{2})+2m},
\eea
where $d_{m,l,s}(h_1,h_2)$ is independent of $|z|$, which can be given in a combinatorial sum (see \eqref{defofdml1new} in Appendix \ref{appendix: two combinatorial formulas} for an explicit formula).
\end{proposition}

From the explicit expression of $d_{m,l,s}(h_1,h_2)$, we have that the summation over $m$ in \eqref{an expression for bh1h2new} is between $(s^2-s+2)/2$ and $(3s^2-s)/2$ when $s\geq 1$ is an integer.

Now we are ready to prove Theorem \ref{structuretheorem}.
\begin{proof}[Proof of Theorem \ref{structuretheorem}]
By Lemma \ref{mix of F}, 
\[
\binom{s}{h_1}\binom{s}{h_2}\prod_{j=1}^{s-h_2} \frac{\partial}{\partial w_j}  \prod_{i=1}^{s-h_1} \frac{\partial}{\partial z_i} F(\z,\w)\Big|_{\substack{\z=\w=-|z|}} 
=\frac{(-s|z|)^{h_{2}-h_{1}}  }{(1-|z|^2)^{s^2+2s-h_{1}-h_{2}}}a_{h_{1},h_{2}}(|z|),
\]
where $a_{h_{1},h_{2}}(|z|)$ is defined in (\ref{defofa}). Note that the factor $1/\Gamma(s-h_{2}+1)=0$ in $a_{h_{1},h_{2}}(|z|)$ when $h_{2}\geq s+1$ and $s$ is a non-negative integer. So we can rewrite (\ref{structureformula3}) as (\ref{structureexpression}) with $C_{h}(N,|z|)$ defined as (\ref{definitionofC}). Let $b_{h_{1},h_{2}}(N,|z|)$ be defined as
(\ref{defofb}). By Propositions \ref{proaboutb00} and \ref{forgeneralh1h2},
we have conclusion (iii) in this theorem. 

By (\ref{structureofderivativeofF}), for $(h_{1},h_{2})\neq (0,0)$,
$$
a_{h_{1},h_{2}}(|z|)b_{h_{1},h_{2}}(N,|z|)
=\binom{s}{h_1}\binom{s}{h_2}s^{h_{2}-h_{1}}\sum_{n=0}^{\infty}\sum_{l=1}^{s}\widetilde{d}_{n,l,s}(h_{1},h_{2})
|z|^{2Nl-s^2+s-2h_{1}+2n},
$$
where
\bes
\widetilde{d}_{n,l,s}(h_{1},h_{2}) = \sum_{k=0}^{\infty}d_{n-k,l,s}(h_{1},h_{2})s^{2k}\binom{s-h_2}{k} \binom{s-h_1}{h_2-h_1+k} (s-k-h_2)!
\ees
and 
$d_{m,l,s}(h_{1},h_{2})$ are coefficients in the expression of $b_{h_{1},h_{2}}(N,|z|)$.
So $C_{h}(N,|z|)$ is given by the following expression.
\bea
&&2\sum_{n=0}^{\infty}\sum_{l=1}^{s}\sum_{0\leq h_{1}<h/2}\binom{s}{h_1}\binom{s}{h-h_{1}}s^{h-2h_{1}}\widetilde{d}_{n,l,s}(h_{1},h-h_{1})
|z|^{2Nl- s^2+s -2h_{1}+2n}\nonumber\\
&&+\sum_{n=0}^{\infty}\sum_{l=1}^{s}\binom{s}{h/2}\binom{s}{h/2}\widetilde{d}_{n,l,s}(h/2,h/2)
|z|^{2Nl- s^2+s -h+2n}
\label{0819eq2}
\eea
Note that $\binom{m}{n}=0$ when $n>m$ and $d_{m,l,s}(h_{1},h_{2})$ are supported on non-negative integers $h_{1},h_{2}$ and finitely many $m$, so for $h\neq 0$, $C_{h}(N,|z|)$ is a polynomial in $|z|^2$ with coefficients polynomials of $N$. This conclusion also holds for $C_{0}(N,|z|)$ by a similar argument to the above. So we completes the proof of (i) in this theorem, except the part concerning the degree and the leading coefficient of $C_{h}(N,|z|)$ , which will be addressed in Proposition \ref{expressionofchnr} below. Since $a_{h_{1},h_{2}}(|z|)$ is independent of $N$, to prove (\ref{otherC}), it suffices to prove 
\bea\label{sufficeformula1}
\lim_{N\rightarrow \infty} G(\z,\w) \Big|_{\substack{\z=\w = - |z|}}=1.
\eea
and 
\bea\label{sufficeformula2}
\lim_{N\rightarrow \infty}\prod_{j=1}^{h_2}
 \frac{\partial}{\partial w_j}  
\prod_{i=1}^{h_1}\frac{\partial}{\partial z_i} G(\z,\w)
\Big|_{\substack{\z=\w=-|z| }} =0
\eea
for $(h_{1},h_{2})\neq (0,0)$.
For (\ref{sufficeformula1}), by Proposition \ref{expressionformixedmomentsofcharacterisricpolynomials},
\beas
 G(\z,\w) \Big|_{\substack{\z=\w = - |z| }} 
= \left(\prod_{i=0}^{s-1} \frac{1}{(i!)^2} \right)
\prod_{i=1}^s \frac{\partial^{i-1}}{\partial z_i^{i-1}}
\prod_{j=1}^s \frac{\partial^{j-1}}{\partial w_j^{j-1}} 
A(\z,\w) \Big|_{\substack{\z=\w = - |z|}}.
\eeas
We do Laplacian expansion on the corresponding matrix via the first $s$ columns and note that $|z|<1$, we obtain
\beas
&& \lim_{N\rightarrow \infty}G(\z,\w) \Big|_{\substack{\z=\w = - |z| }} \\ 
&=& \left(\prod_{i=0}^{s-1} \frac{1}{(i!)^2} \right)
\det( (z_i^{j-1})^{(i-1)} )_{i,j=1\ldots,s} \Big|_{ \z=-|z|  } 
\det( (w_i^{j-1})^{(i-1)} )_{i,j=1\ldots,s} \Big|_{ \w = -|z|  } \\
&=& 1.
\eeas
By Proposition \ref{expressionformixedmomentsofcharacterisricpolynomials} again, the left-hand side of (\ref{sufficeformula2}) equals
\[
\sum_{\substack{\lambda \in Y_{h_1} , \mu \in Y_{h_2} }} 
\frac{f_\lambda f_\mu}{\prod_{i=1}^{s}(\lambda_{i}+s-i)!(\mu_{i}+s-i)!} 
\prod_{i=1}^s  \frac{\partial^{u_j+s-j}}{\partial w_j^{u_j+s-j}} 
\frac{\partial^{\lambda_i+s-i}}{\partial z_i^{\lambda_i+s-i}}
\Delta(\z)\Delta(\w)\Big|_{\z=\w=-|z|}
\]
It is not hard to check that the above formula is zero for any $(h_{1},h_{2})\neq (0,0)$. This proves conclusion (ii) in this theorem.
\end{proof}

\begin{proposition}
Let $h_1,h_2,s$ be non-negative integers with $h_1,h_2\leq s$.
Let $\lambda, \mu$ be Young diagrams of weight $h_1,h_2$ respectively. Let $n_i=\lambda_i+s-i$ and $m_j=\mu_j+s-j$ for $i,j=1,\ldots,s$.
Let $C_{h}(N,|z|)$ be defined as (\ref{structureexpression}). Then the highest power of $|z|^2$ in $C_{h}(N,|z|)$ is $Ns+s^2+s-h$ and the leading coefficient is given by the following formula. 
\bea\label{leading coefficient}
&&2s^{2s-h}\sum_{0\leq h_{1}<h/2}\binom{s}{h_1}\binom{s}{h-h_{1}}d_{\frac{(3s-1)s}{2},s,s}(h_{1},h-h_{1})\nonumber\\
&& + \, s^{2s-h}\binom{s}{h/2}^2d_{\frac{(3s-1)s}{2},s,s}(h/2,h/2),
\label{expressionofchnr}
\eea
where 
\beas
d_{\frac{(3s-1)s}{2},s,s}(h_{1},h_{2})
&=&(-1)^{s}
\sum_{\substack{\lambda \in Y_{h_1}, \mu \in Y_{h_2} }} 
\frac{f_\lambda f_\mu}{\prod_{i=1}^{s}n_{i}!m_{i}!} \Delta(\n) \Delta(\m) \\
&& \times \, \prod_{j=1}^{s}\frac{\big((N+s+j-1)!\big)^2}{(N+2s-n_{j}-1)!(N+2s-m_{j}-1)!} 
\eeas
and $\n=(n_1,\ldots,n_s), \m=(m_1,\ldots,m_s)$.
\end{proposition}

\begin{proof}

By Propositions \ref{proaboutb00} and \ref{forgeneralh1h2}, we see that $d_{m,s,s}(h_{1},h_{2})\neq 0$ if and only if $m=\frac{(3s-1)s}{2}$. By \eqref{0819eq2}, the highest power of $|z|^2$ is $Ns+s^2+s-h$ by taking $n=s-h_{2}+\frac{(3s-1)s}{2}$, and the leading coefficient is (\ref{leading coefficient}).
Moreover, by Proposition  \ref{expressionformixedmomentsofcharacterisricpolynomials},
\bea\label{calculate thew leading term}
d_{\frac{(3s-1)s}{2},s,s}(h_{1},h_{2})
&=&(-1)^{s} \sum_{\substack{\lambda \in Y_{h_1} , \mu \in Y_{h_2} }} 
\frac{f_\lambda f_\mu}{\prod_{i=1}^{s}n_{i}!m_{i}!} 
 \det\Big(\frac{(N+s+j-1)!}{(N+s+j-1-n_{i})!}\Big)_{i,j=1,\ldots,s}\nonumber\\
&&\times\det\Big(\frac{(N+s+j-1)!}{(N+s+j-1-m_{i})!}\Big)_{i,j=1,\ldots,s}.
\eea
It is known that (see e.g., \cite[formulae (4.39)-(4.41)]{CRS06})
\beas
\det \Big(\frac{1}{(2s+1+l_{i}-i-j)!}\Big)_{i,j=1,\ldots,s}
=\prod_{i=1}^{s}\frac{1}{(2s-i+l_{i})!}\prod_{1\leq i<j\leq s}(l_{j}-l_{i}-j+i).
\eeas
Applying this to (\ref{calculate thew leading term}), we obtain the claim in this proposition.
\end{proof}

We remark that a similar argument to the above would give a similar explicit formula for the coefficient of $|z|^{2sN+2m}$ in $C_{h}(N,|z|)$ when $-\lfloor \frac{h}{2}\rfloor+s^2\leq m\leq s^2+s-h$. Moreover, these coefficients are polynomials of $N$ of degree $h$. For coefficients of other non-zero powers in  $C_{h}(N,|z|)$, they are polynomials of $N$ with degree at most $s^2-s+h$.

\section{Proof of the random matrix results: asymptotics}
\label{se:mainpfs}
In this section we complete the proofs of Theorems \ref{th:non-integer-intro}-\ref{th:meso} and Theorem \ref{th:micro-intro}.
\subsection{Global and mesoscopic regimes}
\label{se:globpfs}
As discussed in Section \ref{se:glob}, to prove convergence of expectations in Theorem \ref{th:non-integer-intro} we need to establish uniform integrability of the underlying random variables. For this we have the following bound valid for any positive integer $s$.
\begin{lemma}
\label{le:unif-bnd}
Let $0 \leq |z|<1$ be fixed and $s$ a positive integer. Then $\mathbb{E}(|\Lambda'_{N}(z)|^{2s})$ is bounded in $N$.
\end{lemma}

\begin{proof}
This follows from the exact result in \eqref{exact} and the uniform boundedness of the function $K_{N}(r)$ for fixed $0\leq r<1$.
\end{proof}
We now extend the bound to non-integer exponents.

\begin{lemma}[Uniform integrability]
\label{le:ui}
Assume that $h$ and $s$ are complex numbers such that $\mathrm{Re}(h)>-1$. Consider the quantity
\begin{equation}
X_{N} = \bigg{|}\frac{\Lambda'_{N}(z_2)}{\Lambda_{N}(z_2)}\bigg{|}^{2h}|\Lambda_{N}(z_1)|^{2s}.\label{XN}
\end{equation}
Then for any $\delta>0$ fixed, $|z_1| < 1-\delta$ and $|z_{2}|<1-\delta$ and for $\epsilon>0$ small enough and for $N$ sufficiently large, there exists a constant $C>0$ independent of $N$ and $z_{1},z_{2}$ such that
\begin{equation}
\mathbb{E}(|X_{N}|^{1+\epsilon}) < C. \label{uibnd}
\end{equation}
\end{lemma}

\begin{proof}
Since $|X_{N}|$ only depends on the real parts of $h$ and $s$, we will check \eqref{uibnd} for real $h>-1$ and $s \in \mathbb{R}$. If $h \geq 0$, by Jensen's inequality it suffices to check that \eqref{uibnd} is satisfied with $\epsilon=1$. By the Cauchy-Schwarz inequality 
\begin{equation}
\begin{split}
\mathbb{E}(X_{N}^{2}) &\leq \mathbb{E}(|\Lambda_{N}'(z_2)|^{8h}|)^{1/2}\mathbb{E}(|\Lambda_{N}(z_2)|^{-8h}|\Lambda_{N}(z_1)|^{8s})^{1/2}\\
& \leq \mathbb{E}(|\Lambda_{N}'(z_2)|^{8h}|)^{1/2}\mathbb{E}(|\Lambda_{N}(z_2)|^{-16h})^{1/4}\mathbb{E}(|\Lambda_{N}(z_1)|^{16s})^{1/4}.
\end{split}
\end{equation}
The second two expectations above, including the negative moments of $\Lambda_{N}$, were shown to be uniformly bounded in \cite[Corollary 2]{FK04}. It suffices to show the same for $\mathbb{E}(|\Lambda_{N}'(z_2)|^{8h}|)^{1/2}$. If $0 \leq 8h \leq 1$ then $|\Lambda_{N}'(z_2)|^{8h} \leq |\Lambda_{N}'(z_2)|^{2}+1$ and the latter has bounded expectation by \eqref{1stmom} and \eqref{interiors1}. If $8h >1$, then by Jensen's inequality there exists a positive integer $n>4h$ such that
\begin{equation}
\mathbb{E}(|\Lambda_{N}'(z_2)|^{8h}|)^{1/2} \leq \mathbb{E}(|\Lambda_{N}'(z_2)|^{2n}|)^{2h/n}
\end{equation}
and this is uniformly bounded by Lemma \ref{le:unif-bnd}. This completes the proof for the range $h \geq 0$ and $s \in \mathbb{R}$. For the range $-1 < h < 0$, we write $X_{N}$ as
\begin{equation}
X_{N} = |G_{N}'(z_2)|^{2h}|\Lambda_{N}(z_1)|^{2s}
\end{equation}
and apply H\"older's inequality, yielding
\begin{equation}
\mathbb{E}(X_{N}^{1+\epsilon}) \leq \mathbb{E}(|G_{N}'(z_2)|^{2h(1+\epsilon)q})^{1/q}\mathbb{E}(|\Lambda_{N}(z_1)|^{2s(1+\epsilon)\ell})^{1/\ell}
\end{equation}
where $q = \frac{\ell}{\ell-1}$. Then choosing $\ell$ large enough and $\epsilon>0$ small enough, Theorem \ref{thm:negmomsbnd} shows that $\mathbb{E}(|G_{N}'(z_2)|^{2h(1+\epsilon)q})$ is bounded as required. The second expectation above is bounded for any $s \in \mathbb{R}$, again by \cite[Corollary 2]{FK04}.
\end{proof}

\begin{proof}[Proof of Theorems \ref{th:non-integer-intro} and \ref{th:jointmoms-intro}]
The joint convergence in distribution in Lemma \ref{le:conv-law} and the continuous mapping theorem imply that the quantity $X_{N}$ defined in \eqref{XN} converges in distribution to
\begin{equation}
X = \bigg{|}\frac{\Lambda'(z_2)}{\Lambda(z_2)}\bigg{|}^{2h}|\Lambda(z_1)|^{2s}.
\end{equation}
Then the uniform integrability of Lemma \ref{le:ui} implies that $\lim_{N \to \infty}\mathbb{E}(X_{N}) = \mathbb{E}(X)$ where $\mathbb{E}(X)$ was computed in Lemma \ref{le:non-integer-comp}. The latter gives the final expressions in Theorems \ref{th:non-integer-intro} and \ref{th:jointmoms-intro}, with the positive integer case \eqref{derivmoms2} following from identity \eqref{1f1tolag}.
\end{proof}

\begin{proof}[Proof of Theorem \ref{th:meso}]
We first claim that for any fixed $z$ with $|z|<1$, the following exact identity holds:
\begin{equation}
\label{equal limit}
    \begin{split}
\sum_{\lambda, \mu \in Y_{s}}\frac{f_\lambda f_\mu}{\lambda! \mu!}
&\det\bigg\{(u^{\lambda_i+s-i} \widetilde{K}^{(\lambda_i+s-i)}(u))^{(\mu_j+s-j)}\bigg\}_{i,j=1}^{s}\\
&= \frac{1}{(1-|z|^{2})^{s^{2}+2s}}\,s!\,L_{s}(-|z|^{2}s^{2}),
    \end{split}
\end{equation}
where $\widetilde{K}(u)=\frac{1}{1-u}$ and $u = |z|^{2}$. This follows, since by Theorem \ref{th:exact-intro}, the left-hand side of the above equation is 
$
\lim_{N\rightarrow \infty}\mathbb{E}(|\Lambda'_{N}(z)|^{2s}) 
$ and by Theorem \ref{th:non-integer-intro} the latter coincides with \eqref{derivmoms2}. Next, if $|z|^2=1-N^{-\alpha}$ with $0<\alpha<1$, by Theorem \ref{th:exact-intro},
\beas
\mathbb{E}(|\Lambda'_{N}(z)|^{2s}) 
\sim\sum_{\lambda, \mu \in Y_{s}}\frac{f_\lambda f_\mu}{\lambda! \mu!}
\det\bigg\{(u^{\lambda_i+s-i} \widetilde{K}^{(\lambda_i+s-i)}(u)^{(\mu_j+s-j)}\bigg\}_{i,j=1}^{s},
\eeas 
as $N \to \infty$. Now applying identity \eqref{equal limit}, we obtain
\begin{equation}
    \begin{split}
\mathbb{E}(|\Lambda'_{N}(z)|^{2s}) &\sim  \frac{1}{(1-|z|^{2})^{s^{2}+2s}}s!L_{s}(-|z|^{2}s^{2})\\
&\sim N^{\alpha(s^{2}+2s)}s!L_{s}(-s^{2}).
\end{split}
\end{equation}
\end{proof}
Next we give the proof of Theorem \ref{th:jensen}. We will need the following two lemmas.
\begin{lemma}
\label{lem:moms-linstat}
Consider the linear statistic $X_{N}(f) = \sum_{j=1}^{N}f(e^{\i\theta_{j}})$ where $e^{\i\theta_1},\ldots,e^{\i\theta_{N}}$ are eigenvalues of a Haar distributed unitary matrix $U$. Consider the three test functions for a fixed $0\leq r<1$,
\begin{equation}
g_{1}(e^{\i\theta}) = \frac{e^{-\i\theta}}{1-re^{-\i\theta}}, \qquad g_{2}(e^{\i\theta}) = \frac{e^{-2\i\theta}}{(1-re^{-\i\theta})^{2}}
\end{equation}
and $g_{3}(e^{\i\theta}) = \log(1-re^{-\i\theta})$ with principal branch of the log. Then for $f \in \{g_{1},g_{2},g_{3}\}$, $X_{N}(f)$ converges in distribution to a centered complex normal random variable. Furthermore, all moments of $|X_{N}(f)|$ are bounded, i.e. for any given non-negative integer $m$ there is an absolute constant $C$, independent of $N$, such that
\begin{equation}
\mathbb{E}(|X_{N}(f)|^{2m}) \leq  C \label{momsbnd}.
\end{equation}
\end{lemma}

\begin{proof}
First note that the mean of the corresponding linear statistics are identically zero, $\mathbb{E}(X_{N}(f)) = \frac{1}{2\pi}\int_{0}^{2\pi}f(e^{\i \theta})d\theta = 0$ in each case. The stated Lemma follows from \cite[Lemma 2]{S00} that for any $C^{1}$ function $f$, the cumulants of $X_{N}(f)$ converge to those of a normal random variable as $N \to \infty$. The second cumulant (the variance) in particular converges to a finite constant. In our case, each $f \in \{g_{1},g_{2},g_{3}\}$ and their real and imaginary parts are smooth functions on the unit circle for fixed $r$ with $0\leq r<1$. Hence the moments of $X_{N}(f)$ converge to those of a centered normal random variable with finite variance and are uniformly bounded. Although the results of \cite{S00} are stated for a real valued test function, we can apply it to the real and imaginary parts of $X_{N}(f)$ and deduce \eqref{momsbnd}.
\end{proof}

\begin{lemma}
\label{thm:expect-log-conv}
The following limits hold uniformly for any $r$ belonging to a fixed compact subset of $r \in [0,1)$
\begin{equation}
\lim_{N \to \infty}\mathbb{E}(\log|\Lambda'_{N}(r)|) = \mathbb{E}(\log|\Lambda'(r)|) \label{conv1}
\end{equation}
and
\begin{equation}
\lim_{N \to \infty}\mathbb{E}\left(\mathrm{Re}\left(\frac{\Lambda''_{N}(r)}{\Lambda'_{N}(r)}\right)\right) = \mathbb{E}\left(\mathrm{Re}\left(\frac{\Lambda''(r)}{\Lambda'(r)}\right)\right). \label{conv2}
\end{equation}
\end{lemma}

\begin{proof}
By Lemma \ref{le:conv-law} we have the convergence in distribution of $\Lambda'_{N}(r)$ to $\Lambda'(r)$. So it will suffice to show uniform integrability. For \eqref{conv1}, a trivial bound on the logarithm gives
\begin{equation}
\begin{split}
&|\log|\Lambda'_{N}(r)||^{1+\epsilon} \leq  2|\log|G_{N}'(r)||^{1+\epsilon}+2|G_{N}|^{1+\epsilon}\\
&\leq c_{\epsilon,a}(|G'_{N}(r)|^{2}+|G_{N}'(r)|^{-a})+ 2|G_{N}|^{1+\epsilon}\label{log-bnd}
\end{split}
\end{equation}
for some constants $a, c_{\epsilon,a}>0$ with $0<a<1$. Taking expectations in \eqref{log-bnd}, the first and third terms in \eqref{log-bnd} are bounded by Lemma \ref{lem:moms-linstat} and the middle term by Theorem \ref{thm:negmomsbnd}. Hence for any $\epsilon>0$, there exists a constant $C>0$ independent of $N$ such that
\begin{equation}
\mathbb{E}(|\log|\Lambda'_{N}(r)||^{1+\epsilon}) < C.
\end{equation}
This establishes uniform integrability and hence \eqref{conv1}. 
%

For \eqref{conv2}, convergence in distribution of $r\frac{d}{dr}\log |G_{N}(r)|$ follows similarly as in Lemma \ref{le:conv-law}, see also Lemma \ref{lem:moms-linstat}. For uniform integrability, H\"older's inequality shows that
\begin{equation}
\begin{split}
&\mathbb{E}\left(\bigg{|}\mathrm{Re}\left(\frac{\Lambda_{N}''(r)}{\Lambda_{N}'(r)}\right)\bigg{|}^{1+\epsilon}\right) \leq \mathbb{E}\left(\bigg{|}\frac{\Lambda_{N}''(r)}{\Lambda_{N}'(r)}\bigg{|}^{1+\epsilon}\right) = \mathbb{E}\left(\bigg{|}\frac{G_{N}''(r)}{G_{N}'(r)}+G_{N}'(r)\bigg{|}^{1+\epsilon}\right)\\
&\leq 2(\mathbb{E}(|G_{N}''(r)|^{\ell(1+\epsilon)}))^{1/\ell}(\mathbb{E}(G_{N}'(r)|^{-q(1+\epsilon)}))^{1/q}+2\mathbb{E}(|G_{N}'(r)|^{1+\epsilon})
\end{split}
\end{equation}
where $\Lambda_{N} = e^{G_{N}}$ and $q=\frac{\ell}{\ell-1}$. We bound the first and third terms above using Lemma \ref{lem:moms-linstat} while the middle term is bounded by choosing $\ell$ large enough that $q(1+\epsilon)<2$ and applying Theorem \ref{thm:negmomsbnd}.
\end{proof}

\begin{proof}[Proof of Corollary \ref{th:jensen}]
Starting from the expression in \eqref{jensenformula3intro}, we have
\be
\label{jensenformula3}
\int_{0}^{r}\frac{\mathbb{E}(n_{N}(t))}{t}dt = \mathbb{E}(\log|\Lambda_{N}'(r)|)-\mathbb{E}(\log|\Lambda_{N}'(0)|).
\ee
By Lemma \ref{thm:expect-log-conv}, for any fixed $0 \leq r < 1$
\begin{equation}
\lim_{N \to \infty}\mathbb{E}(\log|\Lambda'_{N}(r)|) = \mathbb{E}(\log|\Lambda'(r)|).
\end{equation}
A direct computation using Theorem \ref{th:non-integer-intro} shows
	\begin{equation}
	\label{derivativeoflimitatzero}
	\begin{split}
	\mathbb{E}(\log|\Lambda'(r)|) - \mathbb{E}(\log|\Lambda'(0)|) &= \frac{d}{ds} \mathbb{E}\left[|\Lambda'(r)|^{s}\right]\Bigg|_{s=0}-\frac{d}{ds} \mathbb{E}\left[|\Lambda'(0)|^{s}\right]\Bigg|_{s=0}\\
	&=-\log(1-r^2).
	\end{split}
	\end{equation}
To show \eqref{nr-intro} we differentiate \eqref{jensenformula3} to obtain
\begin{equation}
\mathbb{E}(n_{N}(r)) = r\frac{d}{dr}\mathbb{E}(\log|\Lambda'_{N}(r)|) = r\mathbb{E}\left(\mathrm{Re}\left(\frac{\Lambda''_{N}(r)}{\Lambda'_{N}(r)}\right)\right). \label{arg-princ}
\end{equation}
By Lemma \ref{thm:expect-log-conv} the expectation \eqref{arg-princ} converges to
\begin{equation}
\mathbb{E}\left(\mathrm{Re}\left(\frac{\Lambda''(r)}{\Lambda'(r)}\right)\right) = \frac{d}{dr}\mathbb{E}(\log|\Lambda'(r)|) = \frac{2r}{1-r^{2}}
\end{equation}
where the last equality follows from \eqref{derivativeoflimitatzero}.
\end{proof}
We remark that formula \eqref{arg-princ}
also follows from applying the argument principle to the holomorphic function $\Lambda'_{N}(r)$ and taking expectations. 

\subsection{Microscopic regime}
\label{se:micropfs}
In this section we complete the proof of Theorem \ref{th:micro-intro} concerning the scaling $|z|^{2}=1-\frac{c}{N}$ with $c \in \mathbb{R}$. We start with the exact result \eqref{exact}
\begin{equation}
\E[|\Lambda_{N}'(z)|^{2s}] = \sum_{\lambda, \mu \in Y_{s}}\frac{f_\lambda f_\mu}{\lambda! \mu!}
\det\bigg\{(u^{\lambda_i+s-i} K_{N}^{(\lambda_i+s-i)}(u))^{(\mu_j+s-j)}\bigg\}_{i,j=1}^{s}, \label{exact2}
\end{equation}
where $u=|z|^{2}$ and $K_{N}(u) = \sum_{j=0}^{N+s-1}u^{j}$.
\begin{proof}[Proof of asymptotics \eqref{micro1}]
In \eqref{exact2}, we represent $K_{N}(u)$ in the form
\begin{equation}
K_{N}(u) = \frac{1-u^{N+s}}{1-u} = (N+s)\int_{0}^{1}(1-x(1-u))^{N+s-1}dx.
\end{equation}
Let $p_{i} = \lambda_{i}+s-i$ and $q_{j} = \mu_{j}+s-j$. Then
\begin{equation}
u^{p_{i}} K_{N}^{(p_i)}(u)= \frac{(N+s)!}{(N+s-p_{i}-1)!}u^{p_{i}}\int_{0}^{1}x^{p_{i}}(1-x(1-u))^{N+s-1-p_{i}}dx. \label{qjderivs}
\end{equation}
When we take the derivative of \eqref{qjderivs} $q_{j}$ times, we can choose to act derivatives on the first term $u^{p_{i}}$ or the second integral term above. The leading contribution will come from applying all the derivatives to the second term as it produces the largest powers of $N$. Setting $u=1-\frac{c}{N}$ and letting $N \to \infty$, by the dominated convergence theorem we have
\begin{equation}
\begin{split}
(u^{p_{i}} K_{N}^{(p_i)}(u))^{(q_j)} &\sim \frac{(N+s)!}{(N+s-p_{i}-q_{j}-1)!}u^{p_{i}}\int_{0}^{1}x^{p_{i}+q_{j}}(1-cx/N)^{N+s-1-p_{i}-q_{j}}dx\\
&\sim N^{p_{i}+q_{j}+1}\int_{0}^{1}x^{p_{i}+q_{j}}e^{-cx}dx. 
\end{split}
\end{equation}
Recall that since $|\lambda|=|\mu|=s$, we have that $\{p_{j}\}_{j=1}^{s}$ and $\{q_{j}\}_{j=1}^{s}$ partition the integer $s(s+1)/2$. The factor $N^{p_{i}+q_{j}+1}$ drops out of the determinant and gives the power
\begin{equation}
N^{\sum_{i=1}^{s}(p_{i}+q_{i}+1)} = N^{s^{2}+2s}.
\end{equation}
Hence we obtain, as $N \to \infty$
\begin{equation}
 \det\bigg\{(u^{p_{i}} K_{N}^{(p_{i})}(u))^{(q_{j})}\bigg\}_{i,j=1}^{s} \sim N^{s^{2}+2s} \det\bigg\{\int_{0}^{1}x^{p_{i}+q_{j}}e^{-cx}dx \bigg\}_{i,j=1}^{s}. \label{hence}
\end{equation}
Inserting \eqref{hence} into \eqref{exact} completes the proof.
\end{proof}
The following two lemmas will be used in the proof of Theorem \ref{th:micro-intro}.

\begin{lemma}[Andr\'eief identity]
\label{le:andre}
For a domain $D$ and two sets of integrable functions $\{f_{j}(x)\}_{j=1}^{s}$ and $\{g_{j}(x)\}_{j=1}^{s}$ we have
\begin{equation}
\det\bigg\{\int_{D}f_{i}(x)\,g_{j}(x)\,dx\bigg\}_{i,j=1}^{s} = \frac{1}{s!}\int_{D^{s}}\,\det\{f_{j}(x_i)\}_{i,j=1}^{s}\,\det\{g_{j}(x_i)\}_{i,j=1}^{s}\,dx_{1}\cdots dx_{s}
\end{equation}
\end{lemma}
\begin{lemma}
\label{le:merge}
Let $s\geq 1$ be an integer and $B(z,w)$ be any sufficiently smooth function of two variables $z$ and $w$. Then we have
\begin{equation}
\frac{\det\bigg\{B(z_i,w_j)\bigg\}_{i,j=1}^{s}}{\Delta(\bm z)\Delta(\bm w)}\bigg{|}_{\bm z =z, \bm w = w} = \left(\prod_{i=0}^{s-1}\frac{1}{(i!)^{2}}\right)\,\det\bigg\{\frac{\partial^{i+j-2}B(z,w)}{\partial z^{i-1}\,\partial w^{j-1}}\bigg\}_{i,j=1}^{s}. \label{mergeakevern}
\end{equation}
\end{lemma}
We now complete the proof of Theorem \ref{th:micro-intro}.
\begin{proof}[Proof of asymptotics \eqref{micro2} and strict positivity]
Let $b_{s}(c)$ denote the leading term in Theorem \ref{th:micro-intro}, formula \eqref{micro1}, that is
\begin{equation}
b_{s}(c) = \sum_{\lambda,\mu \in Y_{s}}\frac{f_\lambda f_\mu}{\lambda! \mu!}
\det\bigg\{ \int_{0}^{1}x^{2s+\lambda_{i}-i+\mu_{j}-j}e^{-cx}dx\bigg\}_{i,j=1}^{s}.
\end{equation}
Using Lemma \ref{le:andre}, we write the above determinant as
\begin{equation}
\begin{split}
&\det\bigg\{\int_{0}^{1}x^{2s+\lambda_{i}+\mu_{j}-i-j}e^{-cx}dx\bigg\}_{i,j=1}^{s} \\
&= \frac{1}{s!}\int_{[0,1]^{s}}\det\bigg\{x_{i}^{\lambda_{j}+s-j}\bigg\}_{i,j=1}^{s}\det\bigg\{x_{i}^{\mu_{j}+s-j}\bigg\}_{i,j=1}^{s}\prod_{j=1}^{s}e^{-cx_{j}}d\bm x,
\end{split}
\end{equation}
where $d\bm x=dx_{1}\cdots dx_{s}$.
The determinants in the integrand above can be recognised in terms of the Schur polynomials, due to their definition
\begin{equation}
s_{\lambda}(\bm x) = \frac{\det\bigg\{x_{i}^{\lambda_{j}+s-j}\bigg\}_{i,j=1}^{s}}{\Delta(\bm x)}.
\end{equation}
We obtain
\begin{equation}
b_{s}(c) = \frac{1}{s!}\int_{[0,1]^{s}}\left(\sum_{\lambda \in Y_{s}}\frac{f_{\lambda}}{\lambda!}s_{\lambda}(\bm x)\right)^{2}\prod_{j=1}^{s}e^{-cx_{j}}\big(\Delta(\bm x)\big)^{2}d\bm x. \label{sumschur}
\end{equation}
By Schur positivity, the integrand in expression \eqref{sumschur} is positive on a set of full measure and therefore we obtain the stated claim about strict positivity. To compute the sum over partitions, we use the known relation between the hook length and the Schur polynomial evaluated at $1_{s} = (1,\ldots,1)$ with $1$ appearing $s$ times, in the form
\begin{equation}
f_{\lambda} = \frac{s_{\lambda}(1_{s})}{\lambda!}\prod_{i=0}^{s}i!. \label{hookform}
\end{equation}
Then the sum in \eqref{sumschur} is
\begin{equation}
\sum_{\lambda \in Y_{s}}\frac{f_{\lambda}}{\lambda!}s_{\lambda}(\bm x) = \left(\prod_{i=0}^{s}i!\right)\frac{(-1)^{s}}{s!}\frac{\partial^{s}}{\partial v^{s}}\sum_{\lambda, l(\lambda)\leq s}\frac{s_{\lambda}(1_{s})}{(\lambda!)^{2}}s_{\lambda}(\bm x)(-v)^{|\lambda|}\bigg{|}_{v=0},
\label{genfn}
\end{equation}
where we introduced a generating function in the variable $v$ to enforce the constraint $|\lambda|=s$. We now replace $s_{\lambda}(1_{s})(-v)^{|\lambda|}$ with $s_{\lambda}(-\bm v)$ where $\bm v$ consists of $s$ new variables. Due to homogeneity of the Schur polynomials, we then recover the desired quantity after taking $\bm v = (v,v,\ldots,v)$ in the end (as before, we denote this operation $\bm v = v$ for convenience). Now the sum over partitions in \eqref{genfn} can be done by standard character expansion techniques. In particular, it follows from \cite[Appendix B]{B00} and the Cauchy-Binet identity that 
\begin{equation}
\begin{split}
\sum_{\lambda, l(\lambda)\leq s}\frac{s_{\lambda}(\bm x)s_{\lambda}(-\bm v)}{(\lambda !)^{2}}=\frac{\det\bigg\{\sum_{\ell=0}^{\infty}\frac{(-x_{i}v_{j})^{\ell}}{(\ell!)^{2}}\bigg\}_{i,j=1}^{s}}{\Delta(\bm x)\Delta(\bm v)}.
\end{split}
\end{equation}
The function inside the above determinant is the series definition of the Bessel function of the first kind
\begin{equation}
J_{0}(2\sqrt{x}) = \sum_{j=0}^{\infty}\frac{(-x)^{j}}{(j!)^{2}}.
\end{equation}
Applying the above procedure to both sums over partitions in \eqref{sumschur}, we obtain 
\begin{equation}
b_{s}(c) = \frac{1}{s!}\left(\prod_{i=0}^{s-1}i!\right)^{2}\frac{\partial^{2s}}{\partial v^{s}\partial w^{s}}Y(v,w)\bigg{|}_{v=w=0}, \label{Yzw}
\end{equation}
where
\bea
Y(v,w) &=& \int_{[0,1]^{s}}\frac{\det\bigg\{J_{0}(2\sqrt{x_{i}v_{j}})\bigg\}_{i,j=1}^{s}\det\bigg\{J_{0}(2\sqrt{x_{i}w_{j}})\bigg\}_{i,j=1}^{s}}{\Delta(\bm v)\Delta(\bm w)}\prod_{j=1}^{s}e^{-cx_{j}}d\bm x\bigg{|}_{\bm v=v, \bm w = w} \nonumber \\
&=&s!\frac{\det\bigg\{\int_{0}^{1}J_{0}(2\sqrt{xv_{i}})J_{0}(2\sqrt{xw_{j}})e^{-cx}dx\bigg\}_{i,j=1}^{s}}{\Delta( \bm v)\Delta( \bm w)}\bigg{|}_{\bm v = v, \bm w = w}\label{Fredholmdeterminant}\\
&=&s!\left(\frac{1}{\prod_{i=0}^{s-1}i!}\right)^{2}\det\bigg\{\frac{\partial^{i+j-2}}{\partial v^{i-1}\partial w^{j-1}}\int_{0}^{1}J_{0}(2\sqrt{xv})J_{0}(2\sqrt{xw})e^{-cx}dx\bigg\}_{i,j=1}^{s},
\nonumber
\eea
and we applied Lemmas \ref{le:andre} and \ref{le:merge}. Inserting the above into \eqref{Yzw} completes the proof.
\end{proof}
We remark that after an appropriate shift and scaling of the parameters $v_{1},\ldots,v_{s}$, $w_{1},\ldots,w_{s}$, the determinant appearing in (\ref{Fredholmdeterminant}) when $c=0$ provides a solution of the integrable partial differential equation studied in \cite{R24}, see \cite[formulae (1.5) and (1.15)]{R24} for more details.

\appendix

\section{Moments of CUE characteristic polynomials}\label{Painlevesix}
Here we review results on moments of the characteristic polynomial in the CUE with an emphasis on connections to Painlev\'e equations. Analogously to \eqref{moms-intro}, we define the moments
\begin{equation}
M_{N}(s,z) = \int_{U(N)}|\Lambda_{N}(z)|^{2s}d\mu. \label{cuemoms}
\end{equation}
The best known result for these moments is perhaps in the case $|z|=1$.
\begin{theorem}[Keating and Snaith \cite{keating2000random}]
For any $s \in \mathbb{C}$ with $\mathrm{Re}(s)>-1/2$, we have
\begin{equation}
M_{N}(s,e^{\i\theta}) = \prod_{j=1}^{N}\frac{\Gamma(j)\Gamma(j+2s)}{(\Gamma(j+s))^{2}}, \qquad \theta \in [0,2\pi). \label{ksprod}
\end{equation}
Furthermore in the limit $N \to \infty$ we have
\begin{equation}
M_{N}(s,e^{\i\theta}) \sim N^{s^{2}}\frac{G^{2}(1+s)}{G(1+2s)}, \label{barnesgasympt}
\end{equation}
where $G(z)$ is the Barnes G-function.
\end{theorem}
We now discuss the form of the moments with $|z| \neq 1$, starting with the integer case $s \in \mathbb{N}$ and then generalizing to $s \in \mathbb{C}$. A particular case of the Selberg integral that will be useful in what follows is
\begin{equation}
\begin{split}
&\int_{[0,1]^{s}}\prod_{j=1}^{s}dt_{j}\,t_{j}^{a}(1-t_{j})^{b}|\Delta(\bm t)|^{2}\\
&= \prod_{j=0}^{s-1}\frac{\Gamma(a+1+j)\Gamma(b+1+j)\Gamma(1+(j+1))}{\Gamma(a+b+2+(s+j-1))}, \label{selberg}
\end{split}
\end{equation}
where $\Delta(\bm{t})$ is the Vandermonde determinant \eqref{vand}. The integral \eqref{selberg} is valid for $\mathrm{Re}(a) > -1$ and $\mathrm{Re}(b)>-1$, see e.g.\, \cite[Chapter 4]{F10book} for an introduction and generalizations. 
\begin{theorem}
For positive integer moments $s \in \mathbb{N}$ and any $z \in \mathbb{C}$, we have
\begin{equation}
M_{N}(s,z) = \frac{1}{c_{s,N}}\int_{[0,1]^{s}}\prod_{j=1}^{s}dt_{j}\,(1-t_{j}(1-|z|^{2}))^{N}\Delta^{2}(\bm{t}) \label{moms-form}
\end{equation}
where the normalization constant is
\begin{equation}
c_{s,N} = \int_{[0,1]^{s}}\prod_{j=1}^{s}dt_{j}\,(1-t_{j})^{N}\Delta^{2}(\bm{t}) = \prod_{j=0}^{s-1}\frac{\Gamma(N+j+1)\Gamma(j+1)\Gamma(j+2)}{\Gamma(N+s+j+1)}. \label{cdefmoms}
\end{equation}
\end{theorem}
\begin{proof}
This is a particular case of moment formulae established for a class of non-Hermitian matrices generalizing the CUE, see \cite[Theorem 1.10]{DS22} and \cite[Theorem 1.1]{SS22} with parameter $M=N$.
\end{proof}
In the previously mentioned case $|z|=1$, formula \eqref{moms-form} reduces to
\begin{equation}
\begin{split}
M_{N}(s,e^{\i\theta}) &= \frac{1}{c_{s,N}}\int_{[0,1]^{s}}\prod_{j=1}^{s}dt_{j}\,\Delta^{2}(\bm{t}) =  \frac{1}{c_{s,N}}\prod_{j=0}^{s-1}\frac{\Gamma(j+1)^{2}\Gamma(j+2)}{\Gamma(s+j+1)}\\
&= \prod_{j=1}^{s}\frac{\Gamma(j)\Gamma(N+s+j)}{\Gamma(s+j)\Gamma(N+j)}, \label{mneval}
\end{split}
\end{equation}
where we used \eqref{selberg} with $a=b=0$ and inserted \eqref{cdefmoms}. This is consistent with \eqref{ksprod} when $s \in \mathbb{N}$. The asymptotics as $N \to \infty$ of expression \eqref{mneval} follows immediately and gives
\begin{equation}
M_{N}(s,e^{\i\theta}) \sim N^{s^{2}}\prod_{j=1}^{s}\frac{\Gamma(j)}{\Gamma(s+j)}, \qquad N \to \infty, \label{uc-asympt}
\end{equation}
which is consistent with \eqref{barnesgasympt} and also appears in \cite{keating2000random}. 

For parameters $a, b$ with $\mathrm{Re}(a)>-1$, $\mathrm{Re}(b)>-1$, consider the joint density on $s$ points $t_{1},\ldots,t_{s} \in [0,1]$ proportional to
\begin{equation}
\prod_{j=1}^{s}t_{j}^{a}(1-t_{j})^{b}\Delta(\bm t)^{2}. \label{jue}
\end{equation}
The density \eqref{jue} is commonly known in random matrix theory as the Jacobi unitary ensemble, see e.g.\ \cite[Chapter 3]{F10book}. Note that the normalization constant for this density is Selberg's integral \eqref{selberg}. Given eigenvalues $t_{1},\ldots,t_{s}$ distributed according to \eqref{jue}, let $t^{(a,b)}_{\mathrm{max}} = \mathrm{max}(t_{1},\ldots,t_{s})$ denote the largest eigenvalue. By the symmetry $M_{N}(s,z) = |z|^{2sN}M_{N}(s,1/z)$, we may assume without loss of generality that $|z| < 1$. We change variables $t_{j} \to t_{j}/(1-|z|^{2})$ for each $j=1,\ldots,s$ in \eqref{moms-form} and obtain
\begin{equation}
M_{N}(s,z) = (1-|z|^{2})^{-s^{2}}\mathbb{P}(t^{(0,N)}_{\mathrm{max}} \leq 1-|z|^{2}). \label{juemax}
\end{equation}
The largest eigenvalue distribution in the Jacobi ensemble has been studied by several authors. For example, applying the main results of \cite{HS99,FW04} to the right-hand side of \eqref{juemax} gives
\begin{equation}
M_{N}(s,z) = (1-|z|^{2})^{-s^{2}}\mathrm{exp}\left(-\int_{1-|z|^{2}}^{1}\frac{\sigma^{(\mathrm{VI})}_{N,s}(t)-c_{1}^{2}t+\frac{c_{1}^{2}+c_{2}^{2}}{2}}{t(1-t)}\,dt\right), \label{pvimn}
\end{equation}
where $\sigma^{(\mathrm{VI})}_{N,s}(t)$ satisfies the Jimbo-Miwa-Okamoto $\sigma$-form of the Painlev\'e VI equation
\begin{equation}
\sigma'(t(1-t)\sigma'')^{2}-\left(\sigma'(2\sigma+(1-2t)\sigma')+c_{1}^{2}c_{2}^{2}\right)^{2}+(\sigma'-c_{1})^{2}(\sigma'-c_{2})^{2} = 0,
\end{equation}
with parameters
\begin{equation}
c_{1} = s+\frac{N}{2}, \qquad c_{2} = \frac{N}{2}.
\end{equation}
Generalizing the asymptotics \eqref{uc-asympt}, we now consider the scaling $|z|^{2}=1-\frac{c}{N}$ for a constant $c \in \mathbb{R}$. As discussed in the introduction, this may be viewed as the microscopic regime.
\begin{theorem}[Microscopic limit of CUE moments]
\label{th:micro-cue-mn}
Let $|z|^{2}=1-c/N$ with $c\in \mathbb{R}$ fixed and $s \in \mathbb{N}$. Then as $N \to \infty$, we have
\begin{equation}
M_{N}(s,z) \sim N^{s^{2}}\prod_{j=1}^{s}\frac{1}{\Gamma(j)\Gamma(j+1)}\int_{[0,1]^{s}}\prod_{j=1}^{s}dt_{j}\,e^{-ct_{j}}\Delta(\bm t)^{2}. \label{micro-moms}
\end{equation}
\end{theorem}
\begin{proof}
Inserting $|z|^{2}=1-\frac{c}{N}$ into formula \eqref{moms-form} and using the limit $(1-ct_{j}/N)^{N} \to e^{-ct_{j}}$ as $N \to \infty$, we obtain the stated result.
\end{proof}
We can also consider the corresponding global and mesoscopic regimes. For the global regime, see also \cite{hughes2001characteristic} and \cite[Corollary 2]{FK04}.
\begin{theorem}[Global and mesoscopic limit of CUE moments]
\label{th:mnglob}
Let $0 \leq |z| < 1$ be fixed and $s \in \mathbb{N}$. Then
\begin{equation}
\lim_{N \to \infty}M_{N}(s,z) = (1-|z|^{2})^{-s^{2}}. \label{global-moms}
\end{equation}
In the mesoscopic regime $|z|^{2} = 1-N^{-\alpha}$ with $0<\alpha<1$, we have
\begin{equation}
M_{N}(s,z) \sim N^{s^{2}\alpha}, \qquad N \to \infty. \label{meso-moms}
\end{equation}
\end{theorem}
\begin{proof}
Making the change of variables $t_{j} \to t_{j}/((1-|z|^{2})N)$ in \eqref{moms-form}, we take the limit $N \to \infty$ using that $(1-t_{j}/N)^{N} \to e^{-t_{j}}$ as $N \to \infty$. This limiting procedure works provided that $(1-|z|^{2})N \to \infty$ as $N \to \infty$ which precisely covers the global and mesoscopic scales. After computing asymptotics of the normalization constant in \eqref{cdefmoms}, the leading coefficient is exactly cancelled by the integral
\begin{equation}
\int_{[0,\infty)^{s}}\prod_{j=1}^{s}dt_{j}\,e^{-t_{j}}\Delta(\bm t)^{2} = \prod_{j=1}^{s}\Gamma(j)\Gamma(j+1).
\end{equation}
\end{proof}
Regarding Theorem \ref{th:micro-cue-mn}, notice that if $c=0$ in \eqref{micro-moms}, application of \eqref{selberg} with $a=b=0$ recovers \eqref{uc-asympt}. For $c > 0$, make the change of variables $t_{j} \to t_{j}/c$ for $j=1,\ldots,s$ in \eqref{micro-moms}. Then expression \eqref{micro-moms} can again be identified as the largest eigenvalue distribution, this time in the Laguerre unitary ensemble, again see \cite[Chapter 3]{F10book} for background. By results of \cite[Section V]{TW94}, see also equation (1.43) in \cite{FW02} with $a=\mu=0$, the latter can be rewritten in terms of Painlev\'e transcendents. Applying these results to \eqref{micro-moms} we obtain
\begin{equation}
M_{N}(s,z) \sim (N/c)^{s^{2}}\mathrm{exp}\left(-\int_{c}^{\infty}\frac{\sigma^{(\mathrm{V})}_{s}(t)}{t}dt\right), \label{micro-pv}
\end{equation}
where $\sigma^{(\mathrm{V})}_{s}(t)$ satisfies the Jimbo-Miwa-Okamoto $\sigma$-form of the Painlev\'e V equation
\begin{equation}
(t\sigma'')^{2}-[\sigma-t\sigma'+2(\sigma')^{2}+2s\sigma']^{2}+(4\sigma')^{2}(s+\sigma')^{2}=0.
\end{equation}
Notice that in the right-hand sides of formulae \eqref{pvimn}, \eqref{global-moms}, \eqref{meso-moms} and \eqref{micro-pv} there is no immediate obstacle to letting $s$ take non-integer values and it is natural to ask if these results continue to hold. The limit \eqref{global-moms} is known for any fixed $s \in \mathbb{C}$ from an application of the strong Szeg\H{o} limit theorem, see \cite[Corollary 2]{FK04}. By again reducing to known relations with Painlev\'e equations, representations \eqref{micro-pv} and \eqref{pvimn} can also be extended to a wider range of $s$. 
\begin{theorem}
The asymptotic formula \eqref{micro-pv} holds for any $s \in \mathbb{C}$ with $\mathrm{Re}(s)>-\frac{1}{2}$.
\end{theorem}
\begin{proof}
By Weyl's integration formula,
\begin{equation}
M_{N}(s,z) =\frac{1}{(2\pi)^{N}N!}\int_{[0,2\pi]^{N}}\prod_{j=1}^{N}d\theta_{j}\,|1-ze^{-\i\theta_{j}}|^{2s}|\Delta(\bm{e^{\i\theta}})|^{2}. \label{weylcik}
\end{equation}
Now use the identity $|1-ze^{-\i\theta}|^{2s} = (e^{\i\theta}-z)^{s}(e^{\i\theta}-\frac{1}{z})^{s}e^{-\i\theta s}e^{-\i\pi s}z^{s}$ where the branch is chosen so that as a function of $w=e^{\i\theta}$, the latter is analytic in the slit plane $\mathbb{C}\setminus([0,z]\cup[1/z,\infty))$. Using standard identities for the CUE we can write \eqref{weylcik} as a Toeplitz determinant
\begin{equation}
M_{N}(s,z) = z^{sN}\det\bigg\{\frac{1}{2\pi}\int_{0}^{2\pi}m(e^{\i\theta},z)e^{\i(k-j)\theta}d\theta\bigg\}_{k,j=0}^{N-1} \label{tdetmn}
\end{equation}
with symbol
\begin{equation}
m(e^{\i\theta},z) = (e^{\i\theta}-z)^{s}\left(e^{\i\theta}-\frac{1}{z}\right)^{s}e^{-\i\theta s}e^{-\i\pi s}.
\end{equation}
Asymptotics of Toeplitz determinants with this symbol were analysed in \cite{CIK11} uniformly in $z$ (in equation (1.7) of \cite{CIK11}, take parameters $\alpha=s$, $\beta=0$ and $V \equiv 0$, $z=e^{t}$). In particular, application of Theorem 1.4 in \cite{CIK11} to the Toeplitz determinant \eqref{tdetmn} combined with identity (1.43) in \cite{FW02} with parameters $a=\mu=0$ gives the asymptotics \eqref{micro-pv} for any $s \in \mathbb{C}$ with $\mathrm{Re}(s)>-\frac{1}{2}$.
\end{proof}
\begin{theorem}
\label{lem:PVI}
Representation \eqref{pvimn} is valid for any $s, z \in \mathbb{C}$ with $\mathrm{Re}(s)>-1$ and $|z|<1$.
\end{theorem}
\begin{proof}
This follows from more general results given in \cite[Theorem 5.3]{DS22} with parameters $\kappa=0$ and $\gamma=2s$, however we give here a proof in the present context of the CUE. Without loss of generality we take $0<z\leq 1$ real. Let $C$ be the unit circle oriented counter clockwise. Starting from \eqref{weylcik}, we use the identity $|1-z/w|^{2s} = (1-z/w)^{s}(1-zw)^{s}$ for any $|w|=1$. We choose the branch cut for the powers so that the function $w \to (1-z/w)^{s}(1-zw)^{s}$ is analytic in the domain $\mathbb{C}\setminus([0,z]\cup[1/z,\infty))$. Then \eqref{weylcik} equals
\begin{equation}
\frac{1}{(2\pi i)^{N}N!}\int_{C^{N}}\prod_{j=1}^{N}\frac{dw_{j}}{w_{j}}\,(1-z/w_{j})^{s}(1-zw_{j})^{s}\prod_{1 \leq i < j \leq N}(w_{j}-w_{i})(w_{j}^{-1}-w_{i}^{-1}). \label{weyl}
\end{equation}
Now we make the change of variables $w_{j} \to -w_{j}z$ for $j=1,\ldots,N$, so that \eqref{weyl} becomes
\begin{equation}
    \begin{split}
\frac{1}{(2\pi i)^{N}N!}\int_{(C/z)^{N}}\prod_{j=1}^{N}\frac{dw_{j}}{w_{j}}(1+1/w_{j})^{s}(1+z^{2}w_{j})^{s}\prod_{1 \leq i < j \leq N}(w_{j}-w_{i})(w_{j}^{-1}-w_{i}^{-1}),
\end{split}
\end{equation}
and the powers now have a cut on $(-\infty,-1/z^{2}]\cup[-1,0]$. We apply Cauchy's theorem to deform the integration contour $C/z$ close to the unit circle $C$ and by taking a limit we can replace $C/z$ with $C$. Here we need $\mathrm{Re}(s)>-1$ to ensure integrability and $\mathrm{Re}(s)>-\frac{1}{2}$ if $z=1$. This gives 
\begin{equation}
M_{N}(s,z) = \frac{1}{(2\pi \i)^{N}N!}\int_{C^{N}}\prod_{j=1}^{N}\frac{dw_{j}}{w_{j}}\,w_{j}^{-s/2}|1+w_{j}|^{s}(1+z^{2}w_{j})^{s}|\Delta(\bm{w})|^{2}, \label{cuepvi}
\end{equation}
where we used the identity $(1+1/w_{j})^{s} = |1+w_{j}|^{s}w_{j}^{-s/2}$. The connection with Painlev\'e VI now follows from comparison with \cite[Proposition 12 and Equation (3.7)]{FW04} with their parameters $\mu=s$, $a=-N-s$, $b=s$. This shows that the integral \eqref{cuepvi} is equal to the right-hand side of \eqref{pvimn} for the stated range of $s$.
\end{proof}

\section{Combinatorial formulae for Theorem \ref{structuretheorem}}
\label{appendix: two combinatorial formulas}

In this appendix, we provide explicit formulas of $d_{m,l,s}(0,0)$ in Proposition \ref{proaboutb00}, and $d_{m,l,s}(h_1,h_2)$ in Proposition \ref{forgeneralh1h2}.
\bea\label{defineofdbarml}
&& d_{m,l,s}(0,0)\nonumber\\
&=& \left( \prod_{i=1}^{s-1} \frac{1}{(i!)^2} \right)(-1)^{m}\sum_{\substack{0\leq i_1<\cdots<i_{s-l}\leq s-1 \\ s \leq j_1 <\cdots < j_l \leq 2s-1 \\
i_{1}+\cdots+i_{s-l}+j_{1}+\cdots+j_{l}=m}}\Big(\prod_{1\leq a<b\leq s-l} (i_b-i_a) \Big) 
\Big(\prod_{1\leq a<b\leq l} (j_b-j_a) \Big)
\nonumber\\
& & \times \, 
\Big(\prod_{a=1}^{s-l} \prod_{b=1}^l (N+j_b-i_a) \Big) \Big(\prod_{a=1}^{s-l} \prod_{b=1}^{l} (N+j_a'-i_b') \Big)\nonumber\\
& & \times \, \Big(\prod_{1\leq a<b\leq s-l} (j_b'-j_a') \Big) 
\Big(\prod_{1\leq a<b\leq s} (i_b'-i_a') \Big),
\eea
where $i_1'<i_2'<\cdots<i_l'=\{0,\ldots,s-1\} \backslash \{i_1,\ldots,i_{s-l}\}$, 
and $j_1'<j_2'<\cdots<j_{s-l}'=\{s,\ldots,2s-1\} \backslash \{j_1,\ldots,j_l\}$.
\be\label{defofdml1new}
d_{m,l,s}(h_{1},h_{2})
=\sum_{\substack{\lambda \in Y_{h_1} , \mu \in Y_{h_2} }}
f_\lambda f_\mu
B_{m,l,s}(\lambda,\mu),
\ee
where $f_\lambda$ is defined in (\ref{hooklengthformula}).
Moreover, for $N\geq 3s^2-s$, $B_{m,l,s}$ involved in (\ref{defofdml1new}) can be expressed as follows. For $\lambda=(\lambda_{1},\ldots,\lambda_{h_{1}}) $ and $\mu=(\mu_{1},\ldots,\mu_{h_{2}}) $, set $\lambda_{h_1+1}=\cdots=\lambda_s=\mu_{h_2+1}=\cdots=\mu_s=0$.
Set $n_i=\lambda_i+s-i$ and $m_j=\mu_j+s-j$ for $i,j=1,\ldots,s$. 
Then
\bea\label{defofnewaml2}
B_{m,l,s}&=& \left( \prod_{i=1}^{s-1} \frac{1}{(i!)^2} \right)\sum_{
\substack{
k_1,\ldots,k_{s-l},u_1',\ldots,u_{s-l}'\in\{0,1,\ldots,s-1\} \\
k_1',\ldots,k_{l}',u_1,\ldots,u_l\in\{s,\ldots,2s-1\} \\
\sum_{q=1}^{s-l}(k_{q}+u_{q}')+\sum_{q=1}^{l}(k_{q}'+u_{q})=2m}}\sum_{\substack{0\leq i_1<\cdots<i_{s-l}\leq s-1 \\ s\leq j_1<\cdots<j_l\leq 2s-1} }  (-1)^{i_1+\cdots+i_{s-l}+j_1+\cdots+j_l}\nonumber\\
&&  \times \,
\prod_{q=1}^{s-l} \binom{k_{q}}{n_{i_{q}+1}}\binom{u_{q}'}{m_{j_{q}'+1-p}}
\prod_{q=1}^{l} \binom{N+k_{q}'}{n_{i_{q}'+1}}\binom{N+u_{q}}{m_{j_{q}+1-s}}\nonumber\\
&&\times \, \prod_{1\leq i <j\leq s-l} (k_j-k_i) (u_i'-u_j') 
\prod_{1\leq i<j\leq l} (k_j'-k_i') (u_i-u_j)\nonumber\\
&&  \times \,\prod_{1\leq j \leq l, 1\leq i\leq s-l} (2s-1-u_j-k_i)
\prod_{1\leq j \leq s-l, 1\leq i\leq l} (2s-1-u_j'-k_i'),
\eea
where $i_1'<i_2'<\cdots<i_l'=\{0,\ldots,s-1\} \backslash \{i_1,\ldots,i_{s-l}\}$, 
and $j_1'<j_2'<\cdots<j_{s-l}'=\{s,\ldots,2s-1\} \backslash \{j_1,\ldots,j_l\}$.

\bibliography{main}
\bibliographystyle{plain}
\end{document}